\documentclass[11pt, letterpaper]{amsart}
\usepackage[left=1in,right=1in,bottom=0.5in,top=0.8in]{geometry}
\usepackage{amsfonts}
\usepackage{amsmath, amssymb}
\usepackage{graphicx}
\usepackage[font=small,labelfont=bf]{caption}
\usepackage{epstopdf}
\usepackage[pdfpagelabels,hyperindex]{hyperref}
\usepackage{xcolor}
\usepackage{amsthm}
\usepackage{float}
\usepackage{pgfplots}
\usepackage{listings}
\usepackage{longtable}
\usepackage{mathrsfs}

\newcommand{\overbar}[1]{\mkern 1.5mu\overline{\mkern-1.5mu#1\mkern-1.5mu}\mkern 1.5mu}

\newtheorem{thm}{Theorem}[section]
\newtheorem{cor}[thm]{Corollary}
\newtheorem{prop}[thm]{Proposition}
\newtheorem{lem}[thm]{Lemma}

\theoremstyle{definition}
\newtheorem{defn}[thm]{Definition}

\theoremstyle{remark}
\newtheorem{rem}[thm]{Remark}

\pgfplotsset{compat=1.9}

\newcommand{\RN}[1]{%
  \textup{\uppercase\expandafter{\romannumeral#1}}%
}
\newcounter{x}

\newcommand{\mres}{\mathbin{\vrule height 1.6ex depth 0pt width
0.13ex\vrule height 0.13ex depth 0pt width 1.3ex}}

\numberwithin{equation}{section}

\author[Junfu Yao]{Junfu Yao}
\email{jyao21@jhu.edu}

\title{A Mountain-pass Theorem in Hyperbolic Space and Its Application}
\begin{document}
\begin{abstract}
We develop a min-max theory for certain complete minimal hypersurfaces in hyperbolic space. In particular, we show that given two strictly stable minimal hypersurfaces that are both asymptotic to the same ideal boundary, there is a new one trapped between the two. As an application, we show that under a low entropy condition, all the minimal hypersurfaces asymptotic to the same ideal boundary are isotopic.
\end{abstract}
\maketitle
\section{Introduction}
Let $n\geq2$ and $\mathbb{H}^{n+1}$ be the ($n+1$)-dimensional hyperbolic space. We recall the standard compatification $\overbar{\mathbb{H}}^{n+1}=\mathbb{H}^{n+1}\cup(\mathbb{R}^{n}\times\{0\})\cup \{\infty\}$ determined by the upper half space model. In this model, $\partial_\infty\mathbb{H}^{n+1}=\overbar{\mathbb{H}}^{n+1}\backslash\mathbb{H}^{n+1}$ is called the ideal boundary. It is readily checked that $\overbar{\mathbb{H}}^{n+1}$ and $\partial_\infty\mathbb{H}^{n+1}$ are associated to a unique conformal structure.

A hypersurface in $\mathbb{H}^{n+1}$, i.e., a properly embedded codimension-one submanifold, $\Sigma\subset\mathbb{H}^{n+1}$, is called minimal if it is a critical point to the area functional
\begin{equation*}
    \text{Vol}(\Sigma)=\int_\Sigma1\ d\mathcal{H}^n_\mathbb{H},
\end{equation*}
where $\mathcal{H}^n_\mathbb{H}$ is the $n$-dimensional Hausdorff measure in hyperbolic space.

There are no closed minimal hypersurfaces in $\mathbb{H}^{n+1}$. So, it is natural to consider minimal hypersurfaces that are asymptotic to a submanifold of the ideal boundary (See Section \ref{94} for terminologies). The existence of area-minimizing minimal hypersurfaces with prescribed ideal boundary was proved by Anderson \cite{Anderson1982}. Minimal surfaces in $\mathbb{H}^3$ with prescribed topological types were treated in \cite{Anderson1983}, \cite{Gabai97}, \cite{deOliveiraSoret98}, \cite{MartinWhite2013} and \cite{AlexakisMazzeo}. Some uniqueness and non-uniqueness results were studied in \cite{Coskunuzer2006Generic}, \cite{Coskunuzer2011Number} and \cite{HuangWang2015Counting}. We also refer to the survey paper \cite{CoskunuzerSurvey} on this topic.

We say a minimal hypersurface, $\Sigma$, is stable if for any compactly supported function $g\in C^\infty_c(\Sigma)$ and the corresponding normal variation, $\{\Sigma_s\}$, of $\Sigma$ determined by $g$, the second variation of the area functional,
\begin{equation*}
    \frac{d^2}{dt^2}\Big|_{s=0}\text{Vol}(\Sigma_s)=\int_\Sigma |\nabla_\Sigma g|^2+(n-|A_\Sigma|^2)g^2 d\mathcal{H}^n_\mathbb{H},
\end{equation*}
is non-negative. Here, $A_\Sigma$ is the second fundamental form of $\Sigma$. A stable minimal hypersurface is called strictly stable if the second variation is positive unless $g\equiv0$.

In the first part of the present paper, we prove a mountain pass theorem in hyperbolic space. Specifically, we prove that given two strictly stable minimal hypersurfaces that have the same asymptotic boundary and bound a domain, there is a third minimal hypersurface between the two. See Section 2 for the terminologies.

\begin{thm}\label{main1}
Let $n\geq2$, $\alpha\in(0,1)$ and $k\geq4$ and let $M$ be a $C^{k,\alpha}$ $(n-1)$-dimensional closed submanifold of $\partial_\infty\mathbb{H}^{n+1}$. Suppose $\Gamma_-$ and $\Gamma_+$ are two distinct strictly stable minimal hypersurfaces in $\mathbb{H}^{n+1}$ that are both $C^{k,\alpha}$-asymptotic to $M$ and $\Gamma_-\preceq\Gamma_+$. Then, there is a minimal hypersurface $\Gamma_0\neq\Gamma_\pm$ that has codimension-$7$ singular set and is $C^{k,\alpha}$-asymptotic to $M$.
\end{thm}

\begin{rem}
Since we need the knowledge of uniformly degenerate operators from \cite{AlexakisMazzeo} (see also \cite{MazzeoEdgeOperator}), $k\geq4$ is to ensure the operator that we shall consider to be at least $C^{2,\alpha}$, so that the theory there can apply.
\end{rem}

To prove Theorem \ref{main1}, we develop a min-max theory in hyperbolic spaces. Our approach is based on De Lellis-Tasnady's \cite{deLellisTasnady} (see also \cite{ColdingdeLellisMinmax} and \cite{DeLellisRamic}) reformulation of the works of Almgren-Pitts \cite{PittsExiReg} and Simon-Smith \cite{SmithExistence2Sphere}. However, as the volume of any hypersurface with nonempty asymptotic boundary is infinite, we follow \cite{BWMountainPass} and work on a certain relative functional. Another point in Theorem \ref{main1} is the asymptotic regularity of the minimal hypersurface produced by the min-max procedure. The asymptotic regularity of area-minimizing currents has been studied in \cite{HardtLinHyperbolicRegularity}, \cite{LinHyperbolicMinimalGraph}, \cite{LinAMCinHighCodim} and \cite{TonegawaCMCinHpSpace}. We follow the idea of Hardt-Lin \cite{HardtLinHyperbolicRegularity} to establish the regularity result in our setting.

\iffalse
\cite{AgolMarquesNeves}, \cite{CarlottoDeLellisMinmax}, \cite{ChambersLiok}, \cite{ChodoshMantoulidisMS}, \cite{GasparGuaraco}, \cite{GuangLiWangZhouMinmax}, \cite{IrieMarquesNeves}, \cite{KetoverYevgenySong}, \cite{KetoverZhou}, \cite{LiWangDim8}, \cite{LiWangGenericMetDim8}, \cite{MarquesNevesWillmore}, \cite{MarquesNevesMorseInd}, \cite{MarquesNevesInfMHS}, \cite{MarquesNevesSong}, \cite{MontezumaMinmax}, \cite{MontezumaMountainpass}, \cite{RivireViscosityMinmax}, \cite{Song2018ExistenceMHS}, \cite{WangExistenceFreeBMHS}, \cite{ZhouMultipOne}, \cite{ZhouZhuCMCH}
\fi

\iffalse
Min-max theory for minimal hypersurfaces has widely studied. For further reading on this topic, we refer to \cite{AgolMarquesNeves}, \cite{CarlottoDeLellisMinmax}, \cite{ChambersLiok}, \cite{ChodoshMantoulidisMS}, \cite{GasparGuaraco}, \cite{GuangLiWangZhouMinmax}, \cite{IrieMarquesNeves}, \cite{KetoverYevgenySong}, \cite{KetoverZhou}, \cite{LiWangDim8}, \cite{LiWangGenericMetDim8}, \cite{MarquesNevesWillmore}, \cite{MarquesNevesMorseInd}, \cite{MarquesNevesInfMHS}, \cite{MarquesNevesSong}, \cite{MontezumaMinmax}, \cite{MontezumaMountainpass}, \cite{RivireViscosityMinmax}, \cite{Song2018ExistenceMHS}, \cite{WangExistenceFreeBMHS}, \cite{ZhouMultipOne}, \cite{ZhouZhuCMCH} and references therein.
\fi

The second part is devoted to give a proof to the conjecture proposed by Bernstein \cite[Conjecture 1.6]{BernsteinEntropy}. In their paper, the author introduced the following notion of \textit{hyperbolic entropy} for any hypersurface $\Sigma$ in $\mathbb{H}^{n+1}$:
\begin{equation*}
    \lambda_{\mathbb{H}}[\Sigma]=\sup_{p_0\in\mathbb{H}^{n+1},\ \tau>0}\int_{\Sigma}\Phi_n^{0,p_0}(-\tau,p)d\mathcal{H}^n_{\mathbb{H}}(p),
\end{equation*}
where $\Phi_n^{0,p_0}(-\tau,p)=K_n(\tau,\text{dist}_{\mathbb{H}^{n+1}}(p,p_0))$, and $K_n(t,\rho)$ is a positive function on $\mathbb{R}^1_+\times\mathbb{R}^1_+$ so that the function $H_n(q,t;q_0,t_0)=K_n(t-t_0,\text{dist}_{\mathbb{H}^{n+1}}(q,q_0))$ is the heat kernel of $\mathbb{H}^n$ with singularity at $q=q_0$ and at time $t=t_0$. That is, for any $q_0\in\mathbb{H}^{n}$, we have
\begin{equation}\label{74}\left\{
    \begin{aligned}
        &\left(\frac{\partial}{\partial t}-\Delta_{\mathbb{H}^n}\right)H_n(q,t;q_0,t_0)=0,\ \text{for }t>t_0\\
        &\lim_{t\to t_0^+}H_n(q,t;q_0,t_0)=\delta_{q_0}.
    \end{aligned}\right.
\end{equation}

In \cite{LiYauConformalInv}, the \textit{conformal volume} of an ($n-1$)-dimensional submanifold $M\subset\partial_\infty\mathbb{H}^{n+1}$ is defined by
\begin{equation*}
    \lambda_c[M]=\sup_{\psi\in\text{Mob}(\partial_\infty\mathbb{H}^{n+1})}\text{vol}(\psi(M)),
\end{equation*}
where $\text{Mob}(\partial_\infty\mathbb{H}^{n+1})$ is the group of M\"{o}bius transformations on $\partial_\infty\mathbb{H}^{n+1}$. From the knowledge of \cite{BryantConformal}, we see that $\lambda_c[M]<\infty$ for any ($n-1$)-dimensional $C^2$ submanifold $M\subset\partial_\infty\mathbb{H}^{n+1}$. Bernstein \cite{BernsteinIsoper} (see also \cite{NguyenWeightedMon}) proved an isoperimetric inequality in dimension $3$ relating the conformal volume and the renormalized area studied in \cite{AlexakisMazzeo} (see also \cite{GrahamWitten}).

For a complete minimal hypersurface $\Sigma$ in $\mathbb{H}^{n+1}$ with $C^1$-asymptotic boundary $M$, Bernstein \cite{BernsteinEntropy} proved that $\lambda_c[M]=\text{vol}(\mathbb{S}^{n-1})\lambda_{\mathbb{H}}[\Sigma]$. They also conjectured that if $\gamma\subset\partial_\infty\mathbb{H}^{3}$ is a closed curve with $\lambda_c[\gamma]\leq2\pi\lambda[\mathbb{S}^1\times\mathbb{R}^1]$\footnote{In \cite{BernsteinEntropy}, the assumption on $\gamma\subset\partial_\infty\mathbb{H}^{3}$ is $\lambda_c[\gamma]\leq\lambda[\mathbb{S}^1\times\mathbb{R}^1]$, but apparently, the author should have meant what we state here.}, where $\lambda[\cdot]$ is the Colding-Minicozzi entropy defined in \cite{CMGeneric}, then all the minimal surfaces with $C^1$-asymptotic boundary $\gamma$ are isotopic to each other. We prove that when $2\leq n\leq6$, this conjecture holds in the following sense:

\begin{thm}\label{main2}
Let $2\leq n\leq6$ and $n\neq3$. Let $\alpha\in(0,1)$ and $k\geq4$. Fix an $(n-1)$-dimensional $C^{k+1,\alpha}$ closed submanifold $M$ in $\partial_\infty\mathbb{H}^{n+1}$. Assume that $\lambda_c[M]<\mathrm{Vol}(\mathbb{S}^{n-1})\lambda[\mathbb{S}^{n-1}\times\mathbb{R}^1]$. Then, all the minimal hypersurfaces with $C^{k+1,\alpha}$-asymptotic boundary $M$ are $C^{m,\alpha}$ isotopic to each other, where $m=\min\{k,n\}$.
\end{thm}

\begin{rem}
The reason we require $n\neq3$ is that, as pointed out in \cite{TonegawaCMCinHpSpace}, when $n=3$, the regularity we require ($k\geq4$) is greater than what can be ensured for a general ideal boundary, so the result may not hold in many cases of interest.
\end{rem}

\begin{rem}
The reason we can not have $C^{k,\alpha}$ isotopies when $k>n$ is that the functions considered in the theory of uniformly degenerate operators introduced in \cite{AlexakisMazzeo} do not have good decay for high order derivatives when approaching the ideal boundary, and this prevents the isotopy from being as smooth as its initial data.
\end{rem}

Roughly speaking, we say $\Gamma_1$ is $C^{m,\alpha}$ isotopic to $\Gamma_2$ if $\Gamma_1$ can be continuously deformed into $\Gamma_2$ in the $C^{m,\alpha}$ topology of hypersurfaces in $\overbar{\mathbb{H}}^{n+1}$. See Section \ref{94} for the precise definition.

The proof of Theorem \ref{main2} is inspired by the work of Bernstein-Wang \cite{BWTopologicalUniqueness} (see also \cite{BWClosedHLowEnt} and \cite{BCW}) on constructing isotopies between self-expanders, which builds on a careful study of self-expanders in their previous works: \cite{BWSpace}, \cite{BWSmoothCompactness}, \cite{BWIntegerDegree}, \cite{BWRelativeEntropy}, \cite{BWMountainPass}. We establish the corresponding results that will be used in the setting of minimal hypersurfaces in hyperbolic space. An important observation in \cite{BWTopologicalUniqueness} is that, under the low entropy condition, if we perturb an unstable minimal hypersurface by the first eigenfunction of the stability operator, it can be deformed to a stable minimal hypersurface.

Combining Theorem \ref{main2} with the result of the existence of minimal disks \cite[Theorem 4.1]{Anderson1983} (see also \cite{Gabai97}), we have

\begin{cor}
Let $\gamma\subset\partial_\infty\mathbb{H}^3$ be a disjoint union of smooth curves. If $\lambda_c[\gamma]<2\pi\lambda[\mathbb{S}^1]$, then any minimal surface $\Sigma$ with $\partial_\infty\Sigma=\gamma$ is a collection of disks.
\end{cor}

\begin{proof}
Let $\gamma_1$, $\gamma_2$, ..., $\gamma_m$ be the connected components of $\gamma$. We solve the asymptotic Plateau problem independently for each $\gamma_j$, $1\leq j\leq m$. By \cite[Theorem 4.1]{Anderson1983}, we get a collection of minimal disks $\{\Sigma_j\}_{j=1}^m$ with $\Sigma_j$ asymptotic to $\gamma_j$. We claim that $\Sigma_i$ and $\Sigma_j$ are disjoint if $i\neq j$. Otherwise, at any point of the intersection $\Sigma_i\cap\Sigma_j$, the density is at least $2$. On the other hand, if we let $\Sigma=\bigcup_{j=1}^m\Sigma_j$, then by \cite[Theorem 1.1]{BernsteinEntropy} $\lambda_\mathbb{H}[\Sigma]=\frac{1}{2\pi}\lambda_c[\gamma]<\lambda[\mathbb{S}^1]<2$. This implies the density at any point of $\Sigma$ is less than $2$, which leads to a contradiction. Hence, $\Sigma_j\cap\Sigma_i=\emptyset$ when $i\neq j$, and $\Sigma$ is thus a smooth surface. By \ref{main2}, any minimal surface $\Sigma'$ with $\partial_\infty\Sigma'=\gamma$ is isotopic to $\Sigma$. This implies $\Sigma'$ is also a collection of topological disks.
\end{proof}

The significance of this result is that under the low entropy condition, one can solve the corresponding asymptotic Plateau problem for disks.

The renormalized area of a surface $\Sigma$, $\mathcal{A}(\Sigma)$, is introduced in \cite{GrahamWitten}. For any boundary defining function $r$, if we write $\mathcal{H}^2_\mathbb{H}(\Sigma\cap\{r\geq\epsilon\})$ as an expansion in $\epsilon$ as $\epsilon\to0$, then the renormalized area is the constant term in the expansion. Another corollary of Theorem \ref{main2} is the following result which is conjectured in \cite{BernsteinIsoper}:

\begin{cor}\label{99}
Suppose $\Sigma$ and $\Sigma'$ are two minimal surfaces in $\mathbb{H}^3$ with smooth asymptotic boundaries. If $\partial_\infty\Sigma=\partial_\infty\Sigma'$ and $\Sigma'$ is not a disk, then
\begin{equation*}
    -3\pi>-\frac{(2\pi)^{3/2}}{\sqrt{e}}=-2\pi\lambda[\mathbb{S}^1]>\mathcal{A}(\Sigma).
\end{equation*}
\end{cor}

In \cite{AlexakisMazzeo}, Alexakis-Mazzeo showed that the renormalized area, $\mathcal{A}(\Sigma)$, is equal to negative one half of the Willmore energy of the double of $\Sigma$ (defined in a proper sense). From the resolution of the Willmore conjecture by Marques-Neves \cite{MarquesNevesWillmore}, we see that $-2\pi\lambda[\mathbb{S}^1]>-\pi^2\geq\mathcal{A}(\Sigma')$. So, as pointed out in \cite{BernsteinIsoper}, the significance of Corollary \ref{99} is that one can use the boundary geometry to bound the renormalized area of a minimal surface of unspecified topological type.

The paper is organized as follows: In Section 2, we fix the notations for the remainder of the paper and recall some results that will be used later on. In Section 3, we discuss the function space $\mathfrak{Y}$, and introduce a general notion of relative entropy for any element in $\mathfrak{Y}^*_{\mathcal{C}}(\Lambda)$, which consists of the limit of $\partial^*U$ with relative entropy bounded by $\Lambda$ in $\mathfrak{Y}^*$. In Section 4, we discuss a class of non-compactly supported vector fields. In Section 5, we discuss the notion of minimizing to the first order with respect to the flows generated by the vector fields introduced in Section 4. In Section 6, we develop a min-max theory in hyperbolic space and prove Theorem \ref{main1}. In Section 7, we discuss some useful properties for hypersurfaces in $\mathbb{H}^{n+1}$ that will be used in the proof of Theorem \ref{main2}. In Section 8, we show that under the low entropy condition, any unstable minimal hypersurface is isotopic to a stable one. In Section 9, we first apply Theorem \ref{main1} to show that Theorem \ref{main2} holds for a generic class of submanifold $M$, and complete of the proof for a general $M$ by perturbing $M$ to a generic submanifold $M'$ and constructing isotopy between minimal hypersurfaces asymptotic to $M$ and minimal hypersurfaces asymptotic to $M'$. Section 8 and 9 are carried out in a similar scheme as in \cite{BWTopologicalUniqueness}.

\subsection*{Acknowledgement}The author would like to thank his advisor, Jacob Bernstein, for suggesting the problem and for constant encouragement and numerous helpful advice during various stages of the work. The author is grateful to Letian Chen for providing many useful comments and reading a draft version of the paper. He also thanks Chutian Ma for stimulating discussions.

\section{Background and Notation}
In this section, we fix the notations for the remainder of the paper and recall some background knowledge.

\subsection{Basic notions}\label{94}
We will always consider the Poincar\'e half-space model for $\mathbb{H}^{n+1}$. That is, the open upper half-space
\begin{equation*}
    \mathbb{R}^{n+1}_+=\{(\textbf{x},y)\in\mathbb{R}^n\times\mathbb{R}:\ y>0\},
\end{equation*}
equipped with the metric
\begin{equation*}
    g_P=\frac{1}{y^2}(d\textbf{x}\otimes d\textbf{x}+dy\otimes dy)=\frac{1}{y^2}g_{\mathbb{R}^{n+1}}.
\end{equation*}
As mentioned in Section 1, $\partial_\infty\mathbb{H}^{n+1}$ is identified with $\mathbb{R}^n\times\{0\}\cup\{\infty\}$. For any point $\textbf{p}_0\in\mathbb{H}^{n+1}$, there is an isometry $i:\ \mathbb{H}^{n+1}\to\mathbb{R}^{n+1}_+$ satisfying $i(\textbf{p}_0)=(\textbf{0},1)$ and $i^*g_P=g_{\mathbb{H}^{n+1}}$. Such isometries correspond to the M\"obius transformations of $\mathbb{S}^n=\partial_\infty\mathbb{H}^{n+1}$. Let $M$ be as stated in Theorem \ref{main1}. By a suitable transformation, we will always assume that $M$ is contained in $\mathbb{R}^n\times\{0\}\subset\partial_\infty\mathbb{H}^{n+1}$.

A boundary defining function $r$ on $\overbar{\mathbb{H}}^{n+1}$ is a function such that the metric $r^2g_{\mathbb{H}^{n+1}}$ extends smoothly to $\overbar{\mathbb{H}}^{n+1}$. Fix an arbitrary boundary defining function $r$ over $\overbar{\mathbb{H}}^{n+1}$. Let $M$ be an $(n-1)$-dimensional embedded (sufficiently smooth) submanifold in $\partial_\infty\mathbb{H}^{n+1}$. A complete embedded hypersurface $\Sigma\subset\mathbb{H}^{n+1}$ is called $C^{k,\alpha}$-$asymptotic$ to $M$, if $\Sigma\cup M$ (denoted by $\overbar{\Sigma}$) is a $C^{k,\alpha}$ hypersurface in $\overbar{\mathbb{H}}^{n+1}$ equipped with the metric $r^2g_P$. $M$ is called the ideal boundary of $\Sigma$ and is denoted by $\partial_\infty \Sigma$. Note that the $C^{k,\alpha}$-asymptotic boundary, $M$, is independent of the choices of boundary defining functions.

Fix a normal vector to $M$ in $\partial_\infty\mathbb{H}^{n+1}$. For any $C^1$-asymptotic hypersurface $\Gamma$ with $\partial_\infty\Gamma=M$, we fix a normal vector of $\Gamma$ to be compatible with that of $M$. Let $U_\Gamma$ be the open set in $\mathbb{H}^{n+1}$ with $\partial U_\Gamma=\Gamma$, whose outward normal equals that of $\Gamma$. For two hypersurfaces $\Gamma_1$ and $\Gamma_2$ with $\partial_\infty\Gamma_1=\partial_\infty\Gamma_2=M$, we shall write $\Gamma_1\preceq\Gamma_2$ if $U_{\Gamma_1}\subset U_{\Gamma_2}$.

Let $k\geq1$, $\alpha\in(0,1)$ and let $\Gamma_1$, $\Gamma_2$ be two hypersurfaces in $\mathbb{H}^{n+1}$ that have $C^{k,\alpha}$-asymptotic boundaries in $\partial_\infty\mathbb{H}^{n+1}$. Then, $\Gamma_1$ and $\Gamma_2$ are called $C^{k,\alpha}$ isotopic if there is an $\textbf{F}$ mapping from $[0,1]$ to the set of hypersurfaces with $C^{k,\alpha}$-asymptotic boundaries, so that as $t\to t_0$, $\textbf{F}(t)$ converges to $\textbf{F}(t_0)$ as hypersurfaces in the $C^{k,\alpha}$ topology of $(\overbar{\mathbb{H}}^{n+1},r^2g_{\mathbb{H}^{n+1}})$, where $r$ is an arbitrary boundary defining function. It can be checked that the notion of isotopy does not rely on the choice of boundary defining functions.

\subsection{Conventions}
We follow the convention in \cite{YaoRel} that the notations with a bar on the top are understood in ($\overbar{\mathbb{R}^{n+1}_+},g_{\mathbb{R}^{n+1}}$), and those without a bar are understood in the topology of $\mathbb{H}^{n+1}$. For example, for a hypersurface $\Sigma$, $g_\Sigma$ (resp. $\overbar{g_\Sigma}$) is the Riemannian metric induced from $g_P$ (resp. $g_{\mathbb{R}^{n+1}}$); $\text{div}_\Sigma$ (resp. $\overbar{\text{div}}_\Sigma$) is referred to the divergence taken with respect to $g_\Sigma$ (resp. $\overbar{g_\Sigma}$). To avoid confusions, $\overbar{S}$ is referred to the topological closure of a set $S$ in $\overbar{\mathbb{R}^{n+1}_+}$, and the closure taken in $\mathbb{H}^{n+1}$ is denoted by $\text{cl}(S)$. Also, we denote by $\mathcal{H}^n$ the $n$-dimensional Hausdorff measure in Euclidean space, and let $\mathcal{H}^n_\mathbb{H}$ be the Hausdorff measure in $\mathbb{H}^{n+1}$.

\subsection{Relative entropy in hyperbolic space}
Let $\Gamma_-\preceq\Gamma_+$ be two strictly stable minimal hypersurfaces in $\mathbb{H}^{n+1}$ that are both $C^{k,\alpha}$-asymptotic to $M$. A bounded open set $\Omega'\subset\mathbb{R}^{n+1}_+$ is called a \textit{thin neighborhood of $\Gamma_-$} if there are constants $C,\epsilon>0$, so that
\begin{align*}
    \Omega'\cap\{y=r\}\subset\overbar{\mathcal{N}}_{C\cdot r^{n+1}}(\Sigma)
\end{align*}
holds for $r\in(0,\epsilon)$. Here, $\overbar{\mathcal{N}}_\delta(\Sigma)$ is the $\delta$-tubular neighborhood taken in the Euclidean topology. In \cite{YaoRel}, we proved that $\Gamma_+$ is a thin neighborhood of $\Gamma_-$ if $M$ is $C^3$. So, we can choose a thin neighborhood, $\Omega'$, of $\Gamma_-$ which contains $\Gamma_+$. We can further require that $\Omega'$ is enclosed by two hypersurfaces $\Gamma'_-$ and $\Gamma'_+$ (which will be chosen in Proposition \ref{20}) with $\Gamma'_-\preceq\Gamma_-\preceq\Gamma_+\preceq\Gamma'_+$.

Define $\mathcal{C}(\Gamma'_-,\Gamma'_+)$ to be the set of all Caccioppoli sets\footnote{See \cite[Section 1]{GiustiBook} for the terminology.} $U$ with $U_{\Gamma'_-}\subset U\subset U_{\Gamma'_+}$. For any function $\psi=\psi(p,\textbf{v})$ over $\text{cl}(\Omega')\times\mathbb{S}^n$, let
\begin{equation}\label{110}
    E_{rel}[\partial^*U,\Gamma_-;\psi]:=\lim_{\epsilon\to0}\left(\int_{\partial^*U\cap\{y\geq\epsilon\}}\psi(p,\overbar{\textbf{n}}_{\partial^*U}(p)) d\mathcal{H}^n_{\mathbb{H}^{n+1}}-\int_{\Gamma_-\cap\{y\geq\epsilon\}}\psi(p,\overbar{\textbf{n}}_{\Gamma_-}(p)) d\mathcal{H}^n_{\mathbb{H}^{n+1}}\right).
\end{equation}
whenever the limit exists. Since the thin neighborhood $\Omega'$ is bounded in $\mathbb{R}^{n+1}_+$, we can pick a boundary defining function $r_1$ satisfying $r_1(\textbf{x},y)=y$ in a compact subset of $\overbar{\mathbb{R}^{n+1}_+}$ containing $\Omega'$. So, when restricted to $\Omega'$, $y\to0$ simply means approaching the ideal boundary. Observe that the limit (\ref{110}) always exists when Spt($\psi$)$\subset\subset\{y>0\}\times\mathbb{S}^n$.

Let $\mathfrak{X}$ be the set of locally Lipschitz functions on $\text{cl}(\Omega')\times\mathbb{S}^n$ satisfying $\psi(p,\textbf{v})=\psi(p,-\textbf{v})$ and
\begin{equation}
    \|\psi\|_\infty+\|\nabla_{\mathbb{S}^n}\psi\|_\infty+\|\nabla_{\mathbb{S}^n}\nabla_{\mathbb{S}^n}\psi\|_\infty+||y(p)\nabla_{\mathbb{R}^{n+1}_+}\psi\|_\infty+\|y(p)\nabla_{\mathbb{R}^{n+1}_+}\nabla_{\mathbb{S}^n}\psi\|_\infty<\infty,\label{weighted norm}
\end{equation}
It is easy to check that $\mathfrak{X}$ equipped with the norm (\ref{weighted norm}) is a Banach space and is an algebra.

Define a family of smooth cut-off functions $\{\phi_{y_1,\delta}\}_{y_1,\delta}$ such that $\phi_{y_1,\delta}$ only depends on $y=y(p)$. We require that $\phi_{y_1,\delta}$ equals $1$ in $\{y\geq y_1\}$ and vanishes in $\{y\leq y_1-\delta\}$. We can further assume that $\delta^k\phi^{(k)}_{y_1,\delta}\leq C_k$, where $C_k$ is independent of $y_1$ and $\delta$. Let $\phi_{y_1,y_2,\delta}=\phi_{y_1,\delta}-\phi_{y_2+\delta,\delta}$. From the definition of $\Omega'$, we see that $\phi_{y_1,\delta}$ and $\phi_{y_1,y_2,\delta}$ are compactly supported when restricted to $\text{cl}(\Omega')$. We recall the following result from \cite{YaoRel} (see \cite[Theorem 3.1, Theorem 4.1]{YaoRel}):
\begin{prop}\label{35}
Let $U\in\mathcal{C}(\Gamma'_-,\Gamma'_+)$. There are constants $C=C(\Gamma_-,\Omega')$, $\epsilon=\epsilon(\Gamma_-,\Omega')>0$ and $E_0=E_0(\Gamma_-,\Omega')>0$, so that for $(y_1$,$y_2$,$\delta)$ satisfying $0<\frac{y_1}{2}<y_1-\delta<y_1<y_2-\delta$ and $y_2+\delta<2y_2<\epsilon$,
\begin{equation}
    E_{rel}[\partial^*U,\Gamma_-;\phi_{y_1,\delta}]\geq E_{rel}[\partial^*U,\Gamma_-;\phi_{y_2,\delta}]-Cy_2\geq-E_0\label{3}.
\end{equation}
This implies $E_{rel}[\partial^*U,\Gamma_-]:=E_{rel}[\partial^*U,\Gamma_-;\mathbf{1}]\in(-\infty,\infty]$ is well defined. Furthermore, for any $\psi\in\mathfrak{X}$,
\begin{equation}
    |E_{rel}[\partial^*U,\Gamma_-;\phi_{y_1,y_2,\delta}\psi]|\leq C(|E_{rel}[\partial^*U,\Gamma_-;\phi_{y_1,y_2,\delta}]|+y_2)\|\psi\|_{\mathfrak{X}}.\label{quant estimate}
\end{equation}
If $E_{rel}[\partial^*U,\Gamma_-]<\infty$, then
\begin{equation}
    |E_{rel}[\partial^*U,\Gamma_-;\psi]|\leq C(|E_{rel}[\partial^*U,\Gamma_-]|+1)\|\psi\|_{\mathfrak{X}}\label{1}.
\end{equation}
\end{prop}

\subsection{Banach manifold structure for minimal hypersurfaces}
We shall also use certain results from \cite{AlexakisMazzeo}. Let $\Lambda^{k,\alpha}_0(\Gamma_-)$ be the H\"older space on $\Gamma_-$, where the derivatives and difference quotients are measured with respect to $y\partial_y$ and $y\partial_x$. Here, $(x,y)$ is the local chart for $\Gamma_-$ near the ideal boundary such that $y$ is the ($n+1$)-th coordinate in $\mathbb{R}^{n+1}_+$ and $x$ restricts to coordinates along $M$. For any $\mu$, let $y^\mu\Lambda^{k,\alpha}_0:=\{f:\frac{f}{y^\mu}\in\Lambda^{k,\alpha}_0\}$. By \cite{AlexakisMazzeo} (see also \cite{MazzeoEdgeOperator}), the stability operator over $\Gamma_-$:
\begin{equation*}
L_{\Gamma_-}:=\Delta+|A_{\Gamma_-}|^2-n:\ y^\mu\Lambda^{k+2,\alpha}_0\longrightarrow y^\mu\Lambda^{k,\alpha}_0
\end{equation*}
is Fredholm of index zero for any $\mu\in(-1,n)$. Moreover, the kernel of $L_{\Gamma_-}$ is contained in $y^n\Lambda^{k,\alpha}_0$ for any such $\mu$. $\Gamma_-$ is called $\textit{non-degenerate}$ if Ker($L_{\Gamma_-}$) is trivial. We record the following results of the Banach manifold structure for minimal surfaces in $\mathbb{H}^3$.

\begin{prop}[Theorem 4.1 of \cite{AlexakisMazzeo}]\label{21}
Let $\mathcal{M}_g$ be the set of all minimal surfaces of genus $g$ in $\mathbb{H}^3$ that are $C^{3,\alpha}$-asymptotic to the ideal boundary. Then, $\mathcal{M}_g$ is a Banach manifold.
\end{prop}

\begin{prop}[Theorem 4.2 of \cite{AlexakisMazzeo}]\label{34}
Let $\mathcal{E}$ be the set of all $C^{3,\alpha}$ embedded curves in $\partial_\infty\mathbb{H}^3$. Let $\Pi:\mathcal{M}_g\to\mathcal{E}$ be the map assigning each minimal surface $\Gamma\in\mathcal{M}_g$ to its asymptotic boundary $\partial_\infty \Gamma\in\mathcal{E}$. Then, $\Pi$ is Fredholm of index $0$.
\end{prop}

In fact, Alexakis-Mazzeo's results can be easily generalized to higher dimensions. We will illustrate this with more details in the proof of Proposition \ref{20} and in Section 9.

\iffalse\begin{prop}\label{21}
Fix $\alpha\in(0,1)$ and $k\geq4$. Let $\mathcal{M}^{k,\alpha}$ be the set of all minimal hypersurfaces in $\mathbb{H}^{n+1}$ that are $C^{k,\alpha}$-asymptotic to the ideal boundary. Then, $\mathcal{M}^{k,\alpha}$ is a Banach manifold.
\end{prop}

Let $\mathcal{E}^{k,\alpha}$ be the set of all $C^{k,\alpha}$ closed $n$-dimensional manifolds in $\partial_\infty\mathbb{H}^{n+1}$, where $\alpha$, $k$ and $n$ are the same as in Proposition \ref{21}. The following result can be proved in the same manner as in \cite[Theorem 4.2]{AlexakisMazzeo} (we omit the proof here):

\begin{prop}\label{34}
The map $\Pi:\ \mathcal{M}^{k,\alpha}\rightarrow\mathcal{E}^{k,\alpha}$, assigning each hypersurface $\Gamma\in\mathcal{M}^{k,\alpha}$ to its asymptotic boundary $\partial_\infty \Gamma$, is Fredholm with index $0$.
\end{prop}\fi

\section{Relative Entropy}
Notice that for any $\psi\in\mathfrak{X}$, since $\psi(p,\textbf{v})=\psi(p,-\textbf{v})$, it can be viewed naturally as a function
over the Grassman $n$-plane bundle of $\Omega'$. However, the function space $\mathfrak{X}$ requires its elements be locally Lipschitz. So, if we add suitable functions to $\mathfrak{X}$ by letting
\begin{equation*}
    \mathfrak{Y}'=\{\psi+\phi:\ \psi\in\mathfrak{X},\ \phi\in C^0(\text{cl}(\Omega')\times\mathbb{S}^n),\ \|y^{-n}\phi\|_\infty<\infty\},
\end{equation*}
then $C^0_c(\text{cl}(\Omega')\times\mathbb{S}^n)$ (the space of compactly supported continuous function in $\text{cl}(\Omega')\times\mathbb{S}^n$) is contained in $\mathfrak{Y}'$. The advantage is that the functionals on $\mathfrak{Y}'$ can be naturally viewed as $n$-varifolds over $\text{cl}(\Omega')$.
$\mathfrak{Y}'$ is a Banach space equipped with the following norm\footnote{See \cite[Chapter 3, Theorem 1.3]{BennettSharpleyBook} for a proof.}:
\begin{equation*}
    \|f\|_{\mathfrak{Y}'}:=\inf\{\|\psi\|_\mathfrak{X}+\|y^{-n}\varphi\|_\infty:\ f=\psi+\varphi\}.
\end{equation*}
The following proposition shows that the estimates in Proposition \ref{35} still holds for weights in $\mathfrak{Y}'$. In particular, we show that for any $U\in\mathcal{C}(\Gamma'_-,\Gamma'_+)$ with $E_{rel}[\partial^*U,\Gamma_-]<\infty$, $E_{rel}[\partial^*U,\Gamma_-;\cdot]$ is a bounded linear functional on $\mathfrak{Y}'$.

\begin{prop}\label{36}
Let $U\in\mathcal{C}(\Gamma'_-,\Gamma'_+)$. The estimate $(\ref{quant estimate})$ holds for weights in $\mathfrak{Y}'$. Moreover, if $E_{rel}[\partial^*U,\Gamma_-]<\infty$, then there is a constant $C=C(\Gamma_-,\Omega')$, so that for any $f\in\mathfrak{Y}'$,
\begin{equation}\label{64}
    \big|E_{rel}[\partial^*U,\Gamma_-;f]\big|\leq C\big(1+\big|E_{rel}[\partial^*U,\Gamma_-]\big|\big)\|f\|_{\mathfrak{Y}'}.
\end{equation}
\end{prop}

\begin{proof}
For $f\in\mathfrak{Y}'$, we write $f=\psi+\varphi$ with $\psi\in\mathfrak{X}$ and $\|y^{-n}\varphi\|_\infty<\infty$. Then,
\begin{align}
    &\left|\int_{\partial^*U}\phi_{y_1,y_2,\delta}\varphi(p,\overbar{n}_{\partial^*U}(p))d\mathcal{H}^n_{\mathbb{H}}-\int_{\Gamma_-}\phi_{y_1,y_2,\delta}\varphi(p,\overbar{n}_{\Gamma_-}(p))d\mathcal{H}^n_{\mathbb{H}}\right|\label{100}\\
    \leq&\|y^{-n}\varphi\|_\infty\cdot\max\left\{\int_{\partial^*U}y^n\phi_{y_1,y_2,\delta}d\mathcal{H}^n_{\mathbb{H}},\int_{\Gamma_-}y^n\phi_{y_1,y_2,\delta}d\mathcal{H}^n_{\mathbb{H}}\right\}\\
    \leq&\|y^{-n}\varphi\|_\infty\cdot\left(\big|E_{rel}[\partial^*U,\Gamma_-;y^n\phi_{y_1,y_2,\delta}]\big|+2\int_{\Gamma_-}y^n\phi_{y_1,y_2,\delta}d\mathcal{H}^n_{\mathbb{H}}\right).
\end{align}
Since $y^nd\mathcal{H}^n_{\mathbb{H}}=d\mathcal{H}^n$ and $\Gamma_-$ is $C^1$-asymptotic to $M$, we have
\begin{align*}
    \int_{\Gamma_-}y^n\phi_{y_1,y_2,\delta}d\mathcal{H}^n_{\mathbb{H}^{n+1}}\leq \mathcal{H}^n(\Gamma_-\cap\{y\leq2y_2\})=\mathcal{O}(y_2).
\end{align*}
By ($\ref{quant estimate}$), $|E_{rel}[\partial^*U,\Gamma_-;y^n\phi_{y_1,y_2,\delta}]\big|\leq C(|E_{rel}[\partial^*U,\Gamma_-;\phi_{y_1,y_2,\delta}]|+y_2)$. Combining these facts and letting $\delta\to0$, we get
\begin{align}\label{62}
    \big|E_{rel}[\partial^*U,\Gamma_-;\varphi\chi_{\{y_1\leq y\leq y_2\}}]\big|\leq C\|y^{-n}\varphi\|_\infty(|E_{rel}[\partial^*U,\Gamma_-;\chi_{\{y_1\leq y\leq y_2\}}]|+\mathcal{O}(y_2)).
\end{align}

Now, we assume $E_{rel}[\partial^*U,\Gamma_-]<\infty$. By Cauchy's criterion, $\lim_{\epsilon\to0}E_{rel}[\partial^*U,\Gamma_-;\varphi\chi_{\{y\geq\epsilon\}}]$ exists and is a real number. Similarly, if we replace $\phi_{y_1,y_2,\delta}$ by another cut-off function $\phi_{y_1,\delta}$ in (\ref{100}), then
\begin{align*}
    \big|E_{rel}[\partial^*U,\Gamma_-;\varphi\phi_{y_1,\delta}]\big|\leq C\|y^{-n}\varphi\|_\infty(|E_{rel}[\partial^*U,\Gamma_-;\phi_{y_1,\delta}]|+1).
\end{align*}
Letting $\delta\to0$, $y_1\to0^+$ and using ($\ref{1}$) for $\psi=f-\varphi\in\mathfrak{X}$,
\begin{align}
    \big|E_{rel}[\partial^*U,\Gamma_-;f]\big|=&\big|E_{rel}[\partial^*U,\Gamma_-;\psi]\big|+\big|E_{rel}[\partial^*U,\Gamma_-;\varphi]\big|\\
    \leq&C\|\psi\|_{\mathfrak{X}}(|E_{rel}[\partial^*U,\Gamma_-]|+1)+C\|y^{-n}\varphi\|_\infty(|E_{rel}[\partial^*U,\Gamma_-]|+1)\\
    \leq& C(|E_{rel}[\partial^*U,\Gamma_-]|+1)(\|\psi\|_{\mathfrak{X}}+\|y^{-n}\varphi\|_\infty).
\end{align}
Taking the infimum over all such decompositions for $f$, we get the desired estimate.
\end{proof}

Notice that in general, $\mathfrak{Y}'$ is not separable. So, for technical reason, we will restrict ourselves to a separable subspace of $\mathfrak{Y}'$. Let $\pi$ be the projection of $\text{cl}(\Omega')\times\mathbb{S}^n$ onto $\mathbb{S}^n$. Let $\mathfrak{Y}_0:=\{f+\pi^*g:f\in C^0_c(\text{cl}(\Omega')\times\mathbb{S}^n),\,g\in C^2(\mathbb{S}^n), g(\textbf{v})=g(-\textbf{v})\}$, and $\mathfrak{Y}$ be the closure of $\mathfrak{Y}_0$ in $\mathfrak{Y}'$-norm. When it is clear from the context, we also use $g$ to represent $\pi^*g$. It is easy to see that $\mathfrak{Y}$ is an algebra.

\begin{prop}
The function space $\mathfrak{Y}$ is a separable Banach space endowed with $\mathfrak{Y}'$-norm $($will be denoted by $\mathfrak{Y}$-norm hereafter$)$. Moreover, $\{f+\pi^*g:\,f\in C^\infty_c(\text{cl}(\Omega')\times\mathbb{S}^n),\,g\in C^\infty(\mathbb{S}^n),g(\mathbf{v})=g(-\mathbf{v})\}$ is dense in $\mathfrak{Y}$.
\end{prop}

\begin{proof}
It is a standard exercise to check $\mathfrak{Y}$ is a Banach space. Observe that
\begin{equation*}
    C^0_c(\text{cl}(\Omega')\times\mathbb{S}^n)=\bigcup_{i=1}^\infty C^0_c(\text{cl}(\Omega'\cap\{y\geq\frac{1}{i}\})\times\mathbb{S}^n).
\end{equation*}
The Stone-Weierstrass theorem implies that each $C^0_c(\text{cl}(\Omega'\cap\{y\geq\frac{1}{i}\})\times\mathbb{S}^n)$ is separable in $C^0$-norm, and that $C^\infty_c(\text{cl}(\Omega'\cap\{y\geq\frac{1}{i}\})\times\mathbb{S}^n)$ is dense in it. On the other hand, for any $f\in C^0_c(\text{cl}(\Omega'\cap\{y\geq\frac{1}{i}\})\times\mathbb{S}^n)$,
\begin{equation*}
    \|f\|_\mathfrak{Y}\leq\|\frac{1}{y^n}f\|_\infty\leq i^n\|f\|_\infty.
\end{equation*}
As a result, $C^0_c(\text{cl}(\Omega')\times\mathbb{S}^n)$ is separable in $\mathfrak{Y}$-norm.

Using a mollifier function, it is easy to check that $\{g\in C^\infty(\mathbb{S}^n):g(\textbf{v})=g(-\textbf{v})\}$ is dense (in $C^2$-norm) in $\{g\in C^2(\mathbb{S}^n):g(\textbf{v})=g(-\textbf{v})\}$. Finally, we show that $C^2(\mathbb{S}^n)$ is separable in $C^2$-norm (which is the same as $\mathfrak{X}$-norm). For every $g\in C^2(\mathbb{S}^n)$, let $\widetilde{g}(\textbf{x})=g(\textbf{x}/|\textbf{x}|)$ be a $C^2_{loc}$ function in $\mathbb{R}^{n+1}\backslash\{0\}$. We know that $\widetilde{g}$ can be approximated by polynomials in $C^2$-norm over any compact set $K\subset\mathbb{R}^{n+1}\backslash\{0\}$ (see, e.g., \cite[Chapter 15-6,Corollary 2]{TrevesTVSBook} for a proof). Since polynomials are separable in $C^2(\overbar{B_2(0)}\backslash B_{\frac{1}{2}}(0))$, This implies $C^2(\mathbb{S}^n)$ is dense.
\end{proof}

Let $\mathfrak{Y}^*$ be the dual space of $\mathfrak{Y}$ and $\mathcal{C}(\Gamma'_-,\Gamma'_+;\Lambda)$ be the set of $U\in\mathcal{C}(\Gamma'_-,\Gamma'_+)$ with $E_{rel}[\partial^*U,\Gamma_-]\leq\Lambda$. Then, Proposition \ref{36} implies $\mathcal{C}(\Gamma'_-,\Gamma'_+;\Lambda)\subset\mathfrak{Y}^*$ for any $\Lambda$. Denote the weak*-closure of $\mathcal{C}(\Gamma'_-,\Gamma'_+;\Lambda)$ by $\mathfrak{Y}_\mathcal{C}^*(\Lambda)$. Since $C^0_c(\text{cl}(\Omega')\times\mathbb{S}^n)\subset\mathfrak{Y}$, it follows from the Riesz representation theorem that every $V$ in $\mathfrak{Y}^*_\mathcal{C}(\Lambda)$ is a Radon measure over $\text{cl}(\Omega')\times\mathbb{S}^n$. The dominated convergence theorem then implies that,
\begin{equation*}
    V[\{y\geq\epsilon\}]=V[\{(\textbf{x},y,\textbf{v})\in\mathbb{R}^n\times\mathbb{R}_+\times\mathbb{S}^n:\ y\geq\epsilon\}]=\lim_{\delta\to0^+}V[\phi_{\epsilon, \delta}]
\end{equation*}
is well defined, where $\phi_{\epsilon, \delta}$ is the cut-off function defined in the paragraph before Proposition \ref{35}. Moreover, we wish to extend the notion of relative entropy to $\mathfrak{Y}_\mathcal{C}^*(\Lambda)$ by formally letting
\begin{equation}
    E_{rel}[V]=\lim_{\epsilon\to0^+}V[{\{y\geq\epsilon\}}].\label{2}
\end{equation}
We prove that this quantity is well defined.

\begin{prop}\label{60}
Let $\Lambda>0$ be a fixed constant and $V$ be an element in $\mathfrak{Y}_\mathcal{C}^*(\Lambda)$. Then, $(\ref{2})$ is well defined and
\begin{equation*}
    -\infty<E_{rel}[V]\leq V[\mathbf{1}]\leq \Lambda<\infty,
\end{equation*}
where $\mathbf{1}$ is the constant function equal to $1$. Furthermore, there is a constant $C=C(\Gamma_-,\Omega',\Lambda)$, so for each $\psi\in\mathfrak{Y}$,
\begin{equation*}
    E_{rel}[V;\psi]:=\lim_{\epsilon\to0}V[\psi\chi_{\{y\geq\epsilon\}}]\text{ exists and }\big|E_{rel}[V;\psi]\big|\leq C\|\psi\|_\mathfrak{Y}.
\end{equation*}
\end{prop}

\begin{rem}
Note that $V[\cdot]$ and $E_{rel}[V;\cdot]$ are not equal in general. However, if $V[\textbf{1}]=E_{rel}[V]$, then these two functionals are the same.
\end{rem}

\begin{proof}
Let $\{U_i\}_i$ be a sequence in $\mathcal{C}(\Gamma'_-,\Gamma'_+)$ so that $\partial^*U_i$ converges to $V$ in the weak*-topology. If we replace $U$ by $U_i$ in (\ref{3}) and take $i\to\infty$, $\delta\to0^+$, then
\begin{equation}
    V[\{y\geq y_1\}]\geq V[\{y\geq y_2\}]-Cy_2.\label{4}
\end{equation}
Taking the lower limit on the left hand side over $y_1$ and the upper limit on the right over $y_2$,
\begin{equation*}
    \liminf_{\epsilon\to0^+}V[\{y\geq \epsilon\}]\geq\limsup_{\epsilon\to0^+}V[\{y\geq \epsilon\}].
\end{equation*}
This proves (\ref{2}) is well defined. If we replace $U$ by $U_i$ in (\ref{3}) and let $y_1\to0^+$, then
\begin{equation*}
    E_{rel}[\partial^*U_i,\Gamma_-]\geq E_{rel}[\partial^*U_i,\Gamma_-,\phi_{y_2,\delta}]-Cy_2.
\end{equation*}
If we let $i\to\infty$ notice that $E_{rel}[\partial^*U_i;\Gamma_-]=E_{rel}[\partial^*U_i;\Gamma_-;\mathbf{1}]$, then
\begin{equation*}
    V[\mathbf{1}]\geq V[\phi_{y_2,\delta}]-Cy_2.
\end{equation*}
This gives $V[\mathbf{1}]\geq E_{rel}[V]$. $V[\mathbf{1}]\leq\Lambda$ is immediate by the definition of weak*-limit, and $E_{rel}[V]>-\infty$ can be obtained by fixing $y_2$ and letting $y_1\to0^+$ in (\ref{4}).

By (\ref{quant estimate}) and (\ref{62}),
\begin{equation}\label{63}
    \big|E_{rel}[\partial^*U_i,\Gamma_-;\phi_{y_1,y_2,\delta}\psi]\big|\leq C(|E_{rel}[\partial^*U_i,\Gamma_-;\phi_{y_1,y_2,\delta}]|+y_2)\|\psi\|_\mathfrak{Y},
\end{equation}
for any $\psi\in\mathfrak{Y}$, $0<\frac{y_1}{2}<y_1-\delta<y_1<y_2-\delta$ and $y_2+\delta<2y_2<\epsilon$, where $\epsilon=\epsilon(\Gamma_-,\Omega')$ is defined in Proposition \ref{35}. Letting $i\to\infty$ in (\ref{63}), since $E_{rel}[V]$ is well defined and finite, by Cauchy's criterion, we have proved the existence of $E_{rel}[V;\psi]$. Applying Proposition \ref{36} to each $U_i$ and letting $i\to\infty$, we get the estimate $\big|E_{rel}[V;\psi]\big|\leq C\|\psi\|_\mathfrak{Y}$.
\end{proof}

Let $\{U_i\}_i$ be a sequence in $\mathcal{C}(\Gamma'_-,\Gamma'_+;\Lambda)$ converging to $V\in\mathfrak{Y}^*_\mathcal{C}(\Lambda)$. (\ref{3}) shows that
\begin{align*}
    \mathcal{H}^n_{\mathbb{H}^{n+1}}(\partial^*U_i\cap\{y\geq y_2\})\leq& E_{rel}[\partial^*U_i,\Gamma_-;\chi_{\{y\geq y_2\}}]+\mathcal{H}^n_{\mathbb{H}^{n+1}}(\Gamma_-\cap\{y\geq y_2\})\\
    \leq& Cy_2+\Lambda+\mathcal{H}^n_{\mathbb{H}^{n+1}}(\Gamma_-\cap\{y\geq y_2\})
\end{align*}
is uniformly bounded above. The compactness for Radon measures implies the existence of a varifold $V_+$ over $G_n(\text{cl}(\Omega'))$ as the limit of a certain subsequence of $\{\partial^*U_i\}$. More precisely, we have

\begin{prop}\label{6}
Let $\Lambda>0$ be a fixed constant and $V$ be an element in $\mathfrak{Y}_\mathcal{C}^*(\Lambda)$. Then, there is a varifold $V_+$ over $G_n(\mathrm{cl}(\Omega'))$, so that for any $f\in C^0_c\big(G_n(\mathrm{cl}(\Omega'))\big)$,
\begin{equation}
    V[f]=\int f(p,\mathbf{v})\frac{1}{y^n(p)}dV_+(p,\mathbf{v})-\int_{\Gamma_-}f(p,\overbar{\mathbf{n}}_{\Gamma_-})d\mathcal{H}^n_{\mathbb{H}}(p).\label{5}
\end{equation}
\end{prop}

\section{Non-compactly supported variational vector fields}

In this section, we study the variational theory for the relative entropy introduced in Section 2. Let $\Phi$ be a $C^2$ diffeomorphism from $\text{cl}(\Omega')$ into itself. For a function $\psi$ over $G_n(\text{cl}(\Omega'))$, the pull-back of $\psi$ under $\Phi$ is defined as
\begin{equation}\label{37}
    \Phi^\# \psi(p,\textbf{v})=\psi(\Phi(p),\nabla_\textbf{v}\Phi(p)).
\end{equation}
If $\psi\in\mathfrak{Y}$ and $\Phi$ is the identity map near the ideal boundary, then $\Phi^\# \psi\in\mathfrak{Y}$. For any $V\in\mathfrak{Y}^*(\Lambda)$, the push-forward of $V$ under $\Phi$ acting on $\psi$ is given by
\begin{equation}\label{101}
    \Phi_\# V[\psi]:=V[\Phi^\#\psi\mathcal{J}^E\Phi],
\end{equation}
where $\mathcal{J}^E\Phi$ the Jacobian of $\Phi$ with respect to the metric of $\mathbb{H}^{n+1}$. That is, $\mathcal{J}^E\Phi=\mathcal{J}\Phi(p,\textbf{v})\frac{p_{n+1}^n}{\Phi(p)_{n+1}^n}$, where the subcript ``$n+1$" means the $(n+1)$-th coordinate and $\mathcal{J}\Phi(p,\textbf{v})$ is the Jacobian computed for the $n$-plane perpendicular to $\textbf{v}$ with respect to the Euclidean metric.

As one may observe from Proposition \ref{60}, it might happen that $V[\mathbf{1}]$ does not coincide with $E_{rel}[V]$. As a counter-example, one can think of a sequence of hypersurfaces $\{\Gamma_i\}_i$, where each $\Gamma_i$ agrees with $\Gamma_-$ in most part, and $\Gamma_i\backslash\Gamma_-$ is a bump over $\Gamma_-\backslash\Gamma_i$ that moves towards the ideal boundary. By properly choosing a subsequence, we can assume that $E_{rel}[\Gamma_i,\Gamma_-]$ converges to a non-zero value (as $i\to\infty$). However, the limit varifold $V$ vanishes when acting on functions in $C^0_c\big(G_n(\text{cl}(\Omega'))\big)$, and hence $E_{rel}[V]=0<V[\mathbf{1}]$. To rule out this phenomenon, we follow \cite{BWMountainPass} (see also \cite{KetoverZhou}) to consider the variations of appropriate non-compactly supported vector fields.

Let $\textbf{N}_\epsilon:=y(q)\overbar{\textbf{e}}_{n+1}^\top\circ\Pi_{\Gamma_-}(q)\chi_\epsilon$, where $\Pi_{\Gamma_-}$ is the normal projection to $\Gamma_-$ and $\chi_\epsilon$ is a cut-off function in $y=y(q)$ which is identity near $\{y=0\}$ and vanishes in $\{y\geq\epsilon\}$. We choose $\epsilon$ so small that $\textbf{N}_\epsilon$ is well defined in $\text{cl}(\Omega')$. We will show in the next proposition that by choosing $\Gamma'_-$, $\Gamma'_+$ and $\epsilon$ properly, $\textbf{N}_\epsilon$ points into $\text{cl}(\Omega')$. So, the $1$-parameter family of diffeomorphisms, $\{\Phi^{\textbf{N}_\epsilon}_t\}_t$, generated by $\textbf{N}_\epsilon$ maps $\text{cl}(\Omega')$ into itself. We also show that $\text{cl}(\Omega')$ is contained in a mean-convex foliation. This is crucial when applying the maximum principle \cite{SolomonWhite} in the proof of Proposition \ref{70}.

\begin{prop}\label{20}
Let $M$ be an $(n-1)$-dimensional $C^{4,\alpha}$ closed manifold in $\mathbb{R}^n_+\times\{0\}$ and $\Gamma_-$, $\Gamma_+$ be two strictly stable minimal hypersurfaces that are both $C^{4,\alpha}$-asymptotic to $M$. Then, there are hypersurfaces $\Gamma'_-$ and $\Gamma'_+$ which bound a thin neighborhood, $\Omega'$, of $\Gamma_-$, such that
\begin{enumerate}
    \item $\Gamma'_-\preceq\Gamma_-\preceq\Gamma_+\preceq\Gamma'_+$;
    \item There is a mean-convex foliation $\{\Gamma_-^s\}_{s\in[-1,1]}$ $($resp. $\{\Gamma_+^s\}_{s\in[-1,1]})$ of $\Gamma_-$ $($resp. $\Gamma_+)$ with $\Gamma_-^{-s}\preceq\Gamma_-\preceq\Gamma_-^s$ $($resp. $\Gamma_+^{-s}\preceq\Gamma_+\preceq\Gamma_+^s)$ for any $s\in[-1,1]$;
    \item $\Gamma'_\pm\cap\Gamma_\pm^s\subset\subset\{y>0\}$;
    \item\label{31} If we choose $\epsilon$ sufficiently small, then $\mathbf{N}_\epsilon\cdot\mathbf{n}_{\partial\Omega'}\leq0$, where $\mathbf{n}_{\partial\Omega'}$ points out of $\Omega'$.
\end{enumerate}
\end{prop}

\begin{rem}
Unlike in \cite{BWMountainPass}, in our setting, the first eigenvalue of the stability operator may not coincide with the infimum of the corresponding energy functional, so the first eigenfunction may not have a sign. So, we instead follow \cite{AlexakisMazzeo} to construct the mean-convex foliations by moving the ideal boundary slightly and solving the appropriate equations.
\end{rem}

\begin{proof}
Let $L_{\Gamma_-}=\Delta_{\Gamma_-}+|A_{\Gamma_-}|^2-n$ be the stability operator over $\Gamma_-$. From the knowledge of Section 2, $L_{\Gamma_-}$ is Fredholm of index $0$, and $\text{Ker(}L_{\Gamma_-}\text{)}$ is contained in $y^n\Lambda^{2,\alpha}_0$. Hence, for $u\in\text{Ker(}L_{\Gamma_-}\text{)}$,
\begin{align*}
    0&=\int_{\Gamma_-}u\cdot(-L_{\Gamma_-}u)d\mathcal{H}^n_{\mathbb{H}}=\int_{\Gamma_-}u\cdot(-\Delta_{\Gamma_-}u)+(n-|A_{\Gamma_-}|^2)u^2d\mathcal{H}^n_{\mathbb{H}}\\
    &=\int_{\Gamma_-}|\nabla_{\Gamma_-}u|^2+(n-|A_{\Gamma_-}|^2)u^2d\mathcal{H}^n_{\mathbb{H}}.
\end{align*}
The  (strict) stability assumption on $\Gamma_-$ implies $u\equiv0$. So, $L_{\Gamma_-}:\ y^\mu\Lambda^{2,\alpha}_0\rightarrow y^\mu\Lambda^{0,\alpha}_0$ is an isomorphism.

The first step is to construct ``approximate" minimal hypersurfaces. Let $M_\psi=\{\textbf{x}(p)+\psi(p)\textbf{n}_{M}(p):\ p\in M\}$ be a perturbation of $M$ corresponding to $\psi\in C^{2,\alpha}(M)$. We define an approximate minimal hypersurface, $\Gamma_{-,\psi}$, asymptotic to $M_\psi$ as follows: it is a graph of function $u_\psi$ over $\Gamma_-$ near the ideal boundary, and $u_\psi(p,0)$, $\partial_y u_\psi(p,0)$ and $\partial^2_{yy}u_\psi(p,0)$ are determined by the formal expansion of a minimal hypersurface asymptotic to $M_\psi$, which are determined by $M_\psi$ as proved in \cite[Lemma 3.3]{YaoRel}. Let $\mathcal{E}(\psi)=\frac{1}{y}\chi u_\psi$, where $\chi$ is a cut-off function that is equal to $1$ near $\{y=0\}$. Then, it can be checked that the mean curvature of $\Gamma_{-,\psi}:=\{\textbf{x}(p)+\mathcal{E}(\psi)\textbf{n}_{\Gamma_-}(p):\ p\in\Gamma_-\}$, $H_\psi$, is in $y^{\mu_0}\Lambda^{0,\alpha}_0$ for some $\mu_0>0$.

Then, We make a small perturbation of the approximate minimal hypersurface, so the implicit function theorem can apply in a proper sense. Let $\Gamma_{\psi,\phi}:=\{\textbf{x}(p)+(\mathcal{E}(\psi)+\phi)\textbf{n}_{\Gamma_-}(p):\ p\in\Gamma_-\}$ for some $\phi\in y^{\mu_0}\Lambda^{2,\alpha}_0$. The mean curvature of $\Gamma_{\psi,\phi}$, $H_{\psi,\phi}$, is in $y^{\mu_0}\Lambda^{0,\alpha}_0$. Hence, $H_{\cdot,\cdot}$ maps from a neighborhood of $(0,0)$ in $C^{2,\alpha}(M)\times y^{\mu_0}\Lambda^{2,\alpha}_0(\Gamma_-)$ to $y^{\mu_0}\Lambda_0^{0,\alpha}(\Gamma_-)$. Moreover,
\begin{equation*}
    DH|_{(0,0)}(\hat{\psi},\hat{\phi})=-L_{\Gamma_-}\circ D\mathcal{E}|_0(\hat{\psi})-L_{\Gamma_-}(\hat{\phi})
\end{equation*}
By the implicit function theorem, there exists a map $\mathcal{G}$ mapping from a neighborhood of $(0,0)$ in $C^{2,\alpha}(M)\times y^{\mu_0}\Lambda^{0,\alpha}_0$ to $y^{\mu_0}\Lambda^{2,\alpha}_0$, such that $H_{\psi,\mathcal{G}(\psi,u)}\equiv u$. Differentiating this equation at $(0,0)$, one has
\begin{equation*}
    D\mathcal{G}|_{(0,0)}(\hat{\psi},\hat{u})=-L^{-1}_{\Gamma_-}(\hat{u})-L^{-1}_{\Gamma_-}\circ(L_{\Gamma_-}\circ D\mathcal{E}|_0)(\hat{\psi}).
\end{equation*}

Next, we show that by appropriately choosing $(\hat{\psi},\hat{u})$, with $\hat{\psi},\hat{u}\geq0$, the map
\begin{equation*}
    (p,t)\mapsto\Gamma_{t\hat{\psi},\mathcal{G}(t\hat{\psi},t\hat{u})}\text{, with }(p,t)\in \Gamma_-\times(-\delta_1,\delta_1)\text{ for some }\delta_1>0
\end{equation*}
forms a mean convex foliation. This is equivalent to show
\begin{equation}
    D\mathcal{E}|_{0}(\hat{\psi})+D\mathcal{G}|_{(0,0)}(\hat{\psi},\hat{u})=D\mathcal{E}|_{0}(\hat{\psi})-L^{-1}_{\Gamma_-}\circ(L_{\Gamma_-}\circ D\mathcal{E}|_0)(\hat{\psi})-L^{-1}_{\Gamma_-}(\hat{u})\label{differential}
\end{equation}
is positive for every $p\in\overbar{\Gamma}_-$.

We claim that the operator $L_{\Gamma_-}:\ y^{\mu_0}\Lambda^{2,\alpha}_0\to y^{\mu_0}\Lambda^{0,\alpha}_0$ satisfies a maximum principle. Let $y_0>0$ be so small that $n-|A_{\Gamma_-}|^2(p)>0$ for all $p$ with $y(p)\leq y_0$. Suppose $L_{\Gamma_-}(u)=f\geq0$. Let $y_1<y_0$ and $u_{y_1}=u-\max_{p\in\{y=y_1\}}u^+(p)$. Then, $L_{\Gamma_-}(u_{y_1})=f+(n-|A_{\Gamma_-}|^2)\max_{p\in\{y=y_1\}}u^+(p)\geq0$. The stability condition on $\Gamma_-\cap\{y\geq y_1\}$ together with the maximum principle \cite[Theorem 2.11]{HanLinEllipticPDE} gives $u_{y_1}\leq0$. That is, $u(q)\leq \max_{p\in\{y=y_1\}}u^+(p)$ for all $q\in\{y\geq y_1\}$. Let $y_1\to0$ and notice that $u(p)\big(\in y^{\mu_0}\Lambda^{2,\alpha}_0\big)$ decays to $0$ as $y(p)\to0$. This gives $u(q)\leq0$ for all $q\in\Gamma_-$. One can then infer from the strong maximum principle that $u(p)$ is strictly negative.

By choosing $\hat{u}=sy^{\mu_0}$ and $\hat{\psi}\equiv1$, (\ref{differential}) equals
\begin{equation*}
    D\mathcal{E}|_{0}(\hat{\psi})-L^{-1}_{\Gamma_-}\circ(L_{\Gamma_-}\circ D\mathcal{E}|_0)(\hat{\psi})-sL^{-1}_{\Gamma_-}(y^{\mu_0}).
\end{equation*}
It is automatically positive near the ideal boundary, since $D\mathcal{E}(\hat{\psi})\sim\frac{1}{y}$ near $\{y=0\}$. A large $s$ will ensure it is also positive away from the ideal boundary. Hence, $\Gamma_-^t:=\Gamma_{\delta_1t\hat{\psi},\mathcal{G}(\delta_1t\hat{\psi},\delta_1t\hat{u})}$ ($-1\leq t\leq1$) gives a mean-convex foliation of $\Gamma_-$. Similarly, we can construct a mean-convex foliation $\{\Gamma_+^t\}_{t\in[-1,1]}$ of $\Gamma_+$.

Finally, we show that a choice of $\Gamma'_\pm$ meets all the requirements. Let $\Gamma'_+$ be a hypersurface with $\Gamma_+\preceq\Gamma'_+\preceq\Gamma_+^1$. Since $\Gamma_+$ is in a thin neighborhood of $\Gamma_-$, we can choose constants $C$ and $\epsilon$, so that $\Gamma'_+\cap\{y<\epsilon\}\subset\{\textbf{x}(p)+Cy^{n+1}\overbar{\textbf{n}}_{\Gamma_-}(p):\ p\in\Gamma_-\}$. By direct computation, $\overbar{\textbf{n}}\circ\textbf{h}-\overbar{\textbf{n}}_{\Gamma_-}=-\overbar{\nabla}_{\Gamma_-} f+\textbf{Q}(f,\overbar{\nabla}_{\Gamma_-} f)$, where $\textbf{h}=\text{id}+f\overbar{\textbf{n}}_{\Gamma_-}=\text{id}+Cy^{n+1}\overbar{\textbf{n}}_{\Gamma_-}$ pararmetrizes $\Gamma'_+$ near the ideal boundary and $|\textbf{Q}(f,\overbar{\nabla}_{\Gamma_-}f)|\leq C(|f|+|\overbar{\nabla}_{\Gamma_-}f|)^2$. Then,
\begin{align*}
    \overbar{\textbf{n}}\circ\textbf{h}\cdot\overbar{\textbf{e}}_{n+1}^\top=(\overbar{\textbf{n}}\circ\textbf{h}-\overbar{\textbf{n}}_{\Gamma_-})\cdot\overbar{\textbf{e}}_{n+1}^\top=-C(n+1)y^n|\overbar{\textbf{e}}_{n+1}^\top|^2+\textbf{Q}(f,\overbar{\nabla}_{\Gamma_-} f)\cdot\overbar{\textbf{e}}_{n+1}^\top
\end{align*}
with
\begin{align*}
    |\textbf{Q}(f,\overbar{\nabla}_{\Gamma_-} f)\cdot\overbar{\textbf{e}}_{n+1}^\top|\leq|\textbf{Q}(f,\overbar{\nabla}_{\Gamma_-} f)|\leq C(y^{n+1}+y^n)^2=\mathcal{O}(y^{n+1}).
\end{align*}
By shrinking $\epsilon$ if necessary, we have $\overbar{\textbf{e}}_{n+1}^\top\circ\Pi(q)\cdot\textbf{n}_{\Gamma'_+}(q)\leq0$, for any $q\in\Gamma'_+\cap\{y<\epsilon\}$. This implies $\textbf{N}_\epsilon$ points into $\Omega'$ on $\Gamma'_+$. Finally, let $\Gamma'_-$ with $\Gamma_-^{-1}\preceq\Gamma'_-\preceq\Gamma_-$ be a hyersurface which coincides with $\{\textbf{x}(p)-Cy^{n+1}\overbar{\textbf{n}}_{\Gamma_-}(p):\ p\in\Gamma_-\}$ in $\{y\leq\epsilon\}$. Then, it can be proved in a similar way that $\textbf{N}_\epsilon$ points into $\Omega'$ on $\Gamma'_-$.
\end{proof}

Let $\mathcal{Y}^-=\{\textbf{Y}=\alpha\textbf{N}_\epsilon+\textbf{V}:\ \textbf{V}\text{ is smooth and supported in }\{y>0\}\text{, and }\textbf{Y}\cdot\textbf{n}_{\partial\Omega'}\leq0\}$. Let $\{\Phi^{\textbf{Y}}_t\}_{t\geq0}$ be the 1-parameter family of diffeomorphisms generated by $\textbf{Y}\in\mathcal{Y}^-$. Define
\begin{equation*}
    \|\textbf{V}\|_{\mathcal{Y}_c}=\|\frac{1}{y^2}\textbf{V}\|_\infty+\|\frac{1}{y}\overbar{\nabla}\textbf{V}\|_\infty+\|\overbar{\nabla}^2\textbf{V}\|_\infty,
\end{equation*}
and let $\|\textbf{Y}\|_{\mathcal{Y}^-}=|\alpha|+\|\textbf{V}\|_{\mathcal{Y}_c}$.

\begin{lem}\label{44}
For any vector field $\mathbf{Y}=\alpha\mathbf{N}_\epsilon+\mathbf{V}\in\mathcal{Y}^-$ with $\|\mathbf{Y}\|_{\mathcal{Y}^-}\leq M_0$, let $\{\Phi_t(\text{p})\}$ be the 1-parameter family of diffeomorphisms generated by $\mathbf{Y}$. If $V\in\mathfrak{Y}^*(\mathrm{cl}(\Omega'))$, then $\Phi_{t\#}V\in\mathfrak{Y}$ and $\|\Phi_{t\#}V\|_{\mathfrak{Y}^*}\leq C\|V\|_{\mathfrak{Y}^*}$ for $0\leq t\leq T$ and $C=C(\Gamma_-,\Omega',M_0,T)$.
\end{lem}

\begin{proof}
By definition, for any $\psi\in\mathfrak{Y}$, $\Phi_{t\#} V[\psi]:=V[\Phi_t^\#\psi\mathcal{J}^E\Phi_t]$. By Lemma \ref{a3} and Lemma \ref{a5}, we know that $\Phi_t^\#\psi\mathcal{J}^E\Phi_t\in\mathfrak{Y}$ and $\|\Phi_t^\#\psi\mathcal{J}^E\Phi_t\|_\mathfrak{Y}\leq C\|\psi\|_\mathfrak{Y}$ for $C=C(\Gamma_-,\Omega',M_0,T)$ and $0\leq t\leq T$. This implies $\Phi_{t\#} V\in\mathfrak{Y}^*$ and $\|\Phi_{t\#}V\|_{\mathfrak{Y}^*}\leq C\|V\|_{\mathfrak{Y}^*}$.
\end{proof}

\begin{prop}\label{52}
For any vector field $\mathbf{Y}=\alpha\mathbf{N}_\epsilon+\mathbf{V}\in\mathcal{Y}^-$, let $\{\Phi^\mathbf{Y}_t(\text{p})\}$ be the 1-parameter family of diffeomorphisms generated by $\mathbf{Y}$. Then,
\begin{equation*}
    [0,\infty)\times\mathcal{Y}^-\times B^{\mathfrak{Y}^*}_{R_0}\ni(t,\mathbf{Y},V)\mapsto\Phi^{\mathbf{Y}}_{t\#}V
\end{equation*}
is continuous, where $B^{\mathfrak{Y}^*}_{R_0}$ is the closed ball of radius $R_0$ centered at the origin in $\mathfrak{Y}^*$ and endowed with the weak$^*$-topology. Furthermore, the map $t\mapsto\Phi^{\mathbf{Y}}_{t\#}V[\mathbf{1}]$ is differentiable with
\begin{equation}\label{48}
    \frac{d}{dt}\Big|_{t=0}\Phi^{\mathbf{Y}}_{t\#}V[\mathbf{1}]=\delta V[\mathbf{Y}]=V[\overbar{\mathrm{div}}\mathbf{Y}-Q_\mathbf{Y}-\frac{n}{y}\mathbf{Y}\cdot\overbar{\mathbf{e}}_{n+1}],
\end{equation}
where $Q_\mathbf{Y}(p,\mathbf{v})=\overbar{\nabla}_\mathbf{v}\mathbf{Y}(p)\cdot\mathbf{v}$.
\end{prop}

\begin{proof}
Fix $\textbf{Y}_0=\alpha_0\textbf{N}_\epsilon+\textbf{V}_0\in\mathcal{Y}^-$, $t_0\geq0$ and $V_0\in B^{\mathfrak{Y}^*}_{R_0}$. For any $f\in C_c^\infty(\text{cl}(\Omega')\times\mathbb{S}^n)$ and ($t,\textbf{Y}$) in a neighborhood of ($t_0,\textbf{Y}_0$), since Spt($\Phi_t^{\textbf{Y}\#}f$) is contained in a compact subset of $\{y>0\}$, we have
\begin{equation}\label{43}
    \lim_{(t,\textbf{Y},V)\to(t_0,\textbf{Y}_0,V_0)}\Phi^{\textbf{Y}}_{t\#}V[f]=\Phi^{\textbf{Y}_0}_{t_0\#}V_0[f].
\end{equation}
On the other hand, for $g\in C^\infty(\mathbb{S}^n)$ with $g(\textbf{v})=g(-\textbf{v})$, since $\overbar{\nabla}\Phi^\textbf{Y}_t$ and $\overbar{\nabla}^2\Phi^\textbf{Y}_t$ vary in a continuous way in $(t,\textbf{Y},p)$, one has
\begin{equation}\label{42}
    \phi_{y_1,\delta}\Phi^{\textbf{Y}\#}_tg(\textbf{v})=\phi_{y_1,\delta}g(\overbar{\nabla}_\textbf{v}\Phi^\textbf{Y}_t(p))\text{ converges to }\phi_{y_1,\delta}\Phi^{\textbf{Y}_0\#}_{t_0}g(\textbf{v})
\end{equation}
in $\mathfrak{X}$-norm for any $0<\delta<y_1$. When $(t,\textbf{Y})$ is near $(t_0,\textbf{Y}_0)$, there are constants $T$ and $M_0$, so that $t\leq T$ and $\|\textbf{Y}\|_{\mathcal{Y}^-}\leq M_0$. So, it follows from (\ref{41}) and the uniform continuity of $g$ that as $y_1\to0$,
\begin{equation}\label{46}
    \|(1-\phi_{y_1,\delta})\big(g(\overbar{\nabla}_{\textbf{v}}\Phi^\textbf{Y}_t(p))-g_1(\textbf{v})\big)\|_\mathfrak{X}\to0,
\end{equation}
where $g_1(\textbf{v})=g(\textbf{v}+(e^{\alpha_0 t}-1)(\textbf{v}\cdot\overbar{\textbf{e}}_{n+1})\overbar{\textbf{e}}_{n+1})$. Combining (\ref{42}) and (\ref{46}), we see that $\Phi^{\textbf{Y}\#}_tg$ converges to $\Phi^{\textbf{Y}_0\#}_{t_0}g$ in $\mathfrak{X}$-norm.

We use cut-off functions in a similar way as in (\ref{42}) and (\ref{46}) to show that $\|\mathcal{J}^E\Phi^\textbf{Y}_t-\mathcal{J}^E\Phi^{\textbf{Y}_0}_{t_0}\|_\mathfrak{X}\to0$. Here, instead of using the function $g_1(\textbf{v})$, we use the corresponding functions that appear in (\ref{103}) and (\ref{104}). Hence,
\begin{equation}\label{45}
    \lim_{(t,\textbf{Y},V)\to(t_0,\textbf{Y}_0,V_0)}\Phi^{\textbf{Y}}_{t\#}V[g]=\Phi^{\textbf{Y}_0}_{t_0\#}V_0[g].
\end{equation}
From (\ref{43}) and (\ref{45}), we have $\Phi^{\textbf{Y}}_{t\#}V[\psi]\to\Phi^{\textbf{Y}_0}_{t_0\#}V_0[\psi]$ for $\psi=f+g$, where $f\in C^\infty_c(\text{cl}(\Omega')\times\mathbb{S}^n)$ and $g\in C^\infty(\mathbb{S}^n)$ with $g(\textbf{v})=g(-\textbf{v})$.

For a general $\psi\in\mathfrak{Y}$, we choose $\{\psi_n\}_{n=1}^\infty\subset\{f+g:f\in C_c^\infty(\text{cl}(\Omega')\times\mathbb{S}^n),g\in C^\infty(\mathbb{S}^n)\}$ with $\|\psi_n-\psi\|_\mathfrak{Y}\to0$. Then,
\begin{equation*}
    |\Phi^\textbf{Y}_{t\#}V[\psi]-\Phi^{\textbf{Y}_0}_{t_0\#}V_0[\psi]|\leq |\Phi^\textbf{Y}_{t\#}V[\psi_n]-\Phi^{\textbf{Y}_0}_{t_0\#}V_0[\psi_n]|+|\Phi^\textbf{Y}_{t\#}V[\psi_n-\psi]|+|\Phi^{\textbf{Y}_0}_{t_0\#}V_0[\psi_n-\psi]|
\end{equation*}
The second and the third terms are bounded by $C\|\psi_n-\psi\|_\mathfrak{Y}$ by Lemma \ref{44}. So, we have proved the continuous dependence for a general $\psi\in\mathfrak{Y}$.

Now, we prove that $t\mapsto\Phi^\textbf{Y}_{t\#}V[\mathbf{1}]$ is differentiable. Since $\Phi^\textbf{Y}_{t\#}V[\mathbf{1}]=V[\mathcal{J}^E\Phi^\textbf{Y}_{t\#}]$, the differentiability follows from Lemma \ref{a4}
\end{proof}

\section{Stationarity of relative entropy and its properties}
In this section, we introduce an appropriate notion of stationarity with the presence of obstacles for the relative entropy and discuss its properties. Let $\textbf{Y}\in\mathcal{Y}^-$ and $\{\Phi_t\}_{t\geq0}$ be the 1-parameter family of diffeomorphisms generated by $\textbf{Y}$. For any $V\in\mathfrak{Y}^*$, We set
\begin{equation*}
    \Phi_{t\#}^+V=\Phi_{t\#}V+V_{\Phi_{t}(\Gamma_-)},
\end{equation*}
where $V_{\Phi_t(\Gamma_-)}$ is the functional given by $E_{rel}[\Phi_t(\Gamma_-),\Gamma_-;\cdot]$. Since $\textbf{Y}$ fix $\Gamma_-$ near $\{y=0\}$, we see that $\Phi_t(\Gamma_-)$ coincides with $\Gamma_-$ outside a compact subset of $\{y>0\}$. Thus, $V_{\Phi_t(\Gamma_-)}$ is well defined. The advantage of this notion is that $\Phi^+_{t\#}$ maps $\mathfrak{Y}^*_\mathcal{C}(\Lambda)$ to $\mathfrak{Y}^*_\mathcal{C}(\Lambda')$ for some $\Lambda'$.

\begin{prop}
Let $U\in\mathcal{C}(\Gamma'_-,\Gamma'_+)$ with $\Gamma=\partial^*U$ and $V_{\Gamma}[\cdot]=E_{rel}[\Gamma,\Gamma_-;\cdot]$. If $E_{rel}[\Gamma,\Gamma_-]<\infty$, then we have
\begin{equation*}
    \Phi^+_{t\#}V_{\Gamma}=\Phi_{t\#}V_{\Gamma}+V_{\Phi_t(\Gamma_-)}=V_{\Phi_t(\Gamma)},
\end{equation*}
where $V_{\Phi_t(\Gamma)}[\cdot]=E_{rel}[\Phi_t(\Gamma),\Gamma_-;\cdot]$.
\end{prop}

\begin{proof}
Fix any $\psi\in\mathfrak{Y}$. By definition, we have
\begin{equation}\label{49}
    \Phi_{t\#}V_\Gamma[\psi]=V_{\Gamma}[\mathcal{J}^E\Phi_t\Phi^\#_t\psi]=\lim_{\epsilon\to0}E_{rel}[\Gamma,\Gamma_-;\Phi^\#_t\psi\mathcal{J}^E\Phi_t\chi_{\{y\geq\epsilon\}}].
\end{equation}
Since $\Phi_t(\Gamma_-)$ coincides with $\Gamma_-$ near $\{y=0\}$, we know that $E_{rel}[\Phi_t(\Gamma_-),\Gamma_-]<\infty$. So,
\begin{align}
    V_{\Phi_t(\Gamma_-)}[\psi]&=\lim_{\epsilon\to0}E_{rel}[\Phi_t(\Gamma_-),\Gamma_-;\Phi^\#_t\psi\mathcal{J}^E\Phi_t\chi_{\{y\geq\epsilon\}}]\label{50}
\end{align}
On the other hand, by letting $f=\Phi^\#_t\psi\mathcal{J}^E\Phi_t$,
\begin{equation}\label{51}
    E_{rel}[\Gamma,\Gamma_-;f\chi_{\{y\geq\epsilon\}}]+E_{rel}[\Phi_t(\Gamma_-),\Gamma_-;f\chi_{\{y\geq\epsilon\}}]=E_{rel}[\Phi_t(\Gamma),\Gamma_-;f\chi_{\{y\geq\epsilon\}}].
\end{equation}
Combining (\ref{49}), (\ref{50}), and (\ref{51}), we get $V_{\Phi_t(\Gamma)}[\psi]=\Phi_{t\#}V_\Gamma[\psi]+V_{\Phi_t(\Gamma_-)}[\psi]$ is well-defined and finite valued for any $\psi\in\mathfrak{Y}$. In particular, $V_{\Phi_t(\Gamma)}[\mathbf{1}]=\Phi_{t\#}V_\Gamma[\mathbf{1}]+V_{\Phi_t(\Gamma_-)}[\mathbf{1}]<\infty$.
\end{proof}

Now, we introduce the notion of $E_{rel}$-minimizing.

\begin{defn}
An element $V\in\mathfrak{Y}^*$ is called $E_{rel}$-minimizing to the first order in $\text{cl}(\Omega')$, if for any $\textbf{Y}\in\mathcal{Y}^-$,
\begin{equation*}
    \delta^+ V[\textbf{Y}]=\frac{d}{dt}\Big|_{t=0}\Phi^+_{t\#}V[\mathbf{1}]=\frac{d}{dt}\Big|_{t=0}\left(\Phi_{t\#}V+V_{\Phi_{t}(\Gamma_-)}\right)[\mathbf{1}]\geq0.
\end{equation*}
\end{defn}

Notice that for any $\textbf{Y}\in\mathcal{Y}^-$, $V_{\Phi_t(\Gamma_-)}$ can be viewed as a compactly supported variation of $\Gamma_-$. Since $\Gamma_-$ is minimal, we have $\frac{d}{dt}\big|_{t=0}V_{\Phi_t(\Gamma_-)}[\mathbf{1}]=0$. So, $\delta^+V[\textbf{Y}]\geq0$ is equivalent to $\delta V[\textbf{Y}]=\frac{d}{dt}\Big|_{t=0}\Phi_{t\#}V[\mathbf{1}]\geq0$.

\begin{prop}\label{68}
For $\mathbf{Y}\in\mathcal{Y}^-$, let $\{\Phi^\mathbf{Y}_t\}_{t\geq0}$ be the 1-parameter family of diffeomorphisms. Let $B^{\mathfrak{Y}^*}_{R_0}$ be the closed ball of radius $R_0$ centered at the origin in $\mathfrak{Y}^*$ and endowed with the weak$^*$-topology. Then, we have
\begin{enumerate}
    \item the map\label{53}
    \begin{equation*}
        [0,\infty)\times\mathcal{Y}^-\times B^{\mathfrak{Y}^*}_{R_0}\ni(t,\mathbf{Y},V)\mapsto(\Phi^{\mathbf{Y}}_t)^+_{\#}V
    \end{equation*}
    is continuous;
    \item the map\label{57}
    \begin{equation*}
        [0,\infty)\times\mathcal{Y}^-\times\mathcal{Y}^-\times B^{\mathfrak{Y}^*}_{R_0}\ni(t,\mathbf{Y},\mathbf{Z},V)\mapsto\delta\Big((\Phi^{\mathbf{Y}}_t)^+_{\#}V\Big)[\mathbf{Z}]
    \end{equation*}
    is continuous.
\end{enumerate}
\end{prop}

\begin{proof}
By proposition \ref{52}, item (\ref{53}) is equivalent to the continuity of the map
\begin{equation}\label{56}
    [0,\infty)\times\mathcal{Y}^-\ni(t,\textbf{Y})\mapsto V_{\Phi^{\textbf{Y}}_t(\Gamma_-)}.
\end{equation}
Let $0\leq t\leq T$ and $\|\textbf{Y}\|_{\mathcal{Y}^-}\leq M_0$. Then, by (\ref{41}) and direct computation, we see that $|\overbar{A}_{\Phi^\textbf{Y}_t(\Gamma_-)}|\leq C(\Gamma_-,\Omega',T,M_0)$. Hence, by \cite[Lemma 5.2]{YaoRel}, there are constants $\epsilon_1=\epsilon_1(\Gamma_-,\Omega',M_0,T)$ and $C_1=C_1(\Gamma_-,\Omega',M_0,T)$, so that for $0<\frac{s_1}{2}<s_1-\delta<s_1<s_2<s_2+\delta<2s_2<\epsilon_1$,
\begin{equation*}
    \big|E_{rel}[\Phi^{\textbf{Y}}_t(\Gamma_-),\Gamma_-;\phi_{s_1,s_2,\delta}]\big|\leq C_1s_2.
\end{equation*}
Letting $s_1\to0$ and using (\ref{quant estimate}), one sees that for any $\psi\in\mathfrak{Y}$,
\begin{equation}\label{54}
    \big|E_{rel}[\Phi^{\textbf{Y}}_t(\Gamma_-),\Gamma_-;(1-\phi_{s_2,\delta})\psi]\big|\leq C_1s_2.
\end{equation}
Since $\phi_{s_2,\delta}\psi$ is compactly supported in $\{y>0\}$, one also has
\begin{equation}\label{55}
    \lim_{(t,\textbf{Y})\to(t_0,\textbf{Y}_0)}V_{\Phi^{\textbf{Y}}_t(\Gamma_-)}[\phi_{s_2,\delta}\psi]=\lim_{(t,\textbf{Y})\to(t_0,\textbf{Y}_0)}E_{rel}[\Phi^{\textbf{Y}}_t(\Gamma_-),\Gamma_-;\phi_{s_2,\delta}\psi]=V_{\Phi^{\textbf{Y}_0}_{t_0}(\Gamma_-)}[\phi_{s_2,\delta}\psi].
\end{equation}
Combining (\ref{54}) and (\ref{55}) proves (\ref{56}).

Now, we prove item (\ref{57}). Observe that if $\|\textbf{Z}-\textbf{Z}_0\|_{\mathcal{Y}^-}\to0$, then 
\begin{equation}\label{58}
    \overbar{\text{div}}(\textbf{Z}-\textbf{Z}_0)-Q_{(\textbf{Z}-\textbf{Z}_0)}-\frac{n}{y}(\textbf{Z}-\textbf{Z}_0)\cdot\overbar{\textbf{e}}_{n+1}\to0\text{ in }\mathfrak{Y}\text{-norm}.
\end{equation}
We write
\begin{equation*}
    \delta\Big((\Phi^{\textbf{Y}}_t)^+_{\#}V\Big)[\textbf{Z}]-\delta\Big((\Phi^{\textbf{Y}_0}_{t_0})^+_{\#}V_0\Big)[\textbf{Z}_0]=\delta\Big((\Phi^{\textbf{Y}}_t)^+_{\#}V\Big)[\textbf{Z}-\textbf{Z}_0]+\delta\Big(\Phi^{\textbf{Y}}_t)^+_{\#}V-(\Phi^{\textbf{Y}_0}_{t_0})^+_{\#}V_0\Big)[\textbf{Z}_0].
\end{equation*}
The first term on the right converges to $0$ by Lemma \ref{44} and (\ref{58}). The second term converges to $0$ by item (\ref{53}).
\end{proof}

\begin{prop}\label{70}
Fix $\Lambda>0$. Let $V\in\mathfrak{Y}^*_\mathcal{C}(\Lambda)$ be $E_{rel}$-minimizing to the first order and have the decomposition $V=\frac{1}{y^n}V_+-\mathcal{H}^n_\mathbb{H}\mres\Gamma_-$ as stated in Proposition \ref{6}. Then, $\mathrm{Spt}(V_+)$ lies in $\mathrm{cl}(U_{\Gamma_+})\backslash U_{\Gamma_-}$. Hence, $V_+$ is stationary. Moreover, $E_{rel}[V]=V[\mathbf{1}]$.
\end{prop}

\begin{proof}
The stationarity and Spt($V_+$)$\subset\text{cl}(U_{\Gamma_+})\backslash U_{\Gamma_-}$ follow from \cite{SolomonWhite} by using the mean-convex foliations constructed in Proposition \ref{20}. It remains to prove $E_{rel}[V]=V[\mathbf{1}]$. Proposition \ref{60} implies that $E_{rel}[V]\leq V[\mathbf{1}]$. To prove the opposite direction, we compute the first variation of $V$ with respect to $(1-\phi_{y_1,\delta})\textbf{N}_\epsilon\in\mathcal{Y}^-$ for $y_1<\epsilon$:
\begin{align}
    \delta V[(1-\phi_{y_1,\delta})\textbf{N}_\epsilon]=&V[\overbar{\text{div}}\Big((1-\phi_{y_1,\delta})\textbf{N}_\epsilon\Big)-\nabla_\textbf{v}\Big((1-\phi_{y_1,\delta})\textbf{N}_\epsilon\Big)\cdot\textbf{v}-\frac{n}{y}\Big((1-\phi_{y_1,\delta})\textbf{N}_\epsilon\Big)\cdot\overbar{\textbf{e}}_{n+1}]\nonumber\\
    =&V[y(1-\phi_{y_1,,\delta}(y))\Big(\overbar{\text{div}}\big(\overbar{\textbf{e}}^\top_{n+1}\circ\Pi(q)\big)-\nabla_\textbf{v}\overbar{\textbf{e}}^\top_{n+1}\circ\Pi(q)\cdot\textbf{v}\Big)]\label{115}\\
    &+V[\Big(\overbar{\textbf{e}}^\top_{n+1}\circ\Pi(q)\cdot\overbar{\textbf{e}}_{n+1}-\big(\overbar{\textbf{e}}^\top_{n+1}\circ\Pi(q)\cdot\textbf{v}\big)\big(\overbar{\textbf{e}}_{n+1}\cdot\textbf{v}\big)\Big)\frac{d}{dy}\Big(y(1-\phi_{y_1,,\delta}(y))\Big)]\label{116}\\
    &-nV[\big(1-\phi_{y_1,,\delta}(y)\big)(\overbar{\textbf{e}}^\top_{n+1}\circ\Pi(q)\cdot\overbar{\textbf{e}}_{n+1}-1)]-nV[1-\phi_{y_1,,\delta}(y)]\label{117}\geq0.
\end{align}
By letting $y_1\to 0^+$, one sees that
\begin{align*}
    &\left\|y(1-\phi_{y_1,,\delta}(y))\Big(\overbar{\text{div}}\big(\overbar{\textbf{e}}^\top_{n+1}\circ\Pi(q)\big)-\nabla_\textbf{v}\overbar{\textbf{e}}^\top_{n+1}\circ\Pi(q)\cdot\textbf{v}\Big)\right\|_\mathfrak{X}\to0,\\
    &\left\|\Big(\overbar{\textbf{e}}^\top_{n+1}\circ\Pi(q)\cdot\overbar{\textbf{e}}_{n+1}-\big(\overbar{\textbf{e}}^\top_{n+1}\circ\Pi(q)\cdot\textbf{v}\big)\big(\overbar{\textbf{e}}_{n+1}\cdot\textbf{v}\big)-1+(\overbar{\textbf{e}}_{n+1}\cdot\textbf{v})^2\Big)\frac{d}{dy}\Big(y(1-\phi_{y_1,,\delta}(y))\Big)\right\|_{\mathfrak{X}}\to0,\\
    &\left\|\big(1-\phi_{y_1,,\delta}(y)\big)(\overbar{\textbf{e}}^\top_{n+1}\circ\Pi(q)\cdot\overbar{\textbf{e}}_{n+1}-1)\right\|_\mathfrak{X}\to0\text{ and }V[1-\phi_{y_1,,\delta}(y)]\to V[1]-E_{rel}[V].
\end{align*}
As a result, (\ref{115}) and the first term in (\ref{117}) converges to $0$, and
\begin{equation*}
    (\ref{116})\to V[1-(\overbar{\textbf{e}}_{n+1}\cdot\textbf{v})^2]-E_{rel}[V;1-(\overbar{\textbf{e}}_{n+1}\cdot\textbf{v})^2];\text{ the second term of }(\ref{117})\to -nV[1]+nE_{rel}[V].
\end{equation*}
So, $\delta V[(1-\phi_{y_1,\delta})\textbf{N}_\epsilon]\geq0$ implies
\begin{align}\label{61}
    E_{rel}[V;n-1+(\overbar{\textbf{e}}_{n+1}\cdot\textbf{v})^2]\geq V[n-1+(\overbar{\textbf{e}}_{n+1}\cdot\textbf{v})^2].
\end{align}
Let $\overbar{\textbf{X}}=\overbar{\textbf{n}}_{\Gamma_-}\circ\Pi(q)$ and $\{\partial^*U_i\}_i$ be a sequence in $\mathcal{C}(\Gamma'_-,\Gamma'_+,\Lambda)$ with $E_{rel}[\partial^*U_i,\Gamma_-.\cdot]\to V[\cdot]$. Since $\overbar{X}=\overbar{\textbf{n}}_{\Gamma_-}$ on $\Gamma_-$, one has
\begin{align}
    &\int_{\partial^*U_i}\phi^2_{y_1,y_2,\delta}(\overbar{\textbf{e}}_{n+1}\cdot\textbf{X})\left\langle\overbar{\textbf{e}}_{n+1},\textbf{n}_{\partial^*U_i}\right\rangle d\mathcal{H}^n_{\mathbb{H}}- \int_{\Gamma}\phi^2_{y_1,y_2,\delta}(\overbar{\textbf{e}}_{n+1}\cdot\textbf{X})\left\langle\overbar{\textbf{e}}_{n+1},\textbf{n}_{\Gamma}\right\rangle d\mathcal{H}^n_{\mathbb{H}}\label{105}\\
    &\leq E_{rel}[\partial^*U_i,\Gamma_-;\phi^2_{y_1,y_2,\delta}\left(\frac{1}{2}(\overbar{\textbf{e}}_{n+1}\cdot\overbar{\textbf{X}})^2+\frac{1}{2}(\overbar{\textbf{e}}_{n+1}\cdot\overbar{\textbf{n}}_{\partial^*U_i})^2\right)],\label{108}
\end{align}
where $\left\langle\textbf{X},\textbf{Y}\right\rangle$ is the inner product in $\mathbb{H}^{n+1}$ and $\textbf{X}\cdot\textbf{Y}$ is the inner product in $\mathbb{R}^{n+1}_+$. Moreover,
\begin{align}
    &\int_{\partial^*U_i}\phi^2_{y_1,y_2,\delta}(\overbar{\textbf{e}}_{n+1}\cdot\textbf{X})\left\langle\overbar{\textbf{e}}_{n+1},\textbf{n}_{\partial^*U_i}\right\rangle d\mathcal{H}^n_{\mathbb{H}}- \int_{\Gamma_-}\phi^2_{y_1,y_2,\delta}(\overbar{\textbf{e}}_{n+1}\cdot\textbf{X})\left\langle\overbar{\textbf{e}}_{n+1},\textbf{n}_{\Gamma_-}\right\rangle d\mathcal{H}^n_{\mathbb{H}}\label{107}\\
    &\geq -2\int_{\Omega}\left|\text{div}\Big(\phi^2_{y_1,y_2,\delta}(\textbf{X}\cdot\overbar{\textbf{e}}_{n+1})\overbar{\textbf{e}}_{n+1}\Big)\right|d\mathcal{H}_{\mathbb{H}}\geq-C y_2,\label{118}
\end{align}
where in the last step, we used the same argument as in the proof of \cite[Theorem 3.1]{YaoRel}. Combining (\ref{105})-(\ref{118}), we see that
\begin{equation}\label{119}
    E_{rel}[\partial^*U_i,\Gamma_-;\phi^2_{y_1,y_2,\delta}\left(\frac{1}{2}(\overbar{\textbf{e}}_{n+1}\cdot\overbar{\textbf{X}})^2+\frac{1}{2}(\overbar{\textbf{e}}_{n+1}\cdot\overbar{\textbf{n}}_{\partial^*U_i})^2\right)]\geq -C y_2
\end{equation}
By (\ref{quant estimate}),
\begin{align}\label{106}
    \big|E_{rel}[\partial^*U_i,\Gamma_-;\phi^2_{y_1,y_2,\delta}(\overbar{\textbf{e}}_{n+1}\cdot\overbar{\textbf{X}})^2]\big|\leq \big(Cy_2+C|E_{rel}[\partial^*U_i,\Gamma_-;\phi_{y_1,y_2,\delta}]|\big)\|\phi_{y_1,y_2,\delta}(\overbar{\textbf{e}}_{n+1}\cdot\overbar{\textbf{X}})^2\|_\mathfrak{X}.
\end{align}
Since $\overbar{\textbf{e}}_{n+1}\cdot\overbar{\textbf{X}}\to0$ as $y_2\to0$, we have $\|\phi_{y_1,y_2,\delta}(\overbar{\textbf{e}}_{n+1}\cdot\overbar{\textbf{X}})^2\|_\mathfrak{X}\to0$. If we choose $y_2$ sufficiently small in (\ref{106}), then
\begin{equation}\label{109}
    \big|E_{rel}[\partial^*U_i,\Gamma_-;\phi^2_{y_1,y_2,\delta}(\overbar{\textbf{e}}_{n+1}\cdot\overbar{\textbf{X}})^2]\big|\leq \frac{1}{2}\big(y_2+|E_{rel}[\partial^*U_i,\Gamma_-;\phi_{y_1,y_2,\delta}]|\big).
\end{equation}
Plugging (\ref{109}) into (\ref{108}) and using (\ref{119}), we get
\begin{equation*}
    E_{rel}[\partial^*U_i,\Gamma_-;\phi^2_{y_1,y_2,\delta}(\overbar{\textbf{e}}_{n+1}\cdot\textbf{v})^2]\geq-Cy_2-\frac{1}{2}E_{rel}[\partial^*U_i,\Gamma_-;\phi_{y_1,y_2,\delta}].
\end{equation*}
Letting $\delta\to0$ and $y_1\to0$, we have $E_{rel}[\partial^*U_i,\Gamma_-;(1-\chi_{\{y\geq y_2\}})(\overbar{\textbf{e}}_{n+1}\cdot\textbf{v})^2]\geq-Cy_2-\frac{1}{2}E_{rel}[\partial^*U_i,\Gamma_-;1-\chi_{\{y\geq y_2\}}]$. This implies
\begin{align*}
    V[(1-\chi_{\{y\geq y_2\}})(\overbar{\textbf{e}}_{n+1}\cdot\textbf{v})^2]\geq -Cy_2-\frac{1}{2}\big|V[(1-\chi_{\{y\geq y_2\}})]\big|.
\end{align*}
Finally, we let $y_2\to0$ and get $ V[(\overbar{\textbf{e}}_{n+1}\cdot\textbf{v})^2]\geq E_{rel}[V;(\overbar{\textbf{e}}_{n+1}\cdot\textbf{v})^2]-\frac{1}{2}V[\mathbf{1}]+\frac{1}{2}E_{rel}[V]$. Combining this with (\ref{61}), one gets $E_{rel}[V]\geq V[\mathbf{1}]$.
\end{proof}

\section{Min-Max Theory}
We establish the min-max theory in hyperbolic space in this section. Let $\Gamma_-$, $\Gamma_+$ and $\Omega'$ be as stated in Theorem \ref{main1}. Let $\widetilde{\Omega}=\text{cl}(U_{\Gamma_+})\backslash U_{\Gamma_-}$.

\begin{defn}
A \textit{generalized smooth family of hypersurfaces} in $\text{cl}(\Omega')$ is a family of pairs, $\{(U_\tau,\Gamma_\tau)\}_{\tau\in[0,1]^k}$, where $U_\tau\in\mathcal{C}(\Gamma'_-,\Gamma'_+)$ and $\Gamma_\tau=\partial^*U_\tau$ satisfies the following:
\begin{enumerate}
    \item $E_{rel}[\Gamma_\tau,\Gamma_-]<\infty$;
    \item For each $\tau\in[0,1]$ there is a finite set $S_\tau\subset\text{cl}(\Omega')$, so that $\Gamma_\tau$ is a smooth hypersurface in $\text{cl}(\Omega')\backslash S_\tau$;
    \item The map $\tau\mapsto V_{\Gamma_\tau}$ is continuous in the weak$^*$-topology of $\mathfrak{Y}^*(\text{cl}(\Omega'))$;
    \item $\Gamma_\tau\rightarrow\Gamma_{\tau_0}$ in $C^\infty_{loc}(\mathbb{R}_+^{n+1}\backslash S_{\tau_0})$,as $\tau\to\tau_0$;
    \item The map $\tau\mapsto\mathbf{1}_{U_\tau}$ is continuous in $L^1_{loc}(\mathbb{R}^{n+1}_+)$.
\end{enumerate}
If $k=1$, $(U_0,\Gamma_0)=(U_{\Gamma_-},\Gamma_-)$ and $(U_1,\Gamma_1)=(U_{\Gamma_+},\Gamma_+)$, then we say that $\{(U_\tau,\Gamma_\tau)\}_{\tau\in[0,1]}$ is a \textit{sweep-out} of $\widetilde{\Omega}$. 
\end{defn}

We have the following result on the existence of a sweep-out. Notice that the argument of \cite[Section 2]{Milnor2015Cobordism} (see also \cite{DeLellisRamic} and \cite{BWMountainPass}) can be adapted in our setting. So, we omit the proof here.

\begin{prop}\label{97}
There is a sweep-out of $\widetilde{\Omega}$.
\end{prop}

Using a calibration argument, we prove in the following proposition that any sweep-out must pass through a hypersurface with strictly larger relative entropy. It can be carried out in the same manner as in \cite[Section 8]{BWMountainPass}. We will sketch a proof in Appendix \ref{111}.

\begin{prop}\label{65}
There exists a positive constant $\delta_0=\delta_0(\Gamma_-,\Gamma_+,\Omega')$, so that for any sweep-out $\{(U_\tau,\Gamma_\tau)\}_{\tau\in[0,1]}$, we have
\begin{equation*}
    \max_{\tau\in[0,1]}E_{rel}[\Gamma_\tau,\Gamma_-]\geq\max\{E_{rel}[\Gamma_+,\Gamma_-],0\}+\delta_0.
\end{equation*}
\end{prop}

\begin{defn}
Two sweep-outs $\{(U_\tau,\Gamma_\tau)\}_{\tau\in[0,1]}$ and $\{(U_\tau',\Gamma_\tau')\}_{\tau\in[0,1]}$ are \textit{homotopic} if there is a generalized smooth family $\{(W_{\tau,\rho},\Xi_{\tau,\rho})\}_{(\tau,\rho)\in[0,1]^2}$ so that
\begin{enumerate}
    \item $(W_{\tau,0},\Xi_{\tau,0})=(U_\tau,\Gamma_\tau)$ for any $\tau\in[0,1]$;
    \item $(W_{\tau,1},\Xi_{\tau,1})=(U'_\tau,\Gamma'_\tau)$ for any $\tau\in[0,1]$;
    \item $(W_{0,\rho},\Xi_{0,\rho})=(U_{\Gamma_-},\Gamma_-)$ for any $\rho\in[0,1]$;
    \item $(W_{1,\rho},\Xi_{1,\rho})=(U_{\Gamma_+},\Gamma_+)$ for any $\rho\in[0,1]$.
\end{enumerate}
Let $X$ be a set of sweep-outs. $X$ is called \textit{homotopically closed} if it contains the homotopy class of each of its elements.
\end{defn}

\begin{defn}
Let $X$ be homotopically closed. The \textit{min-max value} of $X$ is
\begin{align*}
    m_{rel}(X)=\inf\left\{\max_{\tau\in[0,1]} E_{rel}[V_{\Gamma_\tau}]:\ \{(U_\tau,\Gamma_\tau)\}_{\tau\in[0,1]}\in X\right\}.
\end{align*}
The \textit{boundary-max value} is defined as $bM_{rel}=\max\{E_{rel}[\Gamma_+,\Gamma_-],0\}$. A \textit{minimizing sequence} is a sequence of sweep-outs $\{\{(U_\tau^\ell,\Gamma_\tau^\ell)\}_{\tau\in[0,1]}\}_{l=1}^\infty\subset X$ such that
\begin{equation*}
    \underset{\ell\to\infty}{\lim}\ \underset{\tau\in[0,1]}{\max}E_{rel}[V_{\Gamma_\tau^\ell}]=m_{rel}(X).
\end{equation*}
A \textit{min-max sequence} is obtained from a minimizing sequence by taking slices $\{(U_{\tau_\ell}^\ell,\Gamma^\ell_{\tau_\ell})\}$ for $\tau_\ell\in[0,1]$, which satisfies
\begin{equation*}
    \underset{l\to\infty}{\lim}E_{rel}[V_{\Gamma_{\tau_\ell}^\ell}]=m_{rel}(X).
\end{equation*}
\end{defn}

Applying Proposition \ref{65}, the following is immediate:

\begin{prop}
For any homotopically closed set $X$ of sweep-outs, we have $m_{rel}(X)\geq bM_{rel}+\delta_0>bM_{rel}$.
\end{prop}

We now state the main theorem of this section, of which Theorem \ref{main1} is an easy corollary. Let $n$, $\alpha$, $k$, $M$ and $\Gamma_-\preceq\Gamma_+$ be as stated in Theorem \ref{main1}. By \cite[Lemma 3.4]{YaoRel}, $\Gamma_+$ is in a thin neighborhood of $\Gamma_-$ so the variational theory of relative entropy in the preceding sections can apply.

\begin{thm}\label{66}
Let $X$ be a homotopically closed set of sweep-outs. There is a minimizing sequence $\big\{\{(U^\ell_\tau,\Gamma^\ell_\tau)\}_{\tau\in[0,1]}\big\}_{\ell=1}^\infty$, so that there exists a min-max sequence satisfying $\Gamma^\ell_{\tau_\ell}\to \Gamma_0$ in the sense of varifolds, where $\Gamma_0=\partial^*U_0$ for $U_0\in\mathcal{C}(\Gamma_-,\Gamma_+)$ satisfies $E_{rel}[\Gamma_0,\Gamma_-]=m_{rel}(X)$. Moreover, $\Gamma_0$ is a minimal hypersurface except for a singular set of codimension-$7$ and is $C^{k,\alpha}$-asymptotic to $M$.
\end{thm}

The proof of Theorem \ref{66} is divided into several parts. The first ingredient we need is the \textit{pull-tight} argument, which proves the existence of a minimizing sequence, so that every min-max sequence of it has a limit in the weak sense.

Let $\Lambda_0=\max\{E_0,4|m_{rel}(X)|\}$, $\mathcal{V}=\mathfrak{Y}^*_\mathcal{C}(\Lambda_0)$ and
\begin{equation*}
    \mathcal{V}_s:=\{V\in\mathcal{V}: \text{V is }E_{rel}\text{-minimizing to the first order in }\text{cl}(\Omega')\}.
\end{equation*}
By Proposition \ref{36}, we know that $\mathcal{V}\subset B^{\mathfrak{Y}^*}_{\Lambda'}$ for $\Lambda'=\Lambda'(\Gamma_-,\Omega',\Lambda_0)$.

\begin{prop}\label{67}
Let $X$ be a homotopically closed set of sweep-outs with $m_{rel}(X)>bM_{rel}$. There exists a minimizing sequence $\{\{(U^\ell_\tau,\Gamma^\ell_\tau)\}_{\tau\in[0,1]}\}^\ell\subset X$ so that if $\{(U_{\tau_\ell}^\ell,\Gamma^\ell_{\tau_\ell})\}^\ell$ is a min-max sequence, then $\mathcal{D}(V_{\Gamma^\ell_{\tau_\ell}},\mathcal{V}_s)\rightarrow0$. Here, $\mathcal{D}$ is the metric for the weak*-topology on $\mathfrak{Y}^*$.
\end{prop}

Proposition \ref{67} is similar to the setting in \cite[Proposition 7.7]{BWMountainPass}. In fact, the proof there can be adapted to our case once we replace the continuous dependence results that appear in Section 5 and 6 in \cite{BWMountainPass} by Proposition \ref{52} and Proposition \ref{68}. So, we do not include a proof here. The pull-tight argument is also carried out in \cite{DeLellisRamic}, \cite{ColdingdeLellisMinmax}, where they worked in compact ambient manifolds with the area functional instead of the relative entropy.

The second ingredient is the \textit{almost minimizing} property introduced in \cite{PittsExiReg}.

\begin{defn}
Fix a constant $\epsilon>0$ and an open set $W\subset\mathbb{R}_+^{n+1}$. A pair $(U,\partial^*U)$ with $U\in\mathcal{C}(\Gamma'_-,\Gamma'_+)$ is called $\epsilon$-\textit{almost minimizing} in $W$, if there is \textbf{no} generalized smooth family $\{(U_\tau,\Gamma_\tau)\}_{\tau\in[0,1]}$ with
\begin{enumerate}
    \item $(U_0,\Gamma_0)=(U,\partial^*U)$;
    \item There is an open subset $W'$ with $\text{cl}(W')\subset\subset W$, so that $U_\tau\backslash W'=U\backslash W'$ for any $\tau\in[0,1]$;
    \item $E_{rel}[\Gamma_\tau,\Gamma_-]\leq E_{rel}[\partial^*U,\Gamma_-]+\frac{\epsilon}{8}$ for any $\tau\in[0,1]$;
    \item $E_{rel}[\Gamma_1,\Gamma_-]\leq E_{rel}[\partial^*U,\Gamma_-]-\epsilon$.
\end{enumerate}
The pair $(U,\partial^*U)$ is said to be $\epsilon$-almost minimizing in a pair of open sets $(W^1,W^2)$ if it is $\epsilon$-almost minimizing in at least one of $W^1$ and $W^2$. A sequence $\{(U^i,\partial^*U^i)\}_i$ is called \textit{almost minimizing} (or \textit{a.m.} for short) in $W$ if each $(U^i,\partial^*U^i)$ is $\epsilon_i$-almost minimizing for a sequence of $\epsilon_i\rightarrow0$.
\end{defn}

Denote by $\mathcal{CO}$ the set of pairs $(W^1,W^2)$, where $W^1$, $W^2$ are open subsets of $\mathbb{R}_+^{n+1}$ with $\text{dist}_{\mathbb{R}^{n+1}}(W^1,W^2)\geq4\min\left\{\text{diam}_{\mathbb{R}^{n+1}}(W^1),\text{diam}_{\mathbb{R}^{n+1}}(W^2)\right\}$. We now state the Almgren-Pitts combinatorial lemma that appears in \cite[Section 3]{deLellisTasnady} (it is a variant of the original one by Almgrens and Pitts \cite{PittsExiReg}). The proof is identical to that of \cite[Section 3]{deLellisTasnady}.

\begin{prop}\label{69}
Let $X$ be a homotopically closed set of sweep-outs of $\widetilde{\Omega}$ with $m_{rel}(X)>bM_{rel}$. There exists an element $V_0\in\mathfrak{Y}^*_\mathcal{C}(\Lambda_0)$ that is $E_{rel}$-minimizing to the first order in $\mathrm{cl}(\Omega')$, and a min-max sequence $\{(U^\ell_{\tau_\ell},\Gamma^\ell_{\tau_\ell})\}_\ell$ so that
\begin{enumerate}
    \item $V_{\Gamma^\ell_{\tau_\ell}}$ converges in the weak*-topology to $V_0$ as $\ell\to\infty$;
    \item For any $(W^1,W^2)\in\mathcal{CO}$, $(U^\ell_{\tau_\ell},\Gamma^\ell_{\tau_\ell})$ is $\frac{1}{\ell}$-almost minimizing in $(W^1,W^2)$ for $\ell$ sufficiently large.
\end{enumerate}
\end{prop}

The Almgren-Pitts lemma is used to show the existence of a min-max sequence that is almost minimizing in some annuli. Our observation is that the scales of such annuli can be made sufficiently large. This is crucial when proving the boundary regularity. More precisely,

\begin{prop}\label{73}
Let $\mathcal{AN}_\rho(p)$ be the set of all open annuli centered at $p\in\mathbb{R}_+^{n+1}$ with outer radius (measured in Euclidean topology) less than $\rho$. There are a min-max sequence $\{(U^\ell_{\tau_\ell},\Gamma^\ell_{\tau_\ell})\}_{\ell=1}^\infty$ and a function $\varrho:\ \text{cl}(\Omega')\to(0,\infty)$ satisfying
\begin{enumerate}
    \item $\{(U^\ell_{\tau_\ell},\Gamma^\ell_{\tau_\ell})\}_\ell$ is a.m. in every $An\in\mathcal{AN}_{\varrho(p)}(p)$ for all $p\in\text{cl}(\Omega')$;
    \item $V_{\Gamma^\ell_{\tau_\ell}}$ converges in the weak*-topology to $V_0$ as $\ell\to\infty$;
    \item $\varrho$ satisfies $\varrho(p)\geq\delta p_{n+1}$ for $\delta>0$ and all $p\in\text{cl}(\Omega')$.
\end{enumerate}
\end{prop}

\begin{proof}
Let $\{(U^\ell_{\tau_\ell},\Gamma^\ell_{\tau_\ell})\}_\ell$ and $V_0$ be the same as in Proposition \ref{69}. We show that a subsequence of it satisfies the desired properties. Write $\Gamma^\ell=\Gamma^\ell_{\tau_\ell}$ for convenience. Let $r(p)=p_{n+1}$ be the ($n+1$)-th coordinate of $\textbf{x}(p)$. Let $r_i(p)=\frac{r(p)}{i}$, $W^1_i(p)=B_{\frac{r_i(p)}{16}}(p)$ and $W^2_i(p)=\mathbb{R}^{n+1}_+\backslash \overbar{B_{\frac{7r_i(p)}{16}}(p)}$. Then $\left(W^1_i(p),W^2_i(p)\right)\in\mathcal{CO}$ for all $i$ and all $p\in\text{cl}(\Omega')$. Proposition \ref{69} implies, for any fix $i$, either of the following holds:
\begin{enumerate}
    \item $\Gamma^\ell$ is $\frac{1}{\ell}$-a.m. in $W^1_i(p)$, for any $p\in\text{cl}(\Omega')$ and $\ell$ sufficiently large.
    \item For any $K$, there is a $\ell_K\geq K$ and a $p^i_K\in\text{cl}(\Omega')$, so that $\Gamma^{\ell_K}$ is $\frac{1}{\ell_K}$-almost minimizing in $W^2_i(p)$.
\end{enumerate}
If there is no $i$ for which case(1) holds, then case(2) implies there is a sequence $\{p^i\}\subset\text{cl}(\Omega')$ such that $\Gamma^{\ell_i}$ is $\frac{1}{\ell_i}$-a.m. in $W^2_i(p)$. If $p^i$ converges to some point in $\mathbb{R}^n\times\{0\}$, then $r_i(p^i)\leq r(p^i)$ which converges to $0$. So, we can set $\varrho(p)=\frac{r(p)}{2}$ in this case. If $p^i$ converges to some point $p_0\in\text{cl}(\Omega')$, then $r_i(p^i)=\frac{r(p^i)}{i}$ also converges to $0$. So, we can choose $\varrho(p)=\frac{1}{2}\min\{|p-p_0|, r(p)\}$ for $\textup{p}\neq p_0$ and $\varrho(p_0)=\frac{r(p_0)}{2}$. It is easy to see that $\varrho$ satisfies $\varrho(p)\geq\delta p_{n+1}$ for a sufficiently small $\delta>0$.

Now, assume there is some $i_0$ for which case(1) holds. We simply take $\varrho(p)=\frac{r_{i_0}}{16}(p)$.
\end{proof}

Finally, we prove the following proposition which concludes the proof of Theorem \ref{66}.

\begin{prop}\label{23}
Let $\{(U^\ell_{\tau_\ell},\Gamma^\ell_{\tau_\ell})\}$ and $V_0$ be obtained in Proposition \ref{69}. Assume $V_0$ has the decomposition $V_0=\frac{1}{y^n}V_+-\mathcal{H}^n_{\mathbb{H}^{n+1}}\mres \Gamma_-$ as in Proposition \ref{6}. Then, $V_+=\mathcal{H}^n\mres\Gamma_0$, where $\Gamma_0$ is a minimal hypersurface that is $C^{k,\alpha}$-asymptotic to some components of $M$ and has codimension-$7$ singular set. Moreover, $E_{rel}[\Gamma_0,\Gamma_-]=m_{rel}(X)$.
\end{prop}

\begin{proof}
It follows from Proposition \ref{70} that $V_+$ is stationary and that Spt$V_+\subset\text{cl}(U_{\Gamma_+}\backslash U_{\Gamma_-})$. Since the argument in \cite{deLellisTasnady} is local, we obtain that $V_+=\sum_{j=1}^N c_j\mathcal{H}^n\mres\Gamma^j$, where $c_j\in\mathbb{N}$ and $\Gamma^j$ is a smooth minimal hypersurface with $\text{cl}(\Gamma^j)\backslash\Gamma^j$ is a set of codimension-$7$. Note that $V_+$ and $\Gamma^j$ are temporarily only defined in the interior of $\mathbb{R}^{n+1}_+$.

For $(\textbf{x},0)\in\mathbb{R}^n\times\{0\}$, let $d(\textbf{x}):=d_{\mathbb{R}^n\times\{0\}}(\textbf{x},M)$. By \cite[Section 1]{HardtLinHyperbolicRegularity}, as $(\textbf{x},y)\in\text{Spt}V_+$ approaches the ideal boundary, $\frac{d(\textbf{x})}{y}\to0$.

We claim there is an $\epsilon>0$, so that $V_+\cap\{y\leq\epsilon\}$ has no singularities. Assume, on the contrary, there is a sequence of points $\{(\textbf{x}_i,y_i)\}$ lying in the singular set of $V_+$ with $y_i\to0$. Notice that $\partial_\infty(\text{Spt}(V_+))$ (which is defined to be the topological boundary of $\text{Spt}(V_+)$ in $\overbar{\mathbb{R}^{n+1}_+}$) is contained in $M$. So, without loss of generality, we assume that $\textbf{x}_i\to\textbf{a}\in M$ as $i\to\infty$. Assume $\textbf{a}=\textbf{0}\in\mathbb{R}^n$ and $\textbf{n}_{M}(0)=(\textbf{0},1)$. From Section 4 and 5 of \cite{deLellisTasnady}, one can see that for any $p\in\text{cl}(\Omega')$, $\text{Spt}\,V_+$ is stable in any annulus in $\mathcal{AN}_{\varrho(p)}(p)$. Hence, there is an $\epsilon>0$, so that $\text{Spt}\,V_+$ is stable in $B^{\mathbb{R}^{n+1}}_{\epsilon p_{n+1}}(p)$ for any point $p$ near the ideal boundary. So, $V_+$ is stable in $B_{\epsilon y_i}(\textbf{x}_i,y_i)$. Let $(\textbf{x}_i,y_i)=(\textbf{x}'_i,x^{n}_i,y_i)$ and $(\textbf{z}',z^{n},y)\in B_{\epsilon y_i}(\textbf{x}_i,y_i)\cap\text{Spt}(V_+)$. We suppose $\textbf{x}_i=\textbf{a}_i\pm d(\textbf{x}_i)\textbf{n}_{M}(\textbf{a}_i)=\textbf{a}_i\pm d_i\nu_i$ and $\textbf{z}:=(\textbf{z}',z^n)=\textbf{a}_i'\pm d_i'\nu_i'$. Then $\textbf{a}_i\to0$, $\textbf{n}_{M}(\textbf{a}_i)\to(\textbf{0},1)$, and for arbitrary $\delta>0$,
\begin{equation*}
    |(\textbf{a}'_i-\textbf{a}_i)\cdot\overbar{\textbf{e}}_n|\leq\delta|\textbf{a}'_i-\textbf{a}_i|\leq C\delta|\textbf{z}-\textbf{x}_i|\leq C\delta\epsilon y_i
\end{equation*}
holds for $i$ large. This gives
\begin{equation*}
    \frac{|(z^n-x^n_i)|}{y_i}\leq\frac{|a'^n_i-a^n_i|}{y_i}+\frac{d'_i}{y_i}+\frac{d_i}{y_i}\leq C\epsilon\delta+\frac{d'_i}{y_i}+\frac{d_i}{y_i}.
\end{equation*}
Since $\frac{d_i}{y_i}\to0$ and $|y-y_i|\leq\epsilon y_i\leq\frac{1}{2}y_i$, we have $\frac{d'_i}{y_i}\leq\frac{2d'_i}{y}\to0$. Hence,
\begin{equation}\label{71}
   \sup_{(\textbf{x},y)\in\Gamma_0\cap B_{\epsilon y_i}(\textbf{x}_i,y_i)}\frac{|x^n-x^n_i|}{y_i}\to0,\ \text{as }i\to\infty.
\end{equation}
Set $V_i=T_{i\#}\big(V_+\cap B_{\epsilon y_i}(\textbf{x}_i,y_i)\big)$, where $T_{i\#}$ is the push-forward of the map $T_i:\ (\textbf{x},y)\mapsto\frac{1}{y_i}(\textbf{x}-\textbf{x}_i,y)$. Then, it is easy to check that $V_i$ is stable in $B_\epsilon(\textbf{0},1)$ with a codimension-$7$ singular set. (\ref{71}) implies for any $\delta>0$, Spt$\,V_i\subset B_\epsilon(\textbf{0},1)\cap \{|x^n|<\delta\}$ holds for $i$ large. Moreover,
\begin{equation}\label{72}
    \frac{V_i(B_\epsilon(\textbf{0},1))}{\epsilon^n}=\frac{V_+(B_{\epsilon y_i}(\textbf{x}_i,y_i))}{\epsilon^ny_i^n}\leq C_\epsilon|V_0(\phi)|+\frac{\mathcal{H}^n(\Gamma_-\cap B_{\epsilon y_i}(\textbf{x}_i,y_i))}{\epsilon^ny_i^n},
\end{equation}
where $\phi$ is identically $1$ on $B_{\epsilon y_i}(\textbf{x}_i,y_i)$,and $\|\phi\|_{\mathfrak{X}}\leq C$. Notice that $\Gamma_-$ is $C^{k,\alpha}$ near the boundary, so it can be extended in a $C^2$ way across $\mathbb{R}^n\times\{0\}$. Thus, $\frac{1}{\epsilon^ny_i^n}\mathcal{H}^n(\Gamma_-\cap B_{\epsilon y_i}(\textbf{x}_i,y_i))$ is bounded above by the monotonicity formula \cite[Section 17]{SimonBook}. Hence (\ref{72}) is bounded above independent of $i$. We apply \cite[Theorem 1]{SchoenSimonRegularity} and get, for $i$ large, $V_i\cap B_{\frac{\epsilon}{2}}(\textbf{0},1)$ is a disjoint union of graphs of functions. In particular, $V_i$ is regular at $(\textbf{0},1)$, which contradicts to the assumption that $(\textbf{x}_i,y_i)$ is singular. This proves the claim.

By compactness of Caccioppoli sets, we can assume that $U^\ell_{\tau_\ell}$ converges in $L^1_{loc}$ to a Caccioppoli set $U_0$ with $U_{\Gamma_-}\subset U_0\subset \text{cl}(U_{\Gamma_+})$. By the nature of convergence, $\mathcal{H}^n\mres\partial^*U_0\leq V_+$. So, $\text{cl}(\partial^*U_0)\subset\text{Spt}(V_+)$. From the preceding paragraph, $\text{cl}(\partial^*U_0)\cap\{0<y\leq\epsilon\}$ is smooth. If there is a point $p\in\text{Spt}V_+\cap\{0<y\leq\frac{\epsilon}{2}\}\backslash\text{cl}(\partial^*U_0)$, we take $B_r(p)\subset(\text{cl}(\partial^*U_0))^c$. For any function $\phi\in C_c(B_r(p);\mathbb{R}^{n+1})$, $\phi(p)\cdot\textbf{v}$ can be thought of as a function over $G_n(B_r(p))$. So, by the varifold convergence, $\int \phi\cdot\textbf{v}dV_+=\lim_{\ell\to\infty}\int_{\Gamma^\ell_{\tau_\ell}}\phi\cdot\overbar{n}_{\Gamma^\ell_{\tau_\ell}}d\mathcal{H}^n$. On the other hand, the convergence of Caccioppoli sets implies $0=\int_{\partial^*U_0}\phi\cdot\overbar{n}_{\partial^*U_0}d\mathcal{H}^n=\lim_{\ell\to\infty}\int_{\Gamma^\ell_{\tau_\ell}}\phi\cdot\overbar{n}_{\Gamma^\ell_{\tau_\ell}}d\mathcal{H}^n$. So, $\int \phi\cdot\textbf{v}dV_+=0$ for any $\phi\in C_c(B_r(p);\mathbb{R}^{n+1})$. However, since $V_+\cap\{0<y\leq\epsilon\}$ is smooth, we get a contradiction if we take $\phi=\eta\overbar{n}_{V_+}$, where $\eta$ is a cut-off function. Thus, we have proved $\text{Spt}V_+$ coincides with $\text{cl}(\partial^*U_0)$ in $\{0<y\leq\epsilon\}$. This implies $M^j:=\partial_\infty\Gamma^j$ (the topological boundary of $\text{cl}(\Gamma^j)$ in $\overbar{\mathbb{R}^{n+1}_+}$) is equal to some component of $M$.

Now, we deal with the boundary regularity. We first modify the argument in \cite{HardtLinHyperbolicRegularity} to show that each $\Gamma^j$ is $C^1$ up to the ideal boundary. Assume $\textbf{0}\in M^j\subset\mathbb{R}^n$ and $\textbf{n}_{M^j}(0)=(\textbf{0},1)$. We claim that
\begin{equation*}
    \lim_{\Gamma^j\ni(\textbf{x},y)\to(\textbf{0},0)}\overbar{\textbf{n}}_{\Gamma^j}(\textbf{x},y)=(\textbf{0},1,0).
\end{equation*}
Let $\{(\textbf{x}_i,y_i)\}\subset\Gamma^j$ be a sequence of points converging to $(\textbf{0},0)$. By (\ref{71}), (\ref{72}) and \cite[Theorem 1]{SchoenSimonRegularity}, for $i$ large, there is a function $x^n=u_i(\textbf{x}',y)$ whose graph is equal to the component of $\Gamma^j$ containing $(\textbf{x}_i,y_i)$. \cite[Theorem 1]{SchoenSimonRegularity} also indicates $|Du_i|+\epsilon y_i|D^2u_i|\leq C$ in $B_{\frac{\epsilon y_i}{2}}(\textbf{x}_i,y_i)$. Since $\text{Spt}(V_+)$ is contained in a thin neighborhood of $\Gamma_-$, we have $Du_i(\textbf{x}_i,y_i)\to\textbf{0}$, otherwise, the graph of $u_i$ would intersect $\Gamma'_-$ or $\Gamma'_+$. This proves the claim and hence the $C^1$ regularity of $\Gamma^j$ at $M^j$. The higher regularity follows from \cite{LinHyperbolicMinimalGraph} and \cite{TonegawaCMCinHpSpace}. (Notice that if $k>n$, since we have assumed the $C^{k,\alpha}$-asymptotic regularity of $\Gamma_\pm$, the assumption on $M$ in \cite[Theorem 5.1]{TonegawaCMCinHpSpace} is automatically satisfied.)

Since $E_{rel}[\Gamma^\ell_{\tau_\ell}]\to m_{rel}(X)$ and $V_{\Gamma^\ell_{\tau_\ell}}\to V_0$, it is immediate that $E_{rel}[V_0]=V_0[\mathbf{1}]=m_{rel}(X)$. In particular, $E_{rel}[V_0]<\infty$. By the definition of relative entropy, $E_{rel}[V_0]=\lim_{\epsilon\to0}E_{rel}[V_0;\chi_{\{y\geq\epsilon\}}]=E_{rel}[\sum_{j=1}^N c_j\Gamma^j,\Gamma_-]$. Since $E_{rel}[V_0]<\infty$, we see that $V_+$ is of multiplicity one. Now, write $V_+=\mathcal{H}^n\mres\Gamma_0$. Then, $\Gamma_0$ is a minimal hypersurface that is $C^{k,\alpha}$-asymptotic to $M$ and has codimension-7 singular set. Moreover, $E_{rel}[\Gamma_0,\Gamma_-]=m_{rel}(X)$.
\end{proof}

\section{Properties for asymptotic hypersurfaces in hyperbolic space}
In this section, we discuss some useful properties that will be used later on.

\begin{lem}\label{120}
Let $\Sigma$ be a hypersurface in $\mathbb{R}^{n+1}_+$ that has $C^2$-asymptotic boundary in $\mathbb{R}^n\times\{0\}$ and meets $\mathbb{R}^n\times\{0\}$ orthogonally. Then, $\lambda_\mathbb{H}[\Sigma]<\infty$.
\end{lem}

\begin{proof}
Let $B^+_R(p)=B_R(p)\cap\mathbb{R}^{n+1}_+$ and $B^+_R=B_R(0)\cap\mathbb{R}^{n+1}_+$. We choose $R$ sufficiently large, so that $\Sigma\subset B^+_{\frac{R}{2}}$. Denote $B^+_R\cap\{y\geq\epsilon\}$ by $K_\epsilon$. We fix $p_0\in\mathbb{H}^{n+1}$. Since $\Sigma$ is $C^2$-asymptotic to $\partial_\infty\mathbb{H}^{n+1}$, we get $\mathcal{H}^n_\mathbb{H}(B^\mathbb{H}_\rho(p_0))\leq Ce^{C\rho}$ for $C>0$ and $\rho$ sufficiently large. Notice that $\mathcal{H}^n_\mathbb{H}(B^\mathbb{H}_\rho(p_1))\leq\mathcal{H}^n_\mathbb{H}(B^\mathbb{H}_{\rho+d_\mathbb{H}(p_1,p_0)}(p_0))$. So, for each $\epsilon'>0$, there are constants $C',\rho_0>0$, so that $\mathcal{H}^n_\mathbb{H}(B^\mathbb{H}_\rho(p))\leq C'e^{C'\rho}$ for $p\in K_{\epsilon'}$ and $\rho\geq\rho_0$. We now claim that $\sup_{\tau>0,p_0\in K_{\epsilon'}}\int_\Sigma\Phi_n^{0,p_0}(-\tau,p)d\mathcal{H}^n_\mathbb{H}(p)<\infty$. By \cite[Theorem 3.1]{DaviesHpbolicHeatKB},
\begin{equation*}
    \Phi_n^{0,p_0}(-t,p)=K_n(t,d)\leq Ct^{-\frac{n}{2}}e^{-\frac{d^2}{4t}}e^{-\frac{1}{8}(n-1)^2t},
\end{equation*}
where $d=d_{\mathbb{H}}(p,p_0)$. Observe that
\begin{align}
    \int_{\Sigma\backslash\{d\leq\rho_0\}}\Phi_n^{0,p_0}(-t,p)d\mathcal{H}^n_\mathbb{H}(p)&\leq C\sum_{k=1}^\infty t^{-\frac{n}{2}}e^{-\frac{k^2\rho_0^2}{4t}}e^{-\frac{1}{8}(n-1)^2t}\mathcal{H}^n_{\mathbb{H}}(\Sigma\cap\{k\rho_0\leq d\leq(k+1)\rho_0\})\label{91}\\
    &\leq C\sum_{k=1}^\infty t^{-\frac{n}{2}}e^{-\frac{k^2\rho_0^2}{4t}}e^{-\frac{1}{8}(n-1)^2t}e^{C'k\rho_0}\nonumber.
\end{align}
For any $N>0$, (\ref{91}) is uniformly bounded for $t\in(0,N]$. For $t>N$, the boundedness of (\ref{91}) follows in the same manner as in \cite[Proposition 4.2]{BernsteinEntropy} if we choose $\rho_0$ sufficiently large. Moreover, since the area ratio is locally bounded by the monotonicity formula \cite[Section 17]{SimonBook},
\begin{equation*}
    \sup_{\tau>0,p_0\in K_{\epsilon'}}\int_{\Sigma\cap B^\mathbb{H}_{\rho_0}(p_0)}\Phi_n^{0,p_0}(-\tau,p)d\mathcal{H}^n_\mathbb{H}(p)<\infty
\end{equation*}
follows in the same manner as in Euclidean case. This proves the claim.

Finally, we show that $\sup_{\tau>0,p_0\in K_{\epsilon'}^c}\int_\Sigma\Phi_n^{0,p_0}(-\tau,p)d\mathcal{H}^n_\mathbb{H}(p)<\infty$. We choose $\epsilon_1$ so small that for any $p_1\in\partial_\infty\Sigma$, $\Sigma\cap B^+_{2\epsilon_1}(p_1)$ can be written as a graph over a hyperplane parallel to $\overbar{\textbf{e}}_{n+1}$. Let $\epsilon'=\frac{\epsilon_1}{2}$ and $p_0\in K^c_{\epsilon'}$.

Fix $p_1\in\partial_\infty\Sigma$. Assume $p_1=\textbf{0}$ and $\textbf{n}_{\partial_\infty\Sigma'}(\textbf{0})=(0,\cdots,0,1,0)$ in $\mathbb{R}^n\times\{0\}$. Letting $\delta=\frac{\epsilon_1}{\sqrt{n}}$,
\begin{equation}
    \overbar{\Sigma'}\cap\{(-\delta,\delta)^{n}\times[0,\epsilon_1)\}=\{(\textbf{x}',u,y):\,u=u(\textbf{x}',y),(\textbf{x}',y)\in(-\delta,\delta)^{n-1}\times[0,\epsilon_1)\},
\end{equation}
where $u$ is a $C^2$ function with $|Du|\leq C$. Hence,
\begin{align}
    \int_{\Sigma'\cap(-\delta,\delta)^{n}\times(0,\epsilon_1)}\Phi_n^{0,p_0}(-\tau_0,p)d\mathcal{H}^n_\mathbb{H}(p)&=\int_{(-\delta,\delta)^{n-1}\times(0,\epsilon_1)}\Phi_n^{0,p_0}\big(-\tau_0,(\textbf{x}',u,y)\big)\frac{\sqrt{1+|Du|^2}}{y^n}d\textbf{x}'dy\nonumber\\
    &\leq C\int_{(-\delta,\delta)^{n-1}\times(0,\epsilon_1)}\Phi_n^{0,\Pi_1(p_0)}\big(-\tau_0,(\textbf{x}',0,y)\big)\frac{1}{y^n}d\textbf{x}'dy\nonumber\\
    &\leq C\int_{\mathbb{R}^{n-1}\times(0,\infty)}\Phi_n^{0,\Pi_1(p_0)}\big(-\tau_0,(\textbf{x}',0,y)\big)\frac{1}{y^n}d\textbf{x}'dy\nonumber\\
    &\leq C\lambda_\mathbb{H}[\mathbb{H}^n]\nonumber,
\end{align}
where $\Pi_1: (\textbf{x}',x_n,y)\mapsto(\textbf{x}',0,y)$ is the projection to the plane $\{x_n=0\}$. Since $\partial_\infty\Sigma$ is compact, we write it as a finite union of graphs. The integral over each portion is bounded. Thus, we have proved that 
\begin{equation}
    \sup_{\tau>0,p_0\in K_{\epsilon'}^c}\int_{\Sigma'\cap\{y\leq2\epsilon'\}}\Phi_n^{0,p_0}(-\tau_0,p)d\mathcal{H}^n_\mathbb{H}(p)<\infty.
\end{equation}
On the other hand, $\int_{\Sigma\cap\{y\geq2\epsilon'\}}\Phi_n^{0,p_0}(-\tau,p)d\mathcal{H}^n_\mathbb{H}(p)$ is uniformly bounded since $d_\mathbb{H}(K_{\epsilon'}^c,\Sigma\cap\{y\geq2\epsilon'\})>0$. This completes the proof.
\end{proof}

\begin{lem}\label{83}
Let $\Sigma$ and $\Sigma'$ be two hypersurfaces in $\mathbb{R}^{n+1}_+$ that have $C^2$-asymptotic boundaries in $\mathbb{R}^n\times\{0\}$ and both meet $\mathbb{R}^n\times\{0\}$ orthogonally. Assume that $\overbar{\Sigma'}=\{\mathbf{x}(p)+f(p)\mathbf{v}(p):p\in\overbar{\Sigma}\}$, where $\mathbf{v}$ is a $C^2$ transverse section\footnote{See \cite[Section 2.4]{BWSpace} for the precise definition.} on $\overbar{\Sigma}$ satisfying $\mathbf{v}(p)$ is perpendicular to $\mathbf{e}_{n+1}$ on $\partial_\infty\Sigma$. Then, for any $\epsilon>0$, there exists a $\delta>0$, so that $\lambda_\mathbb{H}[\Sigma']\leq\lambda_\mathbb{H}[\Sigma]+\epsilon$ whenever $\|f\|_{C^2(\overbar{\Sigma})}\leq\delta$.
\end{lem}

\begin{proof}
Let $F_\textbf{v}:\,p\mapsto\textbf{x}(p)+f(p)\textbf{v}(p)$ be the map from $\overbar{\Sigma}$ to $\overbar{\Sigma'}$. Then, there is $\delta_1>0$ depending on $\sup |\overbar{A}_{\Sigma}|$, so that $F_{\textbf{v}}$ is an injective map when $\|f\|_{C^2}\leq\delta_1$. By shrinking $\delta_1$ if necessary, we assume that $F_\textbf{v}$ is an immersion. Hence, there are a function $g(q)=f(F_\textbf{v}^{-1}(q))$ and a tranverse section $\textbf{w}(q)=\textbf{v}(F_\textbf{v}^{-1}(q))$ on $\overbar{\Sigma'}$, so that $\overbar{\Sigma}=\{\textbf{x}(q)+g(q)\textbf{w}(q):\,q\in\overbar{\Sigma'}\}$.

From Lemma \ref{120}, we know that both $\lambda_\mathbb{H}[\Sigma]$ and $\lambda_\mathbb{H}[\Sigma']$ are finite. For any $\epsilon>0$, we choose $\tau_0>0$ and $p_0\in\mathbb{H}^{n+1}$ so that
\begin{equation}\label{87}
    \int_{\Sigma'}\Phi_n^{0,p_0}(-\tau_0,p)d\mathcal{H}^n_\mathbb{H}(p)\geq\lambda_\mathbb{H}[\Sigma']-\frac{\epsilon}{2}.
\end{equation}
For any $\epsilon_1>0$ (to be determined later), we can choose $\delta$ sufficiently small, so that 
\begin{equation}\label{88}
    \int_{\Sigma'\cap\{y\geq\epsilon_1\}}\Phi_n^{0,p_0}(-\tau_0,p)d\mathcal{H}^n_\mathbb{H}(p)\leq \int_{\Sigma\cap\{y\geq\epsilon_1\}}\Phi_n^{0,p_0}(-\tau_0,p)d\mathcal{H}^n_\mathbb{H}(p)+\frac{\epsilon}{4}\leq\lambda_\mathbb{H}[\Sigma]+\frac{\epsilon}{4}.
\end{equation}
We show that for appropriately chosen $\epsilon_1$, $\int_{\Sigma'\cap\{y\leq\epsilon_1\}}\Phi_n^{0,p_0}(-\tau_0,p)d\mathcal{H}^n_\mathbb{H}(p)$ can be made smaller than $\frac{\epsilon}{4}$. Fix any $p_1\in\partial_\infty\Sigma'$. Assume $p_1=\textbf{0}$ and $\textbf{n}_{\partial_\infty\Sigma'}(\textbf{0})=(0,\cdots,0,1,0)$ in $\mathbb{R}^n\times\{0\}$. We choose $\epsilon_1<\frac{y(p_0)}{2}$ so small that for some $\delta_2>0$,
\begin{equation}\label{92}
    \overbar{\Sigma'}\cap\{(-\delta_2,\delta_2)^{n}\times[0,\epsilon_1)\}=\{(\textbf{x}',u,y):\,u=u(\textbf{x}',y),(\textbf{x}',y)\in(-\delta_2,\delta_2)^{n-1}\times[0,\epsilon_1)\},
\end{equation}
where $u$ is a $C^2$ function with $|Du|\leq1$. Hence,
\begin{align}
    \int_{\Sigma'\cap(-\delta_2,\delta_2)^{n}\times(0,\epsilon_1)}\Phi_n^{0,p_0}(-\tau_0,p)d\mathcal{H}^n_\mathbb{H}(p)&=\int_{(-\delta_2,\delta_2)^{n-1}\times(0,\epsilon_1)}\Phi_n^{0,p_0}\big(-\tau_0,(\textbf{x}',u,y)\big)\frac{\sqrt{1+|Du|^2}}{y^n}d\textbf{x}'dy\nonumber\\
    &\leq2\int_{(-\delta_2,\delta_2)^{n-1}\times(0,\epsilon_1)}\Phi_n^{0,\Pi_1(p_0)}\big(-\tau_0,(\textbf{x}',0,y)\big)\frac{1}{y^n}d\textbf{x}'dy\nonumber\\
    &\leq2\int_{\mathbb{R}^{n-1}\times(0,\epsilon_1)}\Phi_n^{0,\Pi_1(p_0)}\big(-\tau_0,(\textbf{x}',0,y)\big)\frac{1}{y^n}d\textbf{x}'dy\label{89},
\end{align}
where $\Pi_1: (\textbf{x}',x_n,y)\mapsto(\textbf{x}',0,y)$ is the projection to the plane $\{x_n=0\}$. However, (\ref{89}) can be thought of as an integral over an $n$-dimensional hyperbolic plane and is independent of choice of $p_1$. So, it can be made less than $c\epsilon$ for any $c>0$ if we choose $\epsilon_1$ sufficiently small. Since $\partial_\infty\Sigma'$ is compact, we write it as a finite union of graphs. The integral over each portion is less than $c\epsilon$. Thus, we have proved that 
\begin{equation}\label{90}
    \int_{\Sigma'\cap\{y\leq\epsilon_1\}}\Phi_n^{0,p_0}(-\tau_0,p)d\mathcal{H}^n_\mathbb{H}(p)\leq\frac{\epsilon}{4}.
\end{equation}
Combining (\ref{87}), (\ref{88}) and (\ref{90}), we get $\lambda_\mathbb{H}[\Sigma']\leq\lambda_\mathbb{H}[\Sigma]+\epsilon$.
\end{proof}

We now prove a compactness result for stable minimal hypersurfaces with bounded hyperbolic entropy.

\begin{prop}\label{22}
Assume $2\leq n\leq 6$, $\alpha\in(0,1)$ and $k\geq2$. Let $\{\Sigma_i\}_{i=1}^\infty$ be a sequence of stable minimal hypersurfaces in $\mathbb{H}^{n+1}$ with the same $C^{k,\alpha}$-asymptotic boundary $M=\partial_\infty\Sigma_i$. Then, there is a subsequence of $\{\Sigma_i\}_{i=1}^\infty$ converging in $C^\infty_{loc}(\mathbb{R}^{n+1}_+)$ to a stable minimal hypersurface $\Sigma_0$ in $\mathbb{H}^{n+1}$ that is also $C^{k,\alpha}$-asymptotic to $M$.
\end{prop}

\begin{proof}
From the knowledge of \cite{BryantConformal}, $\lambda_c[M]<\infty$. By \cite[Theorem 1.5]{BernsteinEntropy}, we have $\lambda_{\mathbb{H}}[\Sigma_i]=\lambda_c[M]/\text{Vol}(\mathbb{S}^{n-1})$ is uniformly bounded above. Notice that $\Phi^{0,p_0}(-1,p)$ is strictly positive. So, for any $y_1>0$, $\mathcal{H}^n_\mathbb{H}(\Sigma_i\cap\{y\geq y_1\})$ has a uniform upper bound independent of $i$. The standard compactness result \cite{SchoenSimonRegularity} implies the existence of a subsequence $\{\Sigma_{i_k}\}_k \subset\{\Sigma_i\}_i$ which converges to a stable minimal hypersurface $\Sigma_0$ (possibly with multiplicity) in $C^\infty_{loc}$ sense. Assume $\Sigma_0=\sum_{j=0}^Nc_j\Sigma^j$ with $N$, $c_j\in\mathbb{N}$ and each $\Sigma^j$ connected minimal.

It is easy to check that $\partial_\infty\Sigma^j$ (the topological boundary of $\Sigma^j$ in $\overbar{\mathbb{R}}^{n+1}_+$) is contained in $M$. Let $M^j$ be the component containing $\partial_\infty\Sigma^j$. Then, $\partial_\infty\Sigma^j$ is relatively closed in $M^j$. For any $p\in \partial_\infty\Sigma^j$, assume that $\textbf{x}(p)=\textbf{0}$ and $\textbf{n}_M(\textbf{0})=(0,\cdots,0,1,0)$. Since the argument in the proof of Proposition \ref{23} only relies on the geometry of ideal boundary $M$, there is a $\delta>0$, so that for any $k$,
\begin{equation*}
    \Sigma_{i_k}\cap(-\delta,\delta)^{n}\times[0,\delta)=\{(\textbf{x}',x_n,y):x_n=u_{i_k}(\textbf{x}',y),(\textbf{x}',y)\in(-\delta,\delta)^{n-1}\times[0,\delta)\}.
\end{equation*}
Since $\Sigma_{i_k}$ converges to $\Sigma^j$ in $C^\infty_{loc}$ sense, we see that $\partial_\infty\Sigma^j$ contains a neighborhood of $p$. This implies $\partial_\infty\Sigma^j$ is also relatively open in $M^j$. Hence, $\partial_\infty\Sigma^j=M^j$.

Next, we show that each $\Sigma^j$ is $C^{k,\alpha}$ up to $M$. For $(\textbf{x},0) \in \mathbb{R}^n \times\{0\}$, define $d(\textbf{x}):= d_{\mathbb{R}^n\times\{0\}}(\textbf{x},M)$. Then, for any $\epsilon<1$, the maximum principle \cite{SolomonWhite} implies $\Sigma_{i_k}$ and $\partial B_{(1-\epsilon)d(\textbf{x})}(\textbf{x},0)\cap\mathbb{R}^{n+1}_+$ are disjoint. Hence, $\Sigma^j$ and $\partial B_{(1-\epsilon)d(\textbf{x})}(\textbf{x},0)\cap\mathbb{R}^{n+1}_+$ are also disjoint. Since $\epsilon>0$ is arbitrary, we have $\Sigma^j\cap B_{d(\textbf{x})}(\textbf{x},0)=\emptyset$. Following the proof of Proposition \ref{23}, we see that (\ref{71}) also holds for $\Sigma^j$. To apply \cite[Theorem 1]{SchoenSimonRegularity}, it remains to show a uniform area ratio bound as in (\ref{72}). However, here we do not have any assumptions on the relative entropy. So, we instead use the uniform hyperbolic entropy bound. For any $p_0\in\Sigma^j$,
\begin{align*}
    \int_{\Sigma_{i_k}\cap B^{\mathbb{H}^{n+1}}_1(p_0)}K_n(1,d_{\mathbb{H}^{n+1}}(p-p_0))d\mathcal{H}^n_{\mathbb{H}}(p)\leq\int_{\Sigma_{i_k}}\Phi_n^{0,p_0}(-1,p)d\mathcal{H}^n_{\mathbb{H}}(p)\leq\lambda_c[M]/\text{vol}(\mathbb{S}^{n-1}),
\end{align*}
where $K_n(t,\rho)$ is defined in (\ref{74}). By direct computation, there is a constant $c_0>0$ independent of $p_0$ such that $B_{c_0^{-1}y(p_0)}(p_0)\subset B_1^{\mathbb{H}^{n+1}}(p_0)$. So, there is a constant $c_1>0$ depending only on $K_n(\tau,\rho)$, such that
\begin{equation*}
    c_1\frac{1}{y^n}\mathcal{H}^n(B_{c_0^{-1}y(p_0)}(p_0)\cap\Sigma_{i_k})\leq\int_{\Sigma_{i_k}\cap B^{\mathbb{H}^{n+1}}_1(p_0)}K(1,d_{\mathbb{H}^{n+1}}(p-p_0))d\mathcal{H}^n_{\mathbb{H}}(p).
\end{equation*}
By letting $i_k\to\infty$, we get $\frac{1}{y^n(p_0)}\mathcal{H}^n(B_{c_0^{-1}y(p_0)}(p_0)\cap\Sigma^j)$ has an upper bound independent of the choice of $p_0$. Now, we can proceed as in the proof Proposition \ref{23} to get the $C^{k,\alpha}$-asymptotic regularity of $\Sigma^j$.

Finally, by the nature of convergence, $\lambda_\mathbb{H}[\Sigma_0]\leq\liminf_{k\to\infty}\lambda_\mathbb{H}[\Sigma_{i_k}]=\lambda_c[M]/\text{vol}(\mathbb{S}^{n-1})$. On the other hand, by \cite[Theorem 1.5]{BernsteinEntropy}, $\lambda_\mathbb{H}[\Sigma_0]=\sum_{j=0}^{N}c_j\lambda_c[M^j]/\text{Vol}(\mathbb{S}^{n-1})$, where $M^j=\partial_\infty\Sigma^j$. This implies $\Sigma_0$ is of multiplicity one.
\end{proof}

We next adapt the argument in \cite[Section 4]{BWTopologicalUniqueness} to prove a partial ordering result.

\begin{prop}
Let $2\leq n\leq6$, $\alpha\in(0,1)$ and $k\geq1$. Fix a $C^{k,\alpha}$ closed submanifold $M\subset\mathbb{R}^n\times\{0\}$. For any minimal hypersurfaces $\Gamma_1$ and $\Gamma_2$ in $\mathbb{H}^{n+1}$ with the same $C^{k,\alpha}$-asymptotic boundary $M$, there exist two stable minimal hypersurfaces $\Gamma_-$, $\Gamma_+$ with $\partial_\infty\Gamma_\pm=M$, so that $\Gamma_-\preceq\Gamma_i\preceq\Gamma_+$ for $i=1,2$. Moreover, there exist stable minimal hypersurfaces $\Gamma_G$ and $\Gamma_L$ that are $C^{m,\alpha}$-asymptotic to $M$, where $m=k$ if $n$ is even and $m=\min\{n,k\}$ if $n$ is odd. For any minimal hypersurface $\Gamma$ with $\partial_\infty\Gamma=M$, we have $\Gamma_L\preceq\Gamma\preceq\Gamma_G$. $\Gamma_G$ and $\Gamma_L$ are called the greatest and the least minimal hypersurface with asymptotic boundary $M$.\label{11}
\end{prop}

\begin{rem}
As pointed out in \cite{TonegawaCMCinHpSpace}, in general we can not expect the minimal hypersurfaces to be as smooth as the ideal boundaries. However, if we are given a $C^{k,\alpha}$-asymptotic minimal hypersurface $\Gamma_1$ with $k>n$, the assumption on $\partial_\infty\Gamma_1$ stated in \cite[Theorem 5.1]{TonegawaCMCinHpSpace} is automatically satisfied. So, $\Gamma_\pm$ have $C^{k,\alpha}$-asymptotic regularity, while $\Gamma_G$ and $\Gamma_L$ in general only have at most $C^{n,\alpha}$ regularity if we do not assume the existence of a $C^{k,\alpha}$-asymptotic minimal hypersurface.
\end{rem}

\begin{proof}
Let $\Gamma_1$ and $\Gamma_2$ be two hypersurfaces as stated above. Let $M_\delta^\pm=\{p\pm f_\delta(p)\overbar{\textbf{n}}_{\Gamma_1}(p):\ p\in\Gamma_1\cap\{y=\delta\}\}$, where $f_\delta$ is chosen appropriately to satisfy $M^+_\delta\subset(\mathbb{R}^{n+1}_+\backslash \text{cl}(U_{\Gamma_1}))\cap(\mathbb{R}^{n+1}_+\backslash \text{cl}(U_{\Gamma_2}))$ and $M^-_\delta\subset U_{\Gamma_1}\cap U_{\Gamma_2}$. We further require that $M_\delta^\pm$ converges to $M$ as $\delta\to0$ and that $M_\delta^\pm\cap B^+_{d(\textbf{x})}(\textbf{x},0)=\emptyset$ for all $\textbf{x}\in\mathbb{R}^n$. Now, one can minimize the hyperbolic area functional in $\text{cl}(U_{\Gamma_1}\cap U_{\Gamma_2})\cap\{y\geq\delta\}$ with the fixed boundary $M^-_\delta$ and get an integral current $T_\delta$ with $\partial T_\delta=M^-_\delta$. Using $U_{\Gamma_1}$ and $U_{\Gamma_2}$ as barriers as in \cite[Section 6]{WhitePrescribedGenus}, one gets $T_\delta$ is $C^1$ away from a singular set with Hausdorff dimension less than $n-2$. It is then readily checked that the regular part of $T_\delta$ is stationary and that together with the singular part, it remains stationary. Now, the maximum principle \cite{SolomonWhite} implies the support of $T_\delta$ is away from $\Gamma_1$ and $\Gamma_2$. So, $T_\delta$ is locally minimizing and thus is a smooth hypersurface.

Taking a sequence $\delta_i\to0$, one gets a locally minimizing current $T$ as the limit of $T_{\delta_i}$ and $\Gamma_-:=\text{Spt}\,T\subset\text{cl}(U_{\Gamma_1})\cap\text{cl}(U_{\Gamma_2})$. Since $M^-_{\delta_i}\cap\big(\mathbb{R}^{n+1}_+\cap B_{d(\textbf{x})}(\textbf{x},0)\big)=\emptyset$ for any $\textbf{x}$, it is easy to check that $\Gamma_-\cap B_{d(\textbf{x})}(\textbf{x},0)=\emptyset$. Hence, applying \cite[Section 2,3]{HardtLinHyperbolicRegularity} and \cite{LinHyperbolicMinimalGraph}, one gets $\Gamma_-$ is smooth and $C^{k,\alpha}$ up to $M$. Similarly, with $U_{\Gamma_1}$ (resp. $U_{\Gamma_2}$) replaced by $\mathbb{R}^{n+1}_+\backslash \text{cl}(U_{\Gamma_1})$ (resp. $\mathbb{R}^{n+1}_+\backslash \text{cl}(U_{\Gamma_2})$), one gets another smooth and $C^{k,\alpha}$-asymptotic minimal hypersurface $\Gamma_+\subset(\mathbb{R}^{n+1}_+\backslash U_{\Gamma_1})\cap(\mathbb{R}^{n+1}_+\backslash U_{\Gamma_2})$.

Now, we prove the existence of $\Gamma_G$. Take $U_G=\cup_{\Gamma}U_{\Gamma}$, where the union is taken over all minimal hypersurfaces with asymptotic boundary $M$. By \cite[Theorem 3]{Anderson1982}, there exists at least one such hypersurface. Fix any point $p_1\in\partial U_G$. By the definition of $U_G$, there is a sequence of points $p^i_1$ with $p^i\in U_{\Gamma_i}$ for some minimal hypersurface $\Gamma_i$. Applying the preceding paragraph to $\Gamma_i$, we can, without loss of generality, assume each $\Gamma_i$ is stable. By Proposition \ref{22}, there is a subsequence, $\{\Gamma_i\}_i$, (without relabelling) that converges to a minimal hypersurface $\Gamma^1_G$. It is easy to check that $p_1\in\Gamma^1_G$. If $\Gamma^1_G$ is equal to $\partial U_G$, we choose $\Gamma_G=\Gamma^1_G$. If not, then there are a point $p_2\in\partial U_G\backslash\Gamma^1_G$ and a sequence of points $\widetilde{p}^i$, so that $\widetilde{p}^i\in\Gamma'_i$ for some minimal hypersurface $\Gamma'_i$ with $C^{m,\alpha}$-asymptotic boundary $M$. By the same reasoning for $\Gamma_i$, we assume $\Gamma'_i$ is stable and that $\Gamma'_i\succeq\Gamma^1_G$. Taking a convergent subsequence, one gets $\Gamma'_i$ converges to a minimal hypersurface $\Gamma^2_G$ with $\Gamma^2_G\succeq\Gamma^1_G$ and $p_2\in\Gamma^2_G$. Since $p_1\in\Gamma^1_G\cap\Gamma^2_G$, the maximum principle implies $\Gamma^1_G$ agrees with $\Gamma^2_G$ on the connected component containing $p_1$. Thus, $p_1$ and $p_2$ are in the different components of $\Gamma^2_G$. As $M$ has only finitely many components, we iterate the preceding argument and finally get $\Gamma^n_G=\partial U_G$ for some integer $n$. The proof of the existence of $\Gamma_L$ is similar.
\end{proof}

The next proposition is an asymptotic estimate for the distance between two minimal hypersurfaces that are asymptotic to the same ideal boundary.

\begin{prop}
Suppose $\Gamma_1\preceq\Gamma_2$ ($\Gamma_1\cap\Gamma_2=\emptyset$) are two minimal hypersurfaces that are both $C^3$-asymptotic to $M\subset\mathbb{R}^{n}\times\{0\}\subset\partial_\infty\mathbb{H}^{n+1}$. Then, there exist an $\epsilon_1>0$ and a function $h$, so that\label{10}
\begin{equation}\label{78}
    \Gamma_2\cap\{y\leq\frac{\epsilon_1}{2}\}\subset\big\{p+h(p)\overbar{\mathbf{n}}_{\Gamma_1}:p\in\Gamma_1\subset\{y\leq\epsilon_1\}\big\}\subset\Gamma_2\cap\{y\leq2\epsilon_1\}.
\end{equation}
On the components of $\Gamma_1$ where $h$ is not identically $0$, we have the estimate $\frac{1}{C}y^{n+1}(p)\leq h(p)\leq Cy^{n+1}(p)$ for some $C>0$ and $p\in\Gamma_1\cap\{y\leq\epsilon_1\}$.
\end{prop}

\begin{proof}
The existence of the function $h$ and the upper bound follows from \cite[Lemma 3.4]{YaoRel}. Moreover, we have $h\big|_M=0$, $\overbar{\nabla}h\big|_M=0$ and $\overbar{\nabla}^2h\big|_M=0$. By computation, $f=\frac{h}{y}$ satisfies the equation
\begin{equation*}
    \Delta f+(|A_{\Gamma_1}|^2-n)f=a(p,h,\overbar{\nabla}h,\overbar{\nabla}^2h)f+y\left\langle\textbf{b}(p,h,\overbar{\nabla}h,\overbar{\nabla}^2h),\nabla f\right\rangle,
\end{equation*}
where $|a|+|\textbf{b}|\leq C(|h|+|\overbar{\nabla}h|+|\overbar{\nabla}^2h|)$. Since $|A_{\Gamma_1}|^2=\mathcal{O}(y^2)$, we choose $\epsilon_1$ so small that $|A_{\Gamma_1}|^2-n-a\leq0$ for $y(p)\leq\epsilon_1$. Let $g(p)=y^{n}(p)+Ky^{n+1}(p)$. By the same reasoning as in the proof of \cite[Lemma 3.4]{YaoRel}, we can choose $\epsilon_1$ and $K$ appropriately, so that $\mathcal{L}g=\mathcal{L}(y^n+Ky^{n+1})\geq0$. On the component of $\Gamma_1$ where $f$ is not identically $0$, let $\lambda=\inf_{\{y(p)=\epsilon_1\}}f$. If $\lambda=0$, then the strong maximum principle implies that $f\equiv0$ on this component of $\Gamma_1$, which leads to a contradiction. Now, we have $\mathcal{L}(f-\frac{\lambda}{\epsilon_1^n+K\epsilon_1^{n+1}}g)\leq0$, $f\big|_{y=0}=0=\frac{\lambda}{\epsilon_1^n+K\epsilon_1^{n+1}}g\big|_{y=0}$ and $f\big|_{y=\epsilon_1}\geq \frac{\lambda}{\epsilon_1^n+K\epsilon_1^{n+1}}g\big|_{y=\epsilon_1}$. By the maximum principle, $f\geq\frac{\lambda}{\epsilon_1^n+K\epsilon_1^{n+1}}g$ in $\Gamma_1\cap\{y\leq\epsilon_1\}$. This proves the lower bound.
\end{proof}

\iffalse\begin{lem}\label{77}
Let $\Gamma_1$ be a hypersurface that is $C^3$-asymptotic to $M$. Let $\Gamma_2$ be a local graph over $\Gamma_1$ in a neighborhood $U$ of $p$ in $\Gamma_1$:
\begin{equation*}
    \Gamma_2=\{\textbf{x}(p)+h(p)\overbar{\textbf{n}}_{\Gamma_1}(p):\,p\in U\},
\end{equation*}
where $h\in C^2(U)$. Then, there are constants $\epsilon$ only depending on $\Gamma_1$, so that if $\|h\|_{C^2(U)}\leq\epsilon$, then
\begin{equation*}
    H_{\Gamma_2}=H_{\Gamma_1}-(\Delta_{\Gamma_1}+|A_{\Gamma_1}|^2-n)(\frac{h}{y})+Q(p,h,\overbar{\nabla}h,\overbar{\nabla}^2h),
\end{equation*}
where $Q(p,h,\overbar{\nabla}h,\overbar{\nabla}^2h)=ya(p,h,\overbar{\nabla}h,\overbar{\nabla}^2h)h+y^2\textbf{b}(p,h,\overbar{\nabla}h,\overbar{\nabla}^2h)\cdot\overbar{\nabla}h$ with $|a|+|\textbf{b}|\leq C(|h|+|\overbar{\nabla}h|+|\overbar{\nabla}^2h|)$ for $C=C(\Gamma_1)$.
\end{lem}

\begin{proof}
    By direct computation, $\overbar{\textbf{n}}_{\Gamma_2}-\overbar{\textbf{n}}_{\Gamma_1}=-\overbar{\nabla}h+\mathcal{O}(h^2+|\overbar{\nabla}h|^2)$
\end{proof}
\fi

\section{Deformation of Unstable Minimal Hypersurfaces}\label{85}
In this section, we show that an unstable minimal hypersurface can be deformed into a stable one. Following \cite{BWTopologicalUniqueness}, we perturb the unstable minimal hypersurface by its first eigenfunciton, and show that the mean curvature flow starting from the perturbed hypersurface moves to a stable minimal hypersurface. To get an isotopy that is regular up to the ideal boundary, we modify the deformation near the ideal boundary in an explicit way.

As a starting point, we prove that for unstable or weakly stable minimal hypersurfaces, the infimum of the energy functional for the stability operator is attained by some function in the weighted H\"older space $y^\mu\Lambda^{k,\alpha}_0$ for $\mu>n$.

\begin{prop}\label{14}
Let $\Gamma$ be a minimal hypersurface that is $C^2$-asymptotic to $M$. Assume that
\begin{equation*}
    \lambda_1=\inf_{f\in C^\infty_c(\Gamma)\backslash\{0\}}\frac{\int_{\Gamma}\left\langle\nabla f,\nabla f\right\rangle+(n-|A_\Gamma|^2)f^2 d\mathcal{H}^n_\mathbb{H}}{\int_{\Gamma}f^2d\mathcal{H}^n_\mathbb{H}}\leq0.
\end{equation*}
Then, there is a non-zero function $f\in y^\mu\Lambda^{k,\alpha}_0(\Gamma)$ satisfying $\Delta f+(|A_{\Gamma}|^2-n+\lambda_1)f=0$ for any $k\in\mathbb{N}$. Here $\mu=\frac{n-1+\sqrt{(n+1)^2-4\lambda_1}}{2}$ is the larger root to the equation $t^2+(1-n)t+(\lambda_1-n)=0$. Moreover, on the component of $\Gamma$ where $f$ is not identically $0$, we have $\frac{1}{C}y^{\mu}\leq f\leq Cy^\mu$.
\end{prop}

\begin{proof}
Let $\Gamma^i=\Gamma\cap\{y\geq\frac{1}{i}\}$ and
\begin{equation*}
    \lambda^i_1=\inf_{f\in H^1_0(\Gamma^i)\backslash\{0\}}\frac{\int_{\Gamma^i}\left\langle\nabla f,\nabla f\right\rangle+(n-|A_\Gamma|^2)f^2 d\mathcal{H}^n_\mathbb{H}}{\int_{\Gamma^i}f^2d\mathcal{H}^n_\mathbb{H}}.
\end{equation*}
It is easy to see that $\lambda^i_1\to\lambda_1$ as $i\to\infty$. Since $\Gamma^i$ is a compact manifold with boundary, the standard calculus of variation theory gives the existence of a non-negative function $f^i\in H^1_0(\Gamma^i)\cap C^\infty(\Gamma^i)$, so that $\Delta f^i+(|A_\Gamma|^2-n+\lambda^i_1)f^i=0$ in $\Gamma^i$ and $f^i(p_0)=1$ for a fix point $p_0\in\Gamma^1\subset\Gamma^i$. By choosing $\epsilon_2$ sufficiently small, we assume that $|A_\Gamma|^2-n<0$ for $p\in\Gamma\cap\{y(p)\leq\epsilon_2\}$. For $i$ sufficiently large, the Harnack inequality implies that $f^i(p)\leq C(\epsilon_2)f(p_0)=C(\epsilon_2)$ for $p\in\Gamma\cap\{y\geq\epsilon_2\}$. On the other hand, the maximum principle is valid on $\Gamma\cap\{y\leq\epsilon_2\}$. So, $\sup_{\Gamma\cap\{y\leq\epsilon_2\}}f^i\leq\sup_{\Gamma\cap\{y=\frac{1}{i},\,\epsilon_2\}}f^i=\sup_{\Gamma\cap\{y=\epsilon_2\}}f^i\leq C(\epsilon_2)$. In other words, $f^i$ is uniformly bounded in $i$. Locally, we can write the Laplace operator over $\Gamma$ as $\Delta=y^2\overbar{\Delta}+(2-n)y\overbar{g}^{kj}\partial_ky\partial_j$. So, we can rewrite the equation for $f^i$ as
\begin{equation}
    \overbar{\Delta}f^i+\frac{2-n}{y}\overbar{g}^{kj}\partial_ky\partial_j(f^i)+\frac{|A_\Gamma|^2-n+\lambda^i_1}{y^2}f^i=0.
\end{equation}
The interior Schauder estimate \cite[Section 6.1]{GilbargTrudinger} gives
\begin{equation*}
\frac{1}{i}\|\overbar{\nabla}f^i\|_{C^0(\Gamma\cap\{y\geq\frac{2}{i}\})}+\frac{1}{i^2}\|\overbar{\nabla}^2f^i\|_{C^0(\Gamma\cap\{y\geq\frac{2}{i}\})}+\frac{1}{i^{2+\alpha}}[\overbar{\nabla}^2f^i]_{\alpha;\Gamma\cap\{y\geq\frac{2}{i}\}}\leq C\|f^i\|_{C^0(\Gamma^i)}\leq C(\epsilon_2).
\end{equation*}
Taking a convergent subsequence of $\{f^i\}$, we get the existence of a $C^{2,\alpha}_{loc}$ function $f$ satisfying $\Delta f+(|A_\Gamma|^2-n+\lambda_1)f=0$, $f(p_0)=1$ and $\|f\|_{\Lambda^{2,\alpha}_0}\leq C(\epsilon_2)$.

Direct computation shows
\begin{align*}
    \Delta\left(\frac{f^i}{y}\right)-2\langle\nabla\left(\frac{f^i}{y}\right),y\nabla\frac{1}{y}\rangle+\left(|A_\Gamma|^2-n+\lambda_1^i+2-n+\mathcal{O}(y)\right)\frac{f^i}{y}=0.
\end{align*}
By shrinking $\epsilon_2$ if necessary, $\frac{f^i}{y}$ satisfies the maximum principle on $\Gamma\cap\{\frac{1}{i}\leq y\leq\epsilon_2\}$. Hence, $f^i\leq Cy$ in $\Gamma\cap\{\frac{1}{i}\leq y\leq\epsilon_2\}$. Taking $i\to\infty$, we see that $f$ can be continuously extended to $0$ on $M$. We take $\mathcal{L}_{\lambda_1}=\Delta+(|A_\Gamma|^2-n+\lambda_1)$ and notice that $\mathcal{L}_{\lambda_1}(y^\mu+K'y^{\mu+1})\geq0$ and $\mathcal{L}_{\lambda_1}(y^\mu-K'y^{\mu+1})\leq0$ for some $K'>0$. Then, the estimate $\frac{1}{C}y^\mu\leq f\leq Cy^\mu$ follows from an argument similar to the proof of Proposition \ref{10} and the fact that $f\in C_0(\Gamma\cup M)$.

Finally, $f\in y^\mu\Lambda^{k,\alpha}_0$ for any $k$ follows from $0\leq f\leq Cy^\mu$ and standard elliptic PDE estimates.
\end{proof}

The main theorem of this section is the following:

\begin{thm}\label{76}
Let $2\leq n\leq6$, $\alpha\in(0,1)$ and $k\geq2$. Assume $M$ is of class $C^{k+1,\alpha}$. If $\Gamma$ is an unstable minimal hypersurface that is $C^{k+1,\alpha}$-asymptotic to $M\subset\mathbb{R}^n\times\{0\}$ with $\lambda_c[M]<\mathrm{Vol}(\mathbb{S}^{n-1})\lambda[\mathbb{S}^{n-1}\times\mathbb{R}^1]$, then there is a stable minimal hypersurface $\Gamma'$ which is also $C^{k+1,\alpha}$-asymptotic to $M$, so that $\Gamma$ and $\Gamma'$ are $C^{k,\alpha}$ isotopic. Moreover, if there are two minimal hypersurfaces $\Gamma_1$ and $\Gamma_2$ with $\Gamma_1\preceq\Gamma\preceq\Gamma_2$, then $\Gamma'$ can be chosen to satisfy $\Gamma_1\preceq\Gamma'\preceq\Gamma_2$.
\end{thm}

By \cite{BernsteinEntropy}, $\lambda_c[M]=\text{Vol}(\mathbb{S}^{n-1})\lambda_\mathbb{H}[\Gamma]$. So, $\lambda_c[M]< \text{Vol}(\mathbb{S}^{n-1})\lambda[\mathbb{S}^{n-1}\times\mathbb{R}^1]$ implies $\lambda_{\mathbb{H}}[\Gamma]<\lambda[\mathbb{S}^{n-1}\times\mathbb{R}^1]$. Let $\hat{\Gamma}_\epsilon=\{p+\epsilon f(p)\textbf{n}_{\Gamma}=p+\epsilon yf\overbar{\textbf{n}}_{\Gamma}:\,p\in\Gamma\}$, where $f$ is the first eigenfunction of $\Gamma$. Consider the mean curvature flow $\{\Sigma_t\}_t$ with $\Sigma_0=\hat{\Gamma}_\epsilon$. Since $\sup_{p\in\hat{\Gamma}_\epsilon}|\overbar{A}_{\hat{\Gamma}_\epsilon}|(p)<\infty$ if we choose $\epsilon$ sufficiently small, short time existence of the flow follows from classical theories. Let $T>0$ be the maximal existing time of the flow $\{\Sigma_t\}_t$. We will show in the following lemma that by shrinking $\epsilon$ if necessary, $\{\Sigma_t\}$ remains mean convex.

\begin{lem}\label{79}
By choosing $\epsilon$ sufficiently small, the mean curvature flow $\{\Sigma_t\}_{t\in[0,T)}$ starting from $\hat{\Gamma}_\epsilon$ remains mean convex. Moreover, $H_{\Sigma_t}\geq ce^{-n(\mu+1)t}y^\mu$ for some $c>0$ and $t\in[0,T)$.
\end{lem}

\begin{proof}
By direct computation (see, e.g., \cite{HuiskenContractingCvxHypers}), $\frac{d}{dt}H_{\Sigma_t}=\Delta_{\Sigma_t}H_{\Sigma_t}+(|A_{\Sigma_t}|^2-n)H_{\Sigma_t}$. We also have $\frac{d}{dt}(y^\mu)=\Delta_{\Sigma_t}(y^\mu)+n\mu y^\mu-\mu(\mu+1)y^\mu(1-(\overbar{\textbf{n}}_{\Sigma_t}\cdot\overbar{\textbf{e}}_{n+1})^2)$. Hence, $\left(\frac{d}{dt}-\Delta_{\Sigma_t}\right)\left(e^{-n\mu t}y^\mu\right)\leq0$, and $\left(\frac{d}{dt}-\Delta_{\Sigma_t}\right)\left(e^{nt}H_{\Sigma_t}\right)=|A_{\Sigma_t}|^2H_{\Sigma_t}$. Combining these quantities, we get for any $c>0$,
\begin{align*}
    \left(\frac{d}{dt}-\Delta_{\Sigma_t}\right)\left(ce^{-n\mu t}y^\mu-e^{nt}H_{\Sigma_t}\right)\leq-|A_{\Sigma_t}|^2e^{nt}H_{\Sigma_t}\leq|A_{\Sigma_t}|^2\left(ce^{-n\mu t}y^\mu-e^{nt}H_{\Sigma_t}\right).
\end{align*}
At $t=0$, we know from direct computation that
\begin{align*}
    H_{\hat{\Gamma}_\epsilon}&=-(\Delta_\Gamma+|A_\Gamma|^2-n)(\epsilon f)+a(p,h,\overbar{\nabla}h,\overbar{\nabla}^2h)\epsilon f+y\left\langle\textbf{b}(p,h,\overbar{\nabla}h,\overbar{\nabla}^2h),\nabla \epsilon f\right\rangle\\
    &=\epsilon\big(\lambda_1 f+a(p,h,\overbar{\nabla}h,\overbar{\nabla}^2h)f+y\left\langle\textbf{b}(p,h,\overbar{\nabla}h,\overbar{\nabla}^2h),\nabla f\right\rangle\big)
\end{align*}
where $|a|+|\textbf{b}|\leq C(|h|+|\overbar{\nabla}h|+|\overbar{\nabla}^2h|)$ and $h=\frac{\epsilon f}{y}$. Hence, by Proposition \ref{14}, there is a small $\epsilon$, so that $\hat{\Gamma}_\epsilon$ is mean convex and $ce^{-n\mu t}y^\mu-e^{nt}H_{\Sigma_t}\leq0$ holds at $t=0$ for some $c>0$. On the other hand, since $\Gamma$ is a $C^3$ hypersurface up to the ideal boundary, the perturbed hypersurface $\hat{\Gamma}_\epsilon$ is $C^2$ up to the ideal boundary. Hence, $|A_{\hat{\Gamma}_\epsilon}|=\mathcal{O}(y)$. By the pseudo-locality theorem \cite[Theorem 7.5]{ChenPseudolcality}, we see that for any $T_0<T$, there are constants $C(T_0)>0$ and $\epsilon=\epsilon(T_0)>0$, so that
\begin{equation*}
    \sup_{0\leq t\leq T_0,p\in\Sigma_t\cap\{y\leq\epsilon\}}|A_{\Sigma_t}|(p)\leq C(T_0).
\end{equation*}
Since $\bigcup_{0\leq t\leq T_0}\big(\Sigma_t\cap\{y\geq\epsilon\}\big)\times\{t\}$ is compact, we see that $\sup_{0\leq t\leq T_0}|A_{\Sigma_t}|<\infty$.

Then, by a non-compact maximum principle (e.g., one can use a variant of the monotonicity formula introduced in \cite{BernsteinEntropy} and modify the proof of \cite[Corollary 1.1]{EHEntireGraph}), $ce^{-n\mu t}y^\mu-e^{nt}H_{\Sigma_t}\leq0$ for $t\in[0,T)$.
\end{proof}

\begin{lem}\label{18}
Let $\{\Sigma_t\}_{t}$ be a mean curvature flow which satisfies $H_{\Sigma_t}\geq ce^{-n(\mu+1)t}y^\mu$ for some $c>0$ and $\mu\in\mathbb{R}^1$. Suppose the inital hypersurface $\Sigma_0$ is $C^2$-asymptotic to the ideal boundary and meets the ideal boundary orthogonally. Then, for any $T_1<\infty$ with $T_1\leq T$(the maximal existing time), there is a constant $C_1=C_1(n,T_1,\Sigma_0)>0$, so that $|A_{\Sigma_t}|^2\leq C_1y^{-2\mu}e^{4nt}H_{\Sigma_t}^2$ for any $p\in\Sigma_t$ and $t<T_1$.
\end{lem}

\begin{proof}
When it is clear from the context, we will omit the subscripts and simply write $A$ (resp. $H$) for $A_{\Sigma_t}$ (resp. $H_{\Sigma_t}$). By \cite[Corollary 3.5]{HuiskenContractingCvxHypers}, 
\begin{equation*}
    \frac{d}{dt}|A|^2=\Delta|A|^2-2|\nabla A|^2+2|A|^2(|A|^2-n)+4n(|A|^2-\frac{1}{n}H^2).
\end{equation*}
After computation, we see that
\begin{align*}
    \frac{d}{dt}\left(\frac{|A|^2}{H^2}-\frac{1}{n}\right)=&\Delta\left(\frac{|A|^2}{H^2}-\frac{1}{n}\right)+4n\left(\frac{|A|^2}{H^2}-\frac{1}{n}\right)+\frac{4}{H^3}\left\langle\nabla(|A|^2),\nabla H\right\rangle\\
    &-\frac{6|A|^2}{H^4}|\nabla H|^4-\frac{2|\nabla A|^2}{H^2}.
\end{align*}
Following the computation in \cite[Proof of Proposition 3.3]{BWTopology},
\begin{align*}
    \frac{4}{H^3}\left\langle\nabla(|A|^2),\nabla H\right\rangle\leq\frac{2}{H}\left\langle\nabla\left(\frac{|A|^2}{H^2}\right),\nabla H\right\rangle+\frac{6|A|^2|\nabla H|^2}{H^4}+\frac{2|\nabla A|^2}{H^2}.
\end{align*}
Hence, the function $g=\frac{|A|^2}{H^2}e^{-4nt}-\frac{1}{n}e^{-4nt}$ satisfies $\frac{d}{dt}g\leq\Delta g+ \frac{2}{H}\left\langle\nabla g,\nabla H\right\rangle$. From the assumptions on $\Sigma_0$, we see that $|A_{\Sigma_0}|=\mathcal{O}(y)$ as $y=y(p)\to0$. Since $T_1<\infty$, the pseudo-locality theorem \cite[Theorem 7.5]{ChenPseudolcality} implies that there is an $\epsilon>0$, so that
\begin{equation}\label{82}
    \sup_{t<T_1,p\in\Sigma_t\cap\{y\leq2\epsilon\}}|A|\leq C.
\end{equation}
Together with Lemma \ref{79},
\begin{equation}\label{80}
    \sup_{0\leq t<T_1,p\in\Sigma_t\cap\{y\leq\epsilon\}}g(t,p)\leq Cy^{-2\mu}e^{2n(\mu+1)T_1}.
\end{equation}

On the other hand, we can use a maximum principle for $g$ on $W_\epsilon=\Big(\bigcup_{0\leq t<T_1}\Gamma_t\times\{t\}\Big)\cap\{y\geq\epsilon\}$ to see that $g$ is bounded above by the maximum of $g$ over the parabolic boundary of $W_\epsilon$. Hence,
\begin{equation}\label{81}
    \sup_{0\leq t< T_1,p\in\Sigma_t\cap\{y\geq\epsilon\}}g(t,p)\leq Cy^{-2\mu}e^{2n(\mu+1)T_1}.
\end{equation}
Combining (\ref{80}) and (\ref{81}), we get the desired bound.
\end{proof}

\begin{lem}\label{13}
Let $2\leq n\leq6$ and $\{\Sigma_t\}_{t\in[0,T)}$ be the same flow as in Lemma \ref{18}. If we further assume that $\lambda_\mathbb{H}[\Sigma_0]<\lambda[\mathbb{S}^{n-1}\times\mathbb{R}^1]$, then the maximal time of existence is $T=\infty$.
\end{lem}

\begin{proof}
Assume for now $T<\infty$. From (\ref{82}), we see that
\begin{equation*}
    \lim_{t\to T^-}\sup_{p\in\Sigma_t}|A_{\Sigma_t}|=\lim_{t\to T^-}\sup_{p\in\Sigma_t\cap\{y\geq\epsilon\}}|A_{\Sigma_t}|=\infty.
\end{equation*}
Pick a sequence $\{p_i\}\in\Sigma_{t_i}$ with $t_i\to T$ and $|A_{\Sigma_{t_i}}|(p_i)=\sup_{p\in\Sigma_{t_i}\cap\{y\geq\delta\}}|A_{\Sigma_{t_i}}|$. Without loss of generality, we may assume that $p_i\to \widetilde{p}$. Consider the exponential map $\text{exp}_{\widetilde{p}}:\,\mathbb{R}^{n+1}\to\mathbb{H}^{n+1}$. It is readily checked that $\text{exp}_{\widetilde{p}}$ is a diffeomorphism. Let $(\mathbb{R}^{n+1},g)$ be the pull-back of $(\mathbb{H}^{n+1},g_{\mathbb{H}^{n+1}})$ under the exponential map. With a slight abuse of notions, assume $\{\Sigma_t\}$ is also a flow in $(\mathbb{R}^{n+1},g)$. Let $\{\hat{\Sigma}^\lambda_s\}_{s\in[-T\lambda^{-2},0)}=\{\frac{1}{\lambda}\Sigma_{T+\lambda^2s}\}$. Here, the dilation is understood in $\mathbb{R}^{n+1}$. One can check that $\{\hat{\Sigma}_s^\lambda\}$ is a mean curvature flow in $(\mathbb{R}^{n+1},g^\lambda)$, where $g^\lambda_{ij}(\textbf{x})=g(\lambda\textbf{x})(\partial_i,\partial_j)=g_{ij}(\lambda\textbf{x})$. By \cite{BernsteinEntropy},
\begin{align*}
    \int_{\Sigma_t}\Phi^{0,\textbf{0}}_n(t-T,p)d\mathcal{H}^n_\mathbb{H}(p)
\end{align*}
is monotone decreasing in $t\in[0,T)$. After re-writing it into an integral over $\hat{\Sigma}^\lambda_s$ in $\mathbb{R}^{n+1}$, we have
\begin{equation}\label{16}
    \int_{\Sigma_t}\Phi^{0,\textbf{0}}_n(t-T,p)d\mathcal{H}^n_\mathbb{H}(p)=\int_{\hat{\Sigma}^\lambda_s}\lambda^n K_n(-\lambda^2s,\lambda\textbf{y})\sqrt{\text{det}\,g^\lambda(\textbf{y})}d\mathcal{H}^n(\textbf{y}).
\end{equation}
Since $g^\lambda\to g_{\mathbb{R}^{n+1}}$ in $C^\infty_{loc}(\mathbb{R}^{n+1})$ as $\lambda\to0$ and the arguments in \cite[Chapter 3]{Tonegawa} are local,  one can show with minor modifications that for any sequence $\lambda_{k}\to0^+$, there is a subsequence $\{\lambda_{k_i}\}\subset\{\lambda_k\}$, so that $\{\hat{\Sigma}^{\lambda_i}_s\}$ converges to a self-shrinking Brakke's flow $\{\mu_s\}_{s<0}$ in $(\mathbb{R}^{n+1},g_{\mathbb{R}^{n+1}})$. By the nature of convergence, $\lambda[\mu_s]<\lambda[\mathbb{S}^{n-1}\times\mathbb{R}^1]<2$.

By Lemma \ref{18} and the Brakke's regularity theorem, on the regular part of $\mu_{-1}$, $|A_{\mu_{-1}}|\leq CH_{\mu_{-1}}$. It then follows from \cite[Proposition 5.1]{CIMW} that $\{\mu_s\}$ is actually smooth and hence, by \cite[Theorem 0.17]{CMGeneric}, is either $\mathbb{S}^n$ or $\mathbb{R}^n$. Since $\Sigma_t$ doesn't have closed component, $\mu_s$ must be a plane. This implies $\{\Sigma_t\}$ is smooth at $(\widetilde{p},T)$, which contradicts our assumption.
\end{proof}

\begin{proof}[Proof of Theorem \ref{76}]
Combining Lemma \ref{79}, Lemma \ref{18} and Lemma \ref{13}, we see that for a sufficiently small $\epsilon$, the mean curvature flow $\{\Sigma_t\}$ starting from $\hat{\Gamma}_\epsilon$ exists for all $t\geq0$ and remains mean convex. We first show that as $t\to\infty$, $\Sigma_t\xrightarrow{C^\infty_{loc}}\Gamma'$ for some stable minimal hypersurface $\Gamma'$ that is $C^{k+1,\alpha}$-asymptotic to $M$.

Since $\Sigma_t$ is mean convex, there is a one-parameter family of open sets $\{U_t\}_{t\in[0,\infty)}$ so that $\partial U_t=\Sigma_t$ and $U_t\subset U_{t'}$ if $0\leq t'\leq t$. Let $K=\bigcap_{t\geq0}\text{cl}(U_t)$.
For any $\tau_i\to\infty$, consider the flow $\{\Sigma^i_t\}:=\{\Sigma_{t+\tau_i}\}$. By the compactness of Brakke's flow, there is a subsequence $\{\Sigma^{i_k}_t\}$ converging to a Brakke's flow $\{\nu_t\}_{t\in[0,\infty)}$. For each $t\geq0$, $\mathcal{H}^n_\mathbb{H}\mres\Sigma^{i_k}_t\to\nu_t$ implies $\text{Spt}(\nu_t)\subset\partial K$. On the other hand, the upper semicontinuity of the Gaussian density implies $\partial K\subset\text{Spt}(\nu_t)$. So, $\text{Spt}(\nu_t)=\partial K$ for all $t\geq0$. Furthermore, there is a subsequence of $\{\chi_{U_{t+i_k}}\}$ that converges as Caccioppoli sets. So, we see that $\nu_s$ is a multiplicity one rectifiable Radon measure. So, $\nu_s=\mathcal{H}^n_\mathbb{H}\mres\partial K$ for all $s\geq0$ and $\nu_s$ is thus stationary.

By the nature of convergence and the monotonicity formula \cite{BernsteinEntropy}, $\lambda_\mathbb{H}[\nu_s]<\lambda[\mathbb{S}^{n-1}\times\mathbb{R}^1]$. It then follows from this entropy bound and \cite{WhiteStratification} that the singular part of $\nu_0$ is at most of Hausdorff dimension $n-3$. As the mean curvature vector of $\Sigma_t$ points toward $\nu_0$, we see that $\{\Sigma_t\}_t$ serves as a foliation of $\nu_0$. So, by \cite{SolomonWhite}, $\nu_0$ is locally one-sided minimizing. By \cite{SchoenSimonRegularity}, Spt($\nu_0$)$:=\Gamma'$ is regular. Since $\lambda_\mathbb{H}[\Gamma']<\infty$, the regularity of $\Gamma'$ at the ideal boundary follows in the same manner as in the proof of Proposition \ref{22}. As $\Gamma'$ is smooth, \cite{WhiteLocalReg} implies that $\Sigma_t\xrightarrow{C^\infty_{loc}}\Gamma'$.

Fix $\epsilon>0$, so that $\Gamma\cap\{y\leq2\epsilon\}$ and $\Gamma'\cap\{y\leq2\epsilon\}$ are both $C^{k,\alpha}$ isotopic to $M\times(0,2\epsilon]$. By shrinking $\epsilon$ if necessary, we see from \cite[Lemma 5.1]{YaoRel} that $\Sigma_t\cap\{y\leq2\epsilon\}$ is a graph over $M\times(0,2\epsilon]$ for all $t$. Choose a transverse section $\textbf{v}$ on $\Gamma'$ with $\textbf{v}(p)$ perpendicular to $\overbar{\textbf{e}}_{n+1}$ on $\Gamma'\cap\{y=\epsilon\}$. Then, for $T_0>0$ sufficiently large, $\Sigma_{T_0}\cap\{y\geq\epsilon\}$ is a $\textbf{v}$-graph over $\Gamma'\cap\{y\geq\epsilon\}$. Let $\widetilde{\Sigma}_{T_0}$ coincide with $\Sigma_{T_0}$ on $\{y\geq\epsilon\}$ and extends to $M$ in a $C^{k,\alpha}$ way. We can further require that $\widetilde{\Sigma}_{T_0}$ and $\Gamma'$ are $C^{k,\alpha}$ isotopic. Now, partition $[0,T_0]$ into sub-intervals $[s_k,s_{k+1}]$, ($k=1,2,\cdots,N$). Since $\Sigma_t\cap\{y\leq2\epsilon\}$ is a graph over $M\times(0,2\epsilon]$, we may choose a trasnverse section $\textbf{v}_{k}$ over $\Sigma_{s_k}\cap\{y\geq\epsilon\}$ with $\textbf{v}_{k}$ perpendicular to $\overbar{\textbf{e}}_{n+1}$ on $\Sigma_{s_k}\cap\{y=\epsilon\}$ and that $\Sigma_{s_k}\cap\{y\geq\epsilon\}$ can be expressed as a $\textbf{v}_{k+1}$-graph over $\Sigma_{s_{k+1}}\cap\{y\geq\epsilon\}$. So, $\Sigma_{s_k}\cap\{y\geq\epsilon\}$ is isotopic to $\Sigma_{s_{k\pm1}}\cap\{y\geq\epsilon\}$. Finally, extend $\Sigma_{s_k}\cap\{y\geq\epsilon\}$ to $M$ in a $C^{k,\alpha}$ way with the extension of $\Sigma_{s_k}$ isotopic to that of $\Sigma_{s_{k\pm1}}$. Thus, we have shown that $\Sigma_0$ and $\Gamma'$ are $C^{k,\alpha}$ isotopic. The isotopy of $\Gamma$ and $\Sigma_0$ can be constructed in an obvious way.

Finally, we show that if there are minimal hypersurfaces $\Gamma_1$ and $\Gamma_2$ with $\Gamma_1\preceq\Gamma\preceq\Gamma_2$, then $\Gamma'$ is also trapped between $\Gamma_1$ and $\Gamma_2$. By Proposition \ref{10}, the distance between $\Gamma_1$ and $\Gamma_2$ at $y=r$ is bounded below by $cr^{n+1}$ for some constant $c>0$. By Proposition \ref{14}, we can shrink $\epsilon$ to make $\hat{\Gamma}_\epsilon$ satisfy $\Gamma_1\preceq\hat{\Gamma}_\epsilon\preceq\Gamma_2$. By the maximum principle, $\Gamma_1\preceq\Sigma_t\preceq\Gamma_2$ for all $t\geq0$. Hence, $\Gamma'$ is also trapped between $\Gamma_1$ and $\Gamma_2$.
\end{proof}

Following the arguments in Lemma \ref{18}, Lemma \ref{13} and the proof of Theorem \ref{76}, we have also proved the following:

\begin{thm}\label{86}
Assume $2\leq n\leq6$, $\alpha\in(0,1)$ and $k\geq2$. Let $\Sigma$ be a hypersurface that is $C^{k,\alpha}$-asymptotic to an $(n-1)$-dimensional $C^{k,\alpha}$ closed submanifold $M\subset\mathbb{R}^n\times\{0\}$ and meets $M$ orthogonally. If $\lambda_\mathbb{H}[\Sigma]<\lambda[\mathbb{S}^{n-1}\times\mathbb{R}^1]$, and $\Sigma$ satisfies $H_{\Sigma}\geq cy^\mu$ for some $c>0$ and $\mu\in\mathbb{R}^1$, then $\Sigma$ is $C^{k,\alpha}$ isotopic to a stable minimal hypersurface $\Sigma'$ that is also $C^{k,\alpha}$-asymptotic to $M$.
\end{thm}

\section{Topological Uniqueness for Minimal Hypersurfaces with Small Entropy}
In this final section, we prove Theorem \ref{main2}. Let $2\leq n\leq 6$, $\alpha\in(0,1)$, $k\geq4$ and $m=\min\{k,n\}$ throughout this section. Let $M$ be a $C^{k+1,\alpha}$ closed submanifold in $\partial_\infty\mathbb{H}^{n+1}$ with $\lambda_c[M]< \text{Vol}(\mathbb{S}^{n-1})\lambda[\mathbb{S}^{n-1}\times\mathbb{R}^1]$. After suitable transformations, we assume that $M\subset\mathbb{R}^n\times\{0\}$. By a minor modification of Proposition \ref{21}, Proposition \ref{34}, we can apply \cite{SmaleSard} to get $M'\in\mathcal{M}^{k+1,\alpha}$ being a regular value of $\Pi$ with $M'$ close to $M$ in the $C^{k+1,\alpha}$ topology in $\mathbb{R}^n\times\{0\}$.

We first prove Theorem \ref{main2} when $M$ is a regular value of $\Pi$. As a starting point, we prove

\begin{prop}\label{84}
Let $\ell\geq2$ and $M\in\mathcal{E}^{\ell,\alpha}$ be a regular value of $\Pi$. Then, there are only finitely many stable minimal hypersurfaces in $\mathcal{M}^{\ell,\alpha}$ that are $C^{\ell,\alpha}$-asymptotic to M.
\end{prop}

\begin{proof}
Let $\mathcal{S}:=\{\Gamma\in\mathcal{M}^{\ell,\alpha}:\Pi(\Gamma)=M,\,\Gamma\text{ is stable}\}$. If $\mathcal{S}$ contains infinitely many elements, we choose $\{\Gamma_j\}_{j=1}^\infty\subset\mathcal{S}$ which are distinct to each other. By Proposition \ref{22}, we get a convergent subsequence. With a slight abuse of notions, we assume that $\Gamma_j$ converges (in the $C^{\infty}_{loc}$ sense) to a stable minimal hypersurface $\widetilde{\Gamma}\in\mathcal{S}$. We claim that $\Gamma_j$ converges to $\widetilde{\Gamma}$ in the $C^{\ell,\frac{\alpha}{2}}$ topology in $\overbar{\mathbb{R}}^{n+1}_+$. Fix $p\in M$. Without loss of generality, assume $p=(\textbf{0},0)\in\mathbb{R}^n\times\{0\}$ and $\textbf{n}_{M}(p)=(0,\cdots,0,1,0)$. Since the argument in the proof of Proposition \ref{23} only relies on the geometry of ideal boundary $M$, there is a $\delta>0$, so that for any $j$,
\begin{equation*}
    \Gamma_j\cap(-\delta,\delta)^{n}\times[0,\delta)=\{(\textbf{x}',x_n,y):x_n=u_j(\textbf{x}',y),(\textbf{x}',y)\in(-\delta,\delta)^{n-1}\times[0,\delta)\},
\end{equation*}
where $u_j\in C^{\ell,\alpha}\Big((-\delta,\delta)^{n-1}\times[0,\delta)\Big)$ with $u_j(\textbf{x}',0)=\varphi(\textbf{x}')$ parametrize $M\cap\Big((-\delta,\delta)^{n-1}\times\{0\}\Big)$. By \cite[Theorem 5.2]{LinHyperbolicMinimalGraph},
\begin{equation*}
    \|u_j\|_{C^{\ell,\alpha}\big((-\frac{\delta}{2},\frac{\delta}{2})^{n-1}\times[0,\frac{\delta}{2})\big)}\leq C(n,\ell,\alpha,\delta)\left(1+\|\varphi\|_{C^{\ell,\alpha}\big((-\delta,\delta)^{n-1}\times[0,\delta)\big)}\right).
\end{equation*}
This implies every subsequence of $u_j$ has a sub-subsequence that converges in the $C^{\ell,\frac{\alpha}{2}}_{loc}\big((-\frac{\delta}{2},\frac{\delta}{2})^{n-1}\times[0,\frac{\delta}{2})\big)$ sense to a function $\widetilde{u}$. Since $\Gamma_j\xrightarrow{C^\infty_{loc}}\widetilde{\Gamma}$, the graph of $\widetilde{u}$ coincides with $\widetilde{\Gamma}$. Hence, $\Gamma_j\to\widetilde{\Gamma}$ in the $C^{\ell,\frac{\alpha}{2}}$-topology of $\overbar{\mathbb{R}}^{n+1}_+$.

On the other hand, since $M$ is a regular value, $\Pi$ maps a neighborhood of $\widetilde{\Gamma}$ in $\mathcal{M}^{\ell,\frac{\alpha}{2}}$ diffeomorphically to a neighborhood of $M$ in $\mathcal{E}^{\ell,\frac{\alpha}{2}}$, which contradicts the fact that $\Gamma_k\to\widetilde{\Gamma}$ in $C^{\ell,\frac{\alpha}{2}}$-topology of $\overbar{\mathbb{R}}^{n+1}_+$.
\end{proof}

\begin{prop}\label{96}
Let $M\in\mathcal{E}^{k+1,\alpha}$ be a regular value of $\Pi$ with $\lambda_\mathbb{H}[M]<\mathrm{Vol}(\mathbb{S}^{n-1})\lambda[\mathbb{S}^{n-1}\times\mathbb{R}^1]$. Then, all the minimal hypersurfaces with $C^{k+1,\alpha}$-asymptotic boundary $M$ are $C^{k,\alpha}$-isotopic to each other.
\end{prop}

\begin{proof}
By Theorem \ref{76}, any unstable minimal hypersurface is isotopic to a stable one. So, it suffices to show that all the stable minimal hypersurfaces are isotopic to each other.

Fix a stable minimal hypersurface $\Gamma$ that is $C^{k+1,\alpha}$-asymptotic to $M$. By Proposition \ref{11}, there is a maximal stable minimal hypersurface $\Gamma_G$ with $\Gamma\preceq\Gamma_G$. Let $\mathcal{S}_{\Gamma,\Gamma_G}:=\{\Gamma'\in\mathcal{M}^{k+1,\alpha}:\Pi(\Gamma')=M, \Gamma\preceq\Gamma'\preceq\Gamma_G\}$ and $\Gamma_1$ be a minimal element in $\mathcal{S}_{\Gamma,\Gamma_G}$. That is, there's no $\widetilde{\Gamma}\in\mathcal{S}_{\Gamma,\Gamma_G}$ with $\Gamma\preceq\widetilde{\Gamma}\preceq\Gamma_1$. Since $M$ is a regular value, $\Gamma$ and $\Gamma_1$ are both strictly stable. So, Theorem \ref{main1} implies the existence of a third minimal hypersurface $\Gamma''$ with $\Gamma\preceq\Gamma''\preceq\Gamma_1$. Obviously, $\Gamma''$ is unstable. By Theorem \ref{76}, we see that $\Gamma''$ is $C^{k,\alpha}$ isotopic to both $\Gamma_1$ and $\Gamma$. Replace $\Gamma$ by $\Gamma_1$, and repeat the same argument to $\mathcal{S}_{\Gamma_1,\Gamma_G}$. Since there are only finitely many stable minimal hypersurfaces, an induction argument would imply $\Gamma$ is eventually isotopic to $\Gamma_G$. Since $\Gamma$ is arbitrary chosen, we have proved that all the stable minimal hypersurfaces are isotopic.
\end{proof}

\begin{proof}[Proof of Theorem \ref{main2}]
Notice that all the unstable minimal hypersurfaces with $C^{k+1,\alpha}$-asymptotic boundary $M$ are isotopic to stable ones. It suffices to show that, by appropriately choosing $M'$ which is a regular value of $\Pi$, any stable minimal hypersurface $\Gamma$ with $\Pi(\Gamma)=M$ is $C^{m,\alpha}$-isotopic to another minimal hypersurface $\Gamma'$ with $\Pi(\Gamma')=M'$.

When $\Gamma$ is strictly stable, the map $\Pi$ is a local diffeomorphism from a neighborhood $\mathcal{U}$ of $\Gamma$ in $\mathcal{M}^{m,\alpha}$ to a neighborhood of $M$ in $\mathcal{E}^{m,\alpha}$. Let $\psi\in C^{m,\alpha}(M)$ be a function with
\begin{equation*}
    M_\psi=\big\{\textbf{x}(p)+\psi(p)\nu_{M}(p):\,p\in M\big\},
\end{equation*}
where $\nu_{M}$ is a unit normal of $M$ in $\mathbb{R}^n\times\{0\}$. By \cite[Theorem 3]{Anderson1982}, there is a minimal hypersurface $\Upsilon_\psi$ asymptotic to $M_\psi$. If $\psi$ is sufficiently small in $C^2$-norm, by \cite[Theorem 2.2]{HardtLinHyperbolicRegularity}, there is a $\delta>0$ independent of $\psi$, so that for any $p\in M_\psi$, $\big(B^{n}_\delta(p)\times[0,\delta)\big)\cap\Upsilon_\psi$ can be locally written as a graph over a hyperplane that is parallel to $\textbf{e}_{n+1}$. Similar to the proof of Proposition \ref{20}, we choose an approximate minimal hypersurface asymptotic to $M_\psi$ by letting $\Gamma_\psi:=\{\textbf{x}(p)+\mathcal{E}(\psi)\textbf{n}_{\Gamma}(p):\ p\in\Gamma\}$ coincide with $\Upsilon_\psi$ in $\{y\leq\frac{\delta}{4}\}$ and coincide with $\Gamma$ in $\{y\geq\frac{\delta}{2}\}$. We can patch them together by a cut-off function in $y$ since both of them can be expressed as graphs near the ideal boundary. We see that the mean curvature of $\Gamma_\psi$, $H_\psi$, is in $y^{\mu_0}\Lambda^{m-2,\alpha}_0$ for \textit{any} $\mu_0>0$. To apply the theory of uniformly degenerate operators, we assume $\mu_0\in(0,n)$. Let $\Gamma_{\psi,\phi}:=\{\textbf{x}(p)+(\mathcal{E}(\psi)+\phi)y(p)\overbar{\textbf{n}}_{\Gamma}(p):\ p\in\Gamma\}$ for $\phi\in y^{\mu_0}\Lambda_0^{m,\alpha}(\Gamma)$. Note that for $\mu_0$ close to $n$, $y^{\mu_0+1}\Lambda^{\ell,\alpha}_0$ is continuously embedded into $C^{\ell,\alpha}(\overbar{\Gamma})$ whenever $\ell\leq n$. Since $\phi\in y^{\mu_0}\Lambda^{m,\alpha}_0$, we see that $\Gamma_{\psi,\phi}:=\{\textbf{x}(p)+(\mathcal{E}(\psi)+\phi)y(p)\overbar{\textbf{n}}_{\Gamma}(p):\ p\in\Gamma\}$ is $C^{m,\alpha}$ up to the ideal boundary.

The mean curvature of $\Gamma_{\psi,\phi}$, $H_{\psi,\phi}$, is in $y^{\mu_0}\Lambda^{m-2,\alpha}_0$. Hence, $H_{\cdot,\cdot}$ maps a neighborhood of $(0,0)$ in $C^{m,\alpha}(M)\times y^{\mu_0}\Lambda_0^{m,\alpha}(\Gamma)$ to $y^{\mu_0}\Lambda_0^{m-2,\alpha}(\Gamma)$. By the implicit function theorem, there exists a map $\mathcal{K}$ mapping from a neighborhood of the origin in $C^{m,\alpha}(M)$ to $y^{\mu_0}\Lambda_0^{m,\alpha}(\Gamma)$, such that $H_{\psi,\mathcal{K}(\psi)}\equiv 0$. Now, pick $\psi$ so that $M_\psi=M'\in\mathcal{U}$ is a regular value. Then, $t\mapsto \Gamma^i_{t\psi,\mathcal{K}(t\psi)}$ is a $C^{m,\alpha}$ isotopy between $\Gamma$ and another minimal hypersurface $\Gamma_{\psi,\mathcal{K}(\psi)}$.

Now, assume $\Gamma$ is weakly stable. Write $\Gamma$ into disjoint components $\Gamma=\sqcup\Gamma^i$. Let $M^i=\partial_\infty\Gamma^i$. We proceed as in the preceding paragraph with $\Gamma$ replaced by $\Gamma^i$. Hence, for functions $\psi\in C^{m,\alpha}(\Gamma^i)$ and $\phi\in y^{\mu_0}\Lambda_0^{m,\alpha}$, there are a closed submanifold $M^i_\psi$ with
\begin{equation*}
    M^i_\psi=\big\{\textbf{x}(p)+\psi(p)\nu_{M^i}(p):\,p\in M^i\big\}
\end{equation*}
and an approximate minimal hypersurface $\Gamma^i_{\psi,\phi}:=\{\textbf{x}(p)+(\mathcal{E}(\psi)+\phi)y(p)\overbar{\textbf{n}}_{\Gamma^i}(p):\ p\in\Gamma^i\}$ with the mean curvature, $H_{\psi,\phi}$, lying in $y^{\mu_0}\Lambda^{m-2,\alpha}_0$. If $\Gamma^i$ is strictly stable, the isotopy follows from the preceding paragraph.

If $\Gamma^i$ is weakly stable, since $\lambda=0$ corresponds to the first eigenvalue of $\Gamma^i$, we know that $\text{Ker}(L_{\Gamma^i})=\text{Span}\{f\}$. By Proposition \ref{14}, $f$ is positive with $\frac{1}{C}y^n\leq f\leq Cy^n$. Let $X_{\Gamma^i}:=\{\phi\in y^{\mu_0}\Lambda^{m,\alpha}_0:\,\int_{\Gamma^i} \phi fd\mathcal{H}^n_\mathbb{H}=0\}$. Since $f\in y^n\Lambda^{m,\alpha}_0$, the integral of $\phi f$ over $\Gamma^i$ is well defined. Moreover, $X_{\Gamma^i}$ is a closed subspace of $y^{\mu_0}\Lambda^{m,\alpha}_0$. Let
\begin{equation*}
    P_{\Gamma^i}:\,y^{\mu_0}\Lambda_0^{m-2,\alpha}\to X_{\Gamma^i},\,g\mapsto g-\frac{1}{\|f\|_{L^2(\Gamma^i)}^2}\left\langle f,g\right\rangle_{L^2(\Gamma^i)}f
\end{equation*}
be the projection map. It is readily checked that $P_{\Gamma^i}$ is a bounded linear map. By implicit function theorem, there exists a map $\mathcal{K}$ mapping from a neighborhood of the origin in $C^{m,\alpha}(M^i)$ to $X_{\Gamma^i}$, such that $P_{\Gamma^i}\circ H_{\psi,\mathcal{K}(\psi)}\equiv 0$. Now, pick $\psi$ so that $M^i_\psi$ is a regular value. Then, $t\mapsto \Gamma^i_{t\psi,\mathcal{K}(t\psi)}$ is a $C^{m,\alpha}$ isotopy between $\Gamma^i$ and $\Gamma^i_{\psi,\mathcal{K}(\psi)}$, where $m=\min\{k,n\}$. By Lemma \ref{83}, we can assume $\lambda_\mathbb{H}[\Gamma^i_{\psi,\mathcal{K}(\psi)}]<\lambda[\mathbb{S}^{n-1}\times\mathbb{R}^1]$ by choosing $\psi$ sufficiently small in $C^{m,\alpha}$-norm. Since $P_{\Gamma^i}\circ H_{\psi,\mathcal{K}(\psi)}=0$, the mean curvature $H_{\psi,\mathcal{K}(\psi)}=cf$ for some constant $c\in\mathbb{R}^1$. Pick an appropriate direction of normal vector, we assume that $c>0$. Applying Theorem \ref{86}, we get an isotopy between $\Gamma^i_{\psi,\mathcal{K}(\psi)}$ and a minimal hypersurface $\widetilde{\Gamma}^i$ with $\partial_\infty\widetilde{\Gamma}^i=M^i_\psi$.

In either case, we have proved that $\Gamma$ is isotopic to another minimal hypersurface $\Gamma'$ with $M'=\partial_\infty\Gamma'$ being a regular value. This finishes the proof Theorem \ref{main2}.
\end{proof}

\appendix
\section{Computations for non-compactly supported vector fields}
Let $\textbf{Y}=\alpha\textbf{N}_\epsilon+\textbf{V}\in\mathcal{Y}^-$ and $\{\Phi_t(p)\}$ be the $1$-parameter family of diffeomorphisms generated by $\textbf{Y}$ throughout the appendix. Assume $\|\textbf{Y}\|_{\mathcal{Y}^-}\leq M_0$.

\begin{lem}\label{a1}
There is a constant $C=C(\Gamma_-,\Omega',M_0)$, so that
\begin{align}
    &|\mathbf{Y}(p)-(\alpha\mathbf{x}(p)\cdot\overbar{\mathbf{e}}_{n+1})\overbar{\mathbf{e}}_{n+1}|\leq C(\mathbf{x}(p)\cdot\overbar{\mathbf{e}}_{n+1})^2,\label{47}\\
    &|\overbar{\nabla}\mathbf{Y}-\alpha\overbar{\mathbf{e}}_{n+1}\otimes\overbar{\mathbf{e}}_{n+1}|\leq C\mathbf{x}(p)\cdot\overbar{\mathbf{e}}_{n+1}\text{ and }|\overbar{\nabla}^2\mathbf{Y}|\leq C\label{102}
\end{align}
\end{lem}

\begin{proof}
From the estimate on $\textbf{V}$, it suffices to show that $|\textbf{N}_\epsilon-(\textbf{x}(p)\cdot\overbar{\textbf{e}}_{n+1})\overbar{\textbf{e}}_{n+1}|\leq C(\textbf{x}(p)\cdot\overbar{\textbf{e}}_{n+1})^2$, that $|\overbar{\nabla}\textbf{N}_\epsilon-\overbar{\textbf{e}}_{n+1}\otimes\overbar{\textbf{e}}_{n+1}|\leq C\textbf{x}(p)\cdot\overbar{\textbf{e}}_{n+1}$ and that $|\overbar{\nabla}^2\textbf{N}_\epsilon|\leq C$. The estimate $|\overbar{\nabla}^2\textbf{N}_\epsilon|\leq C$ follows from the fact that $\Gamma_-$ is $C^3$-asymptotic. In $\text{cl}(\Omega')\cap\{y<\frac{\epsilon}{2}\}$, $\textbf{N}_\epsilon=(\textbf{x}(p)\cdot\overbar{\textbf{e}}_{n+1})\overbar{\textbf{e}}^\top_{n+1}\circ\Pi$. So, near the ideal boundary, we have
\begin{align*}
    |\textbf{N}_\epsilon-(\textbf{x}(p)\cdot\overbar{\textbf{e}}_{n+1})\overbar{\textbf{e}}_{n+1}|=(\textbf{x}(p)\cdot\overbar{\textbf{e}}_{n+1})|\overbar{\textbf{e}}^\top_{n+1}\circ\Pi-\overbar{\textbf{e}}_{n+1}|.
\end{align*}
Since $\overbar{\textbf{e}}^\top_{n+1}=\overbar{\textbf{e}}_{n+1}$ on $\overbar{\Gamma_-}\cap\{y=0\}$, (\ref{47}) follows easily. On the other hand, near the ideal boundary, we have
\begin{align*}
    \overbar{\nabla}\textbf{N}_\epsilon=\overbar{\textbf{e}}_{n+1}\otimes\overbar{\textbf{e}}_{n+1}+(\textbf{x}(p)\cdot\overbar{\textbf{e}}_{n+1})\overbar{\nabla}(\overbar{\textbf{e}}^\top_{n+1}\circ\Pi).
\end{align*}
Since $\Gamma_-$ is $C^3$-asymptotic, $\overbar{\textbf{e}}^\top_{n+1}\circ\Pi$ is $C^2$ up to the ideal boundary. So, $|\overbar{\nabla}(\overbar{\textbf{e}}^\top_{n+1}\circ\Pi)|\leq C$. This completes the proof.
\end{proof}

\begin{lem}\label{a2}
Let $0\leq t\leq T<\infty$. There is a constant $C=C(\Gamma_-,\Omega',M_0,T)$, so that
\begin{align}
    \left|\frac{\mathbf{x}(p)\cdot\overbar{\mathbf{e}}_{n+1}}{\Phi_t(p)\cdot\overbar{\mathbf{e}}_{n+1}}-e^{-\alpha t}\right|&\leq C\mathbf{x}(p)\cdot\overbar{\mathbf{e}}_{n+1}\label{38}\\
    |\overbar{\nabla}\Phi_t-I_{n+1}-(e^{\alpha t}-1)\overbar{\mathbf{e}}_{n+1}\otimes\overbar{\mathbf{e}}_{n+1}|&\leq C \mathbf{x}(p)\cdot\overbar{\mathbf{e}}_{n+1}\text{ and }|\overbar{\nabla}^2\Phi_t|\leq C\label{41}.
\end{align}
\end{lem}

\begin{proof}
Since $\frac{d}{dt}\Phi_t(p)=\textbf{Y}(\Phi_t(p))$, using (\ref{47}), one sees that
\begin{align*}
    \left|\frac{d}{dt}\left(\frac{\Phi_t(p)\cdot\overbar{\textbf{e}}_{n+1}}{\textbf{x}(p)\cdot\overbar{\textbf{e}}_{n+1}}\right)\right|\leq C\left|\frac{\Phi_t(p)\cdot\overbar{\textbf{e}}_{n+1}}{\textbf{x}(p)\cdot\overbar{\textbf{e}}_{n+1}}\right|.
\end{align*}
So,
\begin{equation}\label{113}
    |\Phi_t(p)\cdot\overbar{\textbf{e}}_{n+1}|\leq C\textbf{x}(p)\cdot\overbar{\textbf{e}}_{n+1}
\end{equation}
Similarly, one can compute that
\begin{align*}
    \frac{d}{dt}\left(\frac{\textbf{x}(p)\cdot\overbar{\textbf{e}}_{n+1}}{\Phi_t(p)\cdot\overbar{\textbf{e}}_{n+1}}\right)=-\frac{\textbf{x}(p)\cdot\overbar{\textbf{e}}_{n+1}}{(\Phi_t(p)\cdot\overbar{\textbf{e}}_{n+1})^2}\textbf{Y}(\Phi_t(p))\cdot\overbar{\textbf{e}}_{n+1}=-\alpha\frac{\textbf{x}(p)\cdot\overbar{\textbf{e}}_{n+1}}{\Phi_t(p)\cdot\overbar{\textbf{e}}_{n+1}}+\mathcal{O}(\textbf{x}(p)\cdot\overbar{\textbf{e}}_{n+1}),
\end{align*}
where in the last step, we used (\ref{47}). This implies that
\begin{equation*}
    \left|\frac{d}{dt}\left(e^{\alpha t}\frac{\textbf{x}(p)\cdot\overbar{\textbf{e}}_{n+1}}{\Phi_t(p)\cdot\overbar{\textbf{e}}_{n+1}}\right)\right|\leq C(\textbf{x}(p)\cdot\overbar{\textbf{e}}_{n+1}).
\end{equation*}
Taking the integral and noticing that $\Phi_0$ is the identity map, we get (\ref{38}). To prove (\ref{41}), notice that $\frac{d}{dt}\overbar{\nabla}\Phi_t=\overbar{\nabla}\frac{d}{dt}\Phi_t=\overbar{\nabla}(\textbf{Y}(\Phi_t(p)))=\overbar{\nabla}\textbf{Y}\circ\overbar{\nabla}\Phi_t$. Combining this with (\ref{102}),
\begin{align*}
    \left|\frac{d}{dt}\overbar{\nabla}\Phi_t-\alpha(\overbar{\mathbf{e}}_{n+1}\otimes\overbar{\mathbf{e}}_{n+1})\circ\overbar{\nabla}\Phi_t\right|\leq C\Phi_t(p)\cdot\overbar{\mathbf{e}}_{n+1}
\end{align*}
That is,
\begin{align}\label{112}
    \left|\frac{d}{dt}\left(e^{-\alpha tA}\overbar{\nabla}\Phi_t\right)\right|\leq C\Phi_t(p)\cdot\overbar{\mathbf{e}}_{n+1}e^{-\alpha tA},
\end{align}
where $A=\overbar{\mathbf{e}}_{n+1}\otimes\overbar{\mathbf{e}}_{n+1}$. As a result, $e^{-\alpha tA}=\sum_{k=0}^\infty(-\alpha tA)^k/k!=I_{n+1}+(e^{-\alpha t}-1)A$. Taking the integral in (\ref{112}),
\begin{align*}
    |\overbar{\nabla}\Phi_t-I_{n+1}-(e^{\alpha t}-1)\overbar{\mathbf{e}}_{n+1}\otimes\overbar{\mathbf{e}}_{n+1}|\leq C\Phi_t(p)\cdot\overbar{\mathbf{e}}_{n+1}\leq C\textbf{x}(p)\cdot\overbar{\textbf{e}}_{n+1},
\end{align*}
where in the last step, we used (\ref{113}). Finally, we show $|\overbar{\nabla}^2\Phi_t|\leq C$. Since
\begin{align*}
    \frac{d}{dt}\overbar{\nabla}^2\Phi_t=\overbar{\nabla}^2\frac{d}{dt}\Phi_t=\overbar{\nabla}^2\textbf{Y}(\overbar{\nabla}\Phi_t,\overbar{\nabla}\Phi_t)+\overbar{\nabla}\textbf{Y}\circ\overbar{\nabla}^2\Phi_t,
\end{align*}
we have $|\frac{d}{dt}\overbar{\nabla}^2\Phi_t|\leq C(1+|\overbar{\nabla}^2\Phi_t|)$. This implies $|\overbar{\nabla}^2\Phi_t|\leq C$.
\end{proof}

\begin{lem}\label{a3}
$\mathcal{J}^E\Phi_t(p,\mathbf{v})\in\mathfrak{Y}$. Moreover, $\|\mathcal{J}^E\Phi_t(p,\mathbf{v})\|_{\mathfrak{Y}}\leq C$, where $C=C(\Gamma_-,\Omega',M_0,T)$.
\end{lem}

\begin{proof}
Since $\mathfrak{Y}$ is an algebra and $\mathcal{J}^E\Phi_t=\mathcal{J}\Phi_t\left(\frac{\textbf{x}(p)\cdot\overbar{\textbf{e}}_{n+1}}{\Phi_t(p)\cdot\overbar{\textbf{e}}_{n+1}}\right)^n$, it suffices to show that $\mathcal{J}\Phi_t\in\mathfrak{Y}$ and $\frac{\textbf{x}(p)\cdot\overbar{\textbf{e}}_{n+1}}{\Phi_t(p)\cdot\overbar{\textbf{e}}_{n+1}}\in\mathfrak{Y}$.
By direct computations,
\begin{align*}
    (\textbf{x}(p)\cdot\overbar{\textbf{e}}_{n+1})\overbar{\nabla}\left(\frac{(\textbf{x}(p)\cdot\overbar{\textbf{e}}_{n+1})}{\left(\Phi_t(p)\cdot\overbar{\textbf{e}}_{n+1}\right)}\right)&=\frac{(\textbf{x}(p)\cdot\overbar{\textbf{e}}_{n+1})}{\left(\Phi_t(p)\cdot\overbar{\textbf{e}}_{n+1}\right)}\left(\overbar{\textbf{e}}_{n+1}-(\textbf{x}(p)\cdot\overbar{\textbf{e}}_{n+1})\frac{\overbar{\nabla}\Phi_t(p)\cdot\overbar{\textbf{e}}_{n+1}}{\Phi_t(p)\cdot\overbar{\textbf{e}}_{n+1}}\right)\\
    &=-\frac{(\textbf{x}(p)\cdot\overbar{\textbf{e}}_{n+1})}{\left(\Phi_t(p)\cdot\overbar{\textbf{e}}_{n+1}\right)}\int_0^t\big(\textbf{x}(p)\cdot\overbar{\textbf{e}}_{n+1}\big)\frac{d}{d\tau}\Big(\frac{\overbar{\nabla}\Phi_\tau(p)\cdot\overbar{\textbf{e}}_{n+1}}{\Phi_\tau(p)\cdot\overbar{\textbf{e}}_{n+1}}\Big)d\tau.
\end{align*}
Notice that
\begin{align*}
     \frac{d}{d\tau}\Big(\frac{\overbar{\nabla}\Phi_\tau(p)\cdot\overbar{\textbf{e}}_{n+1}}{\Phi_\tau(p)\cdot\overbar{\textbf{e}}_{n+1}}\Big)=\frac{\overbar{\nabla}\textbf{Y}(\Phi_\tau(p))\circ\overbar{\nabla}\Phi_\tau(p)\cdot\overbar{\textbf{e}}_{n+1}}{\Phi_\tau(p)\cdot\overbar{\textbf{e}}_{n+1}}-\frac{\overbar{\nabla}\Phi_\tau(p)\cdot\overbar{\textbf{e}}_{n+1}}{\big(\Phi_\tau(p)\cdot\overbar{\textbf{e}}_{n+1}\big)^2}\big(\textbf{Y}(\Phi_\tau(p))\cdot\overbar{\textbf{e}}_{n+1}\big).
\end{align*}
From (\ref{47}), (\ref{102}), (\ref{38}) and (\ref{41}), one sees that
\begin{equation}\label{39}
    \left|(\textbf{x}(p)\cdot\overbar{\textbf{e}}_{n+1})\overbar{\nabla}\left(\frac{(\textbf{x}(p)\cdot\overbar{\textbf{e}}_{n+1})}{\left(\Phi_t(p)\cdot\overbar{\textbf{e}}_{n+1}\right)}\right)\right|\leq C\textbf{x}(p)\cdot\overbar{\textbf{e}}_{n+1}.
\end{equation}
Write 
\begin{equation}\label{103}
    \frac{\textbf{x}(p)\cdot\overbar{\textbf{e}}_{n+1}}{\Phi_t(p)\cdot\overbar{\textbf{e}}_{n+1}}-e^{-\alpha t}=\phi_{y_1,\delta}(\frac{\textbf{x}(p)\cdot\overbar{\textbf{e}}_{n+1}}{\Phi_t(p)\cdot\overbar{\textbf{e}}_{n+1}}-e^{-\alpha t})+(1-\phi_{y_1,\delta})(\frac{\textbf{x}(p)\cdot\overbar{\textbf{e}}_{n+1}}{\Phi_t(p)\cdot\overbar{\textbf{e}}_{n+1}}-e^{-\alpha t}).
\end{equation}
Denote the second term on the right hand side of (\ref{103}) by $g_{y_1}$. Combining (\ref{38}) and (\ref{39}), one sees that $\|g_{y_1}\|_\mathfrak{X}\leq C(\Gamma_-,\Omega',M_0,T)y_1$. Hence, $\frac{\textbf{x}(p)\cdot\overbar{\textbf{e}}_{n+1}}{\Phi_t(p)\cdot\overbar{\textbf{e}}_{n+1}}-e^{-\alpha t}$ can be approximated by compactly supported functions in $\mathfrak{X}$-norm. This implies $\frac{\textbf{x}(p)\cdot\overbar{\textbf{e}}_{n+1}}{\Phi_t(p)\cdot\overbar{\textbf{e}}_{n+1}}\in\mathfrak{Y}$.

We next prove $\mathcal{J}\Phi_t(p,\textbf{v})\in\mathfrak{Y}$. For $\textbf{v}\in\mathbb{S}^n$, we choose a local orthogonal frame $\{\tau_1(\textbf{v}),\cdots,\tau_n(\textbf{v})\}$ near $\textbf{v}$ such that $\tau_i$ is smooth and $|\nabla_{\mathbb{S}^n}\tau_i|+|\nabla^2_{\mathbb{S}^n}\tau_i|\leq C$, where $C$ is a constant independent of $p$. Combining (\ref{47}), (\ref{102}) and (\ref{41}), one can prove in a similar way that
\begin{align}\label{104}
    \mathcal{J}\Phi_t-\sqrt{1+(e^{\alpha t}-1)(1-(\overbar{\textbf{e}}_{n+1}\cdot\textbf{v})^2)}=\phi_{y_1,\delta}\left(\mathcal{J}\Phi_t-\sqrt{1+(e^{\alpha t}-1)(1-(\overbar{\textbf{e}}_{n+1}\cdot\textbf{v})^2)}\right)+h_{y_1},
\end{align}
where $h_{y_1}$ satisfies $\|h_{y_1}\|_{\mathfrak{X}}\leq C(\Gamma_-,\Omega',M_0,T)y_1$. This implies $\mathcal{J}\Phi_t\in\mathfrak{Y}$.
\end{proof}

\begin{lem}\label{a4}
The map $t\mapsto\mathcal{J}^E\Phi_{t}(p,\textbf{v})$ is in $C^1([0,\infty);\mathfrak{Y})$. Moreover, 
\begin{equation*}
    \frac{d}{dt}\big|_{t=0}\mathcal{J}^E\Phi_{t}=\overbar{\mathrm{div}}\mathbf{Y}-Q_\mathbf{Y}-\frac{n}{y}\mathbf{Y}\cdot\overbar{\mathbf{e}}_{n+1},
\end{equation*}
where $Q_\mathbf{Y}(p,\mathbf{v})=\overbar{\nabla}_\mathbf{v}\mathbf{Y}(p)\cdot\mathbf{v}$.
\end{lem}

\begin{proof}
To show the differentiability of the map $t\mapsto\mathcal{J}^E\Phi_{t}(p,\textbf{v})$, it suffices to show the differentiability of the maps $t\mapsto\frac{\textbf{x}(p)\cdot\overbar{\textbf{e}}_{n+1}}{\Phi_t(p)\cdot\overbar{\textbf{e}}_{n+1}}$ and $t\mapsto\mathcal{J}\Phi_{t}(p,\textbf{v})$.

Fix $t_0\in[0,\infty)$. Notice that 
\begin{align*}
    \frac{\textbf{x}(p)\cdot\overbar{\textbf{e}}_{n+1}}{\Phi_t(p)\cdot\overbar{\textbf{e}}_{n+1}}-\frac{\textbf{x}(p)\cdot\overbar{\textbf{e}}_{n+1}}{\Phi_{t_0}(p)\cdot\overbar{\textbf{e}}_{n+1}}=-\int_{t_0}^t\frac{\textbf{x}(p)\cdot\overbar{\textbf{e}}_{n+1}}{(\Phi_\tau(p)\cdot\overbar{\textbf{e}}_{n+1})^2}\left(\textbf{Y}(\Phi_\tau(p))\cdot\overbar{\textbf{e}}_{n+1}\right)d\tau.
\end{align*}
Let
\begin{equation*}
    g_{t_0,t}(p)=\frac{\textbf{x}(p)\cdot\overbar{\textbf{e}}_{n+1}}{(\Phi_t(p)\cdot\overbar{\textbf{e}}_{n+1})^2}\left(\textbf{Y}(\Phi_t(p))\cdot\overbar{\textbf{e}}_{n+1}\right)-\frac{\textbf{x}(p)\cdot\overbar{\textbf{e}}_{n+1}}{(\Phi_{t_0}(p)\cdot\overbar{\textbf{e}}_{n+1})^2}\left(\textbf{Y}(\Phi_{t_0}(p))\cdot\overbar{\textbf{e}}_{n+1}\right).
\end{equation*}
To show the differentiability of $t\mapsto\frac{\textbf{x}(p)\cdot\overbar{\textbf{e}}_{n+1}}{\Phi_t(p)\cdot\overbar{\textbf{e}}_{n+1}}$, it suffices to show $\|g_{t_0,t}\|_\mathfrak{Y}\to0$ as $t\to t_0$. Firstly, we have $\|\phi_{y_1,\delta}g_{t_0,t}\|_\mathfrak{X}\to0$ as $t\to t_0$ for any fixed $y_1>\delta>0$. Then, from (\ref{47})-(\ref{41}), $\|(1-\phi_{y_1,\delta})g_{t_0,t}\|_\mathfrak{Y}\leq Cy_1$. So, $\|g_{t_0,t}\|_\mathfrak{Y}\to0$ is proved.

Similarly, if we let $\{\tau_1(\textbf{v}),\cdots,\tau_n(\textbf{v})\}$ be a local orthogonal frame in $\mathbb{S}^n$ near $\textbf{v}$, then
\begin{equation*}
    \mathcal{J}\Phi_t(p,\textbf{v})-\mathcal{J}\Phi_{t_0}(p,\textbf{v})=\int_{t_0}^t\frac{1}{2\mathcal{J}\Phi_r}\frac{d}{dr}\left(\det(\overbar{\nabla}_{\tau_i}\Phi_r\cdot\overbar{\nabla}_{\tau_j}\Phi_r)\right)dr.
\end{equation*}
Since $\frac{d}{dr}\left(\det(\overbar{\nabla}_{\tau_i}\Phi_r\cdot\overbar{\nabla}_{\tau_j}\Phi_r)\right)$ is a polynomial of $(\overbar{\nabla}_{k}\textbf{Y}\cdot\overbar{\nabla}_{\tau_j}\Phi_r)(\overbar{\nabla}_{\tau_i}\Phi_r\cdot\overbar{\textbf{e}}_k)$ ($1\leq i,j\leq n$ and $1\leq k\leq n+1$). So, the differentiability of $t\mapsto\mathcal{J}\Phi_{t}(p,\textbf{v})$ is equivalent to the following
\begin{align*}
    &\|(\overbar{\nabla}_{k}\textbf{Y}\cdot\overbar{\nabla}_{\tau_j}\Phi_t)(\overbar{\nabla}_{\tau_i}\Phi_t\cdot\overbar{\textbf{e}}_k)-(\overbar{\nabla}_{k}\textbf{Y}\cdot\overbar{\nabla}_{\tau_j}\Phi_{t_0})(\overbar{\nabla}_{\tau_i}\Phi_{t_0}\cdot\overbar{\textbf{e}}_k)\|_\mathfrak{X}\to0;\\
    &\|\frac{1}{\mathcal{J}\Phi_t(p,\textbf{v})}-\frac{1}{\mathcal{J}\Phi_{t_0}(p,\textbf{v})}\|_\mathfrak{X}\to0\text{ as }t\to t_0.
\end{align*}
They can be verified using a cut-off function similar to the proof of the differentiability of $t\mapsto\frac{\textbf{x}(p)\cdot\overbar{\textbf{e}}_{n+1}}{\Phi_t(p)\cdot\overbar{\textbf{e}}_{n+1}}$. Finally, letting $t_0=0$, one sees that
\begin{equation*}
    \frac{d}{dt}\Big|_{t=0}\frac{\textbf{x}(p)\cdot\overbar{\textbf{e}}_{n+1}}{\Phi_t(p)\cdot\overbar{\textbf{e}}_{n+1}}=\frac{\textbf{Y}(p)\cdot\overbar{\textbf{e}}_{n+1}}{\textbf{x}(p)\cdot\overbar{\textbf{e}}_{n+1}},\ \frac{d}{dt}\Big|_{t=0}\mathcal{J}\Phi_t(p,\textbf{v})=\sum_{i=1}^n\overbar{\nabla}_{\tau_i}\textbf{Y}\cdot\tau_i=\overbar{\text{div}}(\textbf{Y})-\overbar{\nabla}_\textbf{v}\textbf{Y}\cdot\textbf{v}.
\end{equation*}
This implies $\frac{d}{dt}\big|_{t=0}\mathcal{J}^E\Phi_{t}=\overbar{\mathrm{div}}\mathbf{Y}-Q_\mathbf{Y}-\frac{n}{y}\mathbf{Y}\cdot\overbar{\mathbf{e}}_{n+1}$.
\end{proof}

\begin{lem}\label{a5}
Let $\psi\in\mathfrak{Y}$. There is a constant $C=C(\Gamma_-,\Omega',M_0,T)$, so that for $0\leq t\leq T$, $\|\Phi_t^\#\psi\|_\mathfrak{Y}\leq C\|\psi\|_\mathfrak{Y}$.
\end{lem}

\begin{proof}
Let $\psi=f+g$ with $\|f\|_{\mathfrak{X}}+\|y^{-n}g\|_\infty\leq 2\|\psi\|_\mathfrak{Y}$. By definition, $\Phi_t^\# \psi=\Phi_t^\#f+\Phi_t^\#g=f(\Phi_t(p),\nabla_\textbf{v}\Phi_t(p))+g(\Phi_t(p),\nabla_\textbf{v}\Phi_t(p))$. So,
\begin{align}\label{114}
    \|y^{-n}\Phi_t^\#g\|_\infty\leq \left\|\frac{\Phi_t(p)\cdot\overbar{\textbf{e}}_{n+1}}{\textbf{x}(p)\cdot\overbar{\textbf{e}}_{n+1}}\right\|_\infty^n\|(\Phi_t(p)\cdot\overbar{\textbf{e}}_{n+1})^{-n}\Phi_t^\#g\|_\infty\leq C\|y^{-n}g\|_\infty,
\end{align}
where in the last step, we used (\ref{113}). Similarly, using (\ref{38}), (\ref{41}) and (\ref{113}), we have
\begin{align*}
    \|y\nabla_{\mathbb{R}^{n+1}_+}(\Phi_t^\#f)\|_\infty&\leq C(\|y\nabla_{\mathbb{R}^{n+1}_+}f\|_\infty+\|\nabla_{\mathbb{S}^n}f\|_\infty);\\
    \|\nabla_{\mathbb{S}^n}(\Phi_t^\#f)\|_\infty&\leq C\|\nabla_{\mathbb{S}^n}f\|_\infty;\\
     \|\nabla_{\mathbb{S}^n}\nabla_{\mathbb{S}^n}(\Phi_t^\#f)\|_\infty&\leq C(\|\nabla_{\mathbb{S}^n}f\|_\infty+\|\nabla_{\mathbb{S}^n}\nabla_{\mathbb{S}^n}f\|_\infty);\\
     \|y\nabla_{\mathbb{R}^{n+1}_+}\nabla_{\mathbb{S}^n}(\Phi_t^\#f)\|_\infty&\leq C(\|y\nabla_{\mathbb{R}^{n+1}_+}\nabla_{\mathbb{S}^n}f\|_\infty+\|\nabla_{\mathbb{S}^n}\nabla_{\mathbb{S}^n}f\|_\infty).
\end{align*}
This implies $\|\Phi_t^\#f\|_\mathfrak{X}\leq C\|f\|_\mathfrak{X}$. Combining this with (\ref{114}), we see that
\begin{equation*}
    \|\Phi_t^\# \psi\|_\mathfrak{Y}\leq\|\Phi_t^\#f\|_{\mathfrak{X}}+\|y^{-n}\Phi_t^\#g\|_\infty\leq C(\|f\|_\mathfrak{X}+\|y^{-n}g\|_\infty)\leq 2C\|\psi\|_\mathfrak{Y}.
\end{equation*}
\end{proof}

\section{Proof of Proposition \ref{65}\label{111}}
A proof to Proposition \ref{65} is included here. We adapt the argument in \cite[Section 8]{BWMountainPass}.

\begin{proof}[Proof of Proposition \ref{65}]
It suffices to prove the existence of a $\delta_0$, such that $\max_{\tau\in[0,1]}E_{rel}[\Gamma_\tau,\Gamma_-]\geq\delta_0$, since if we swap $\Gamma_-$ with $\Gamma_+$, it can be proved in a similar way that $\max_{\tau\in[0,1]}E_{rel}[\Gamma_\tau,\Gamma_+]\geq\delta_0$, and this implies $\max_{\tau\in[0,1]}E_{rel}[\Gamma_\tau,\Gamma_-]\geq E_{rel}[\Gamma_+,\Gamma_-]+\delta_0$.

For any $U\in\mathcal{C}(\Gamma'_-,\Gamma'_+)$, let $V_U=U\Delta U_{\Gamma_-}$ be the symmetric difference. Since $\tau\mapsto\mathbf{1}_{U_\tau}$ is continuous in $L^1_{loc}$ and $\text{Vol}_\mathbb{H}\big((U_{\Gamma_+}\backslash U_{\Gamma_-})\cap\{y\leq\epsilon\}\big)$ tends to $0$ as $\epsilon\to0$, we see that $\text{Vol}_\mathbb{H}(V_{U_\tau})$ is a continuous function in $\tau\in[0,1]$. Let $M_0=\text{Vol}_\mathbb{H}(V_{U_1})$ and
\begin{equation*}
    \mathcal{C}_m=\{U\in\mathcal{C}(\Gamma'_-,\Gamma'_+):\,\text{Vol}_\mathbb{H}(V_U)=m\}.
\end{equation*}
By Proposition $\ref{97}$, there is a sweep-out from $\Gamma_-$ to $\Gamma_+$. So, $\mathcal{C}_m$ is non-empty for each $m\in[0,M_0]$. We will show that for sufficiently small $m_0$, there is a $\delta_0=\delta_0(\Gamma_-,\Gamma_+,\Gamma'_-,\Gamma'_+)$, so that $E_{rel}[\partial^*U,\Gamma_-]\geq\delta_0$ for all $U\in\mathcal{C}_{m_0}$. And this proves the proposition.

Fix $m\in[0,M_0]$. Let $\{U^i\}_{i=1}^\infty$ be a sequence of Caccioppoli sets in $\mathcal{C}_m$, which minimizes the functional $E_{rel}[\partial^*\cdot,\Gamma_-]$. By Proposition \ref{35}, there exists an $E_0=E_0(\Gamma_-,\Omega)$, so that $E_{rel}[\partial^*U,\Gamma_-]\geq -E_0$ for all $U\in\mathcal{C}(\Gamma'_-,\Gamma'_+)$. So, $E_{m,\inf}:=\lim_{i\to\infty} E_{rel}[\partial^*U^i,\Gamma_-]>-\infty$. If $E_{m,\inf}=\infty$, we take $m_0$ to be this specific $m$ and $\delta_0$ to be arbitrary positive real number. So, we assume that $E_{m,\inf}\in(-\infty,\infty)$. In this case, by the compactness of BV-functions, there is a subsequence of $\{U^i\}$ converging in $L^1_{loc}$ to a Caccioppoli set $U_0$. Then, $\text{Vol}_\mathbb{H}(V_{U_0})=m$ and $E_{rel}[\partial^*U_0,\Gamma_-]=E_{m,\inf}$. From Proposition \ref{20}, we know that there is a mean-convex foliation, $\{\Gamma_-^s\}_{s\in[-1,1]}$, of $\Gamma_-$ that covers $\overbar{\mathcal{N}_{\epsilon}}(\Gamma_-)$ for some $\epsilon>0$. We make a distinction between two cases.
\begin{itemize}
    \item[(1)] If $\overbar{\partial^*U_0}\subset\overbar{\mathcal{N}_{\epsilon}}(\Gamma_-)$, let $U_+=U_0\cup U_{\Gamma_-}$ and $U_-=U_0\cap U_{\Gamma_-}$. In $U_+\backslash U_{\Gamma_-}$, consider a vector field $\textbf{X}_+$ that is equal to $n_{\Gamma_-^s}$ on $\Gamma_-^s$. Since $U_+\backslash U_{\Gamma_-}\subset\overbar{\mathcal{N}_{\epsilon}}(\Gamma_-)$, $\textbf{X}_+$ is well defined in $U_+\backslash U_{\Gamma_-}$. Direct computation shows $\text{div}(\textbf{X}_+)=H_{\Gamma_-^s}\geq0$. Applying the divergence theorem to $\phi_{y_1,\delta}\textbf{X}_+$,
    \begin{align}\label{98}
        E_{rel}[\partial^*U_+,\Gamma_-;\phi_{y_1,\delta}]\geq&\int_{\partial^*U_+}\left\langle\phi_{y_1,\delta}\textbf{X}_+,n_{\partial^*U_+}\right\rangle d\mathcal{H}^n_\mathbb{H}-\int_{\Gamma_-}\left\langle\phi_{y_1,\delta}\textbf{X}_+,n_{\Gamma_-}\right\rangle d\mathcal{H}^n_\mathbb{H}\\
        =&\int_{U_+\backslash U_{\Gamma_-}}\text{div}(\phi_{y_1,\delta}\textbf{X}_+)d\text{Vol}_\mathbb{H}\geq\int_{U_+}\left\langle\nabla\phi_{y_1,\delta},\textbf{X}_+\right\rangle d\text{Vol}_\mathbb{H}\nonumber\\
        \geq& -\int_{U_+\backslash U_{\Gamma_-}\cap\{y_1-\delta\leq y\leq y_1\}}\frac{C}{\delta}yd\text{Vol}_\mathbb{H}\geq -Cy_1.\nonumber
    \end{align}
    Letting $y_1$, $\delta\to0$, we get $E_{rel}[\partial^*U_+,\Gamma_-]\geq0$. Similarly, we can prove $E_{rel}[\partial^*U_-,\Gamma_-]\geq0$. Notice that $E_{rel}[\partial^*U_+,\Gamma_-]+E_{rel}[\partial^*U_-,\Gamma_-]=E_{rel}[\partial^*U_0,\Gamma_-]$. Hence, $E_{rel}[\partial^*U_0,\Gamma_-]\geq0$. The equality holds only if $U_0\Delta U_{\Gamma_-}$ has measure zero. Since $\text{Vol}_{\mathbb{H}}(V_{U_0})=m>0$, the equality can not hold. Thus, $E_{rel}[\partial^*U_0,\Gamma_-]>0$.
    \item[(2)] If $\overbar{\partial^*U_0}\backslash\overbar{\mathcal{N}_{\epsilon}}(\Gamma_-)\neq\emptyset$, then $\big(\overbar{\partial^*U_0}\backslash\overbar{\mathcal{N}_{\epsilon}}(\Gamma_-)\big)\subset U^c_{\Gamma_-}$ since $U_{\Gamma_-}\backslash U_{\Gamma'_-}\subset\overbar{\mathcal{N}_{\epsilon}}(\Gamma_-)$. Since $\Omega'$ is a bounded in $\mathbb{R}^{n+1}_+$, there is a $y_0>0$, so that $V_{U_0}\subset\{y\leq y_0\}$. Let $\overbar{d}_{\Gamma_-}$ be the distance to $\Gamma_-$ in Euclidean topology. By the coarea formula,
    \begin{equation*}
        \frac{1}{y_0}\int_{\frac{\epsilon}{4}}^{\frac{\epsilon}{2}}\mathcal{H}^{n}_\mathbb{H}(V_{U_0}\cap\{\overbar{d}_{\Gamma_-}=r\})dr\leq\text{Vol}(V_{U_0})=m.
    \end{equation*}
    By the mean value theorem, we can pick an $\epsilon_1\in(\frac{\epsilon}{4},\frac{\epsilon}{2})$, so that $\mathcal{H}^{n}_\mathbb{H}(V_{U_0}\cap\{\overbar{d}_{\Gamma_-}=\epsilon_1\})\leq\frac{8y_0m}{\epsilon}$. Denote $\text{cl}(\Omega')\backslash\overbar{\mathcal{N}}_{\epsilon_1}(\Gamma_-)$ by $W_{\epsilon_1}$. Let $U'\in\mathcal{C}(\Gamma'_-,\Gamma'_+)$ satisfy that $\mathcal{H}^n_{\mathbb{H}}(\partial^*U'\cap W_{\epsilon_1})$ minimizes the hyperbolic area among all the sets in $\mathcal{C}(\Gamma'_-,\Gamma'_+)$ that coincide with $U_0$ in $\overbar{\mathcal{N}}_{\epsilon_1}(\Gamma_-)\cap\text{cl}(\Omega')$. We claim that $U'\subset\text{cl}(U_{\Gamma_+})$. Otherwise, an calibration argument similar to (\ref{98}) would imply $E_{rel}[\partial^*U',\Gamma_-]> E_{rel}[\partial^*(U'\cap U_{\Gamma_+}),\Gamma_-]$, which gives a contradiction. Then, the maximal principle \cite{SolomonWhite} implies that once $\partial^*U'$ touches $\Gamma_+$, it coincides with $\Gamma_+$ in a neighborhood. So, $\partial^*U'$ is stationary in $W_{\epsilon_1}$. If $\text{Spt}(\partial^*U')\subset\overbar{\mathcal{N}_{\epsilon}}(\Gamma_-)$, then applying case (1) to $\partial^*U'$, we get $E_{rel}[\partial^*U_0,\Gamma]\geq E_{rel}[\partial^*U',\Gamma_-]>0$. Otherwise, we pick a point $p_0\in\text{Spt}(\partial^*U')\backslash\overbar{\mathcal{N}_{\epsilon}}(\Gamma_-)$. Then, $B_{\frac{\epsilon}{2}}(p_0)\subset W_{\epsilon_1}$. The monotonicity formula \cite[Section 17]{SimonBook} implies $\mathcal{H}^n_{\mathbb{H}}(\partial^*U'\cap W_{\epsilon_1})\geq\gamma_0$, where $\gamma_0=\gamma_0(\epsilon,n)$. Hence, by letting $U''=U'\cap\overbar{\mathcal{N}}_{\epsilon_1}(\Gamma_-)$, we have
    \begin{align*}
        E_{rel}[\partial^*U'',\Gamma_-]&\leq E_{rel}[\partial^*U',\Gamma_-]+\mathcal{H}^n_\mathbb{H}(V_{U_0}\cap\{\overbar{d}_{\Gamma_-}=r\})-\mathcal{H}^n_{\mathbb{H}}(\partial^*U'\cap W_{\epsilon_1})\\
        &\leq E_{rel}[\partial^*U_0,\Gamma_-]+\frac{8y_0m}{\epsilon}-\gamma_0.
    \end{align*}
    $E_{rel}[\partial^*U'',\Gamma_-]\geq0$ follows from a calibration argument used in case (1). By choosing $m_0=m$ sufficiently small, we see that $E_{rel}[\partial^*U_0,\Gamma_-]>0$.
\end{itemize}
So, in either case, there is an $m_0\in[0,M_0]$, such that  $\delta_0=E_{rel}[\partial^*U_0,\Gamma_-]=E_{m_0,\inf}>0$.
\end{proof}

\bibliographystyle{alpha}
\bibliography{minmax.bib}

\begin{thebibliography}{CIIW13}

\bibitem[AM10]{AlexakisMazzeo}
Spyridon Alexakis and Rafe Mazzeo.
\newblock Renormalized area and properly embedded minimal surfaces in
  hyperbolic 3-manifolds.
\newblock {\em Communications in Mathematical Physics}, 297(3):621--651, 2010.

\bibitem[And82]{Anderson1982}
Michael~T. Anderson.
\newblock Complete minimal varieties in hyperbolic space.
\newblock {\em Inventiones mathematicae}, 69(3):477--494, Oct 1982.

\bibitem[And83]{Anderson1983}
Michael~T. Anderson.
\newblock Complete minimal hypersurfaces in hyperbolic n-manifolds.
\newblock {\em Commentarii Mathematici Helvetici}, 58(1):264--290, Dec 1983.

\bibitem[BCW]{BCW}
Jacob Bernstein, Letian Chen, and Lu~Wang.
\newblock Existence of monotone morse flow lines of the expander functional.
\newblock In preparation.

\bibitem[Ber21]{BernsteinEntropy}
Jacob Bernstein.
\newblock Colding minicozzi entropy in hyperbolic space.
\newblock {\em Nonlinear Analysis}, 210:112401, 2021.

\bibitem[Ber22]{BernsteinIsoper}
Jacob Bernstein.
\newblock A sharp isoperimetric property of the renormalized area of a minimal
  surface in hyperbolic space.
\newblock {\em Proc. Amer. Math. Soc.}, 2022.

\bibitem[Bry88]{BryantConformal}
Robert~L. Bryant.
\newblock {\em Surfaces in conformal geometry}, volume~48 of {\em Proc. Sympos.
  Pure Math.}
\newblock American Mathematical Society, Providence, RI, 1988.

\bibitem[BS88]{BennettSharpleyBook}
Colin Bennett and Robert Sharpley.
\newblock {\em Interpolation of Operators}, volume 129 of {\em Pure and Applied
  Mathematics}.
\newblock Academic Press, Boston, MA, 1988.

\bibitem[BW]{BWMountainPass}
Jacob Bernstein and Lu~Wang.
\newblock A mountain-pass theorem for asymptotically conical self-expanders.
\newblock {\em Peking Math. J.}
\newblock To Appear.

\bibitem[BW17]{BWTopology}
Jacob Bernstein and Lu~Wang.
\newblock A topological property of asymptotically conical self-shrinkers of
  small entropy.
\newblock {\em Duke Math. J.}, 166(3):403--435, 2017.

\bibitem[BW18]{BWIntegerDegree}
Jacob Bernstein and Lu~Wang.
\newblock An integer degree for asymptotically conical self-expanders.
\newblock \url{https://arxiv.org/abs/1807.06494}, 2018.
\newblock Preprint.

\bibitem[BW19a]{BWSmoothCompactness}
Jacob Bernstein and Lu~Wang.
\newblock Smooth compactness for spaces of asymptotically conical
  self-expanders of mean curvature flow.
\newblock {\em Int. Math. Res. Not. IMRN}, 2021(12):9016--9044, 2019.

\bibitem[BW19b]{BWTopologicalUniqueness}
Jacob Bernstein and Lu~Wang.
\newblock Topological uniqueness for self-expanders of small entropy.
\newblock \url{https://arxiv.org/abs/1902.02642}, 2019.
\newblock Preprint.

\bibitem[BW20]{BWClosedHLowEnt}
Jacob Bernstein and Lu~Wang.
\newblock Closed hypersurfaces of low entropy in $\mathbb{R}^4$ are
  isotopically trivial.
\newblock {\em arXiv: Differential Geometry}, 2020.

\bibitem[BW21]{BWSpace}
Jacob Bernstein and Lu~Wang.
\newblock The space of asymptotically conical self-expanders of mean curvature
  flow.
\newblock {\em Math. Ann.}, 380:175--230, 2021.

\bibitem[BW22]{BWRelativeEntropy}
Jacob Bernstein and Lu~Wang.
\newblock Relative expander entropy in the presence of a two-sided obstacle and
  applications.
\newblock {\em Advances in Mathematics}, 399:108284, 2022.

\bibitem[CIIW13]{CIMW}
Tobias~Holck Colding, Tom Ilmanen, William P.~Minicozzi II, and Brian White.
\newblock {The round sphere minimizes entropy among closed self-shrinkers}.
\newblock {\em Journal of Differential Geometry}, 95(1):53 -- 69, 2013.

\bibitem[CL03]{ColdingdeLellisMinmax}
Tobias~Holck Colding and Camillo~De Lellis.
\newblock The min--max construction of minimal surfaces.
\newblock {\em Surveys in differential geometry}, 8:75--107, 2003.

\bibitem[CM12]{CMGeneric}
Tobias~Holck Colding and William~P. {Minicozzi II}.
\newblock Generic mean curvature flow i: generic singularities.
\newblock {\em Ann. of Math. (2)}, 175(2):755--833, 2012.

\bibitem[Cos06]{Coskunuzer2006Generic}
Baris Coskunuzer.
\newblock Generic uniqueness of least area planes in hyperbolic space.
\newblock {\em Geometry \& Topology}, 10:401--412, 2006.

\bibitem[Cos11]{Coskunuzer2011Number}
Baris Coskunuzer.
\newblock On the number of solutions to the asymptotic plateau problem.
\newblock {\em Journal of G{\"o}kova Geometry Topology}, 5:1--19, 2011.

\bibitem[Cos14]{CoskunuzerSurvey}
Baris Coskunuzer.
\newblock Asymptotic plateau problem: a survey.
\newblock 2014.

\bibitem[CY07]{ChenPseudolcality}
Binglong Chen and Le~Yin.
\newblock Uniqueness and pseudolocality theorems of the mean curvature flow.
\newblock {\em Communications in Analysis and Geometry}, 15:435--490, 2007.

\bibitem[DLR18]{DeLellisRamic}
Camillo De~Lellis and Jusuf Ramic.
\newblock Min-max theory for minimal hypersurfaces with boundary.
\newblock {\em Annales de l'Institut Fourier}, 68(5):1909--1986, 2018.

\bibitem[DM88]{DaviesHpbolicHeatKB}
E.~B. Davies and Nikolaos Mandouvalos.
\newblock Heat kernel bounds on hyperbolic space and kleinian groups.
\newblock {\em Proceedings of The London Mathematical Society}, pages 182--208,
  1988.

\bibitem[dOS98]{deOliveiraSoret98}
Geraldo de~Oliveira and Marc Soret.
\newblock Complete minimal surfaces in hyperbolic space.
\newblock {\em Mathematische Annalen}, 311(3):397--419, Jul 1998.

\bibitem[EH89]{EHEntireGraph}
Klaus Ecker and Gerhard Huisken.
\newblock Mean curvature evolution of entire graphs.
\newblock {\em Ann. of Math.}, 130(3):453--471, 1989.

\bibitem[Gab97]{Gabai97}
David Gabai.
\newblock On the geometric and topological rigidity of hyperbolic 3-manifolds.
\newblock {\em Journal of the American Mathematical Society}, 10(1):37--74,
  1997.

\bibitem[Giu84]{GiustiBook}
Enrico Giusti.
\newblock {\em Minimal Surfaces and Functions of Bounded Variation}.
\newblock Monographs in Mathematics. Birkh{\"a}user Boston, 1984.

\bibitem[GT83]{GilbargTrudinger}
D.~Gilbarg and N.S. Trudinger.
\newblock {\em Elliptic Partial Differential Equations of Second Order}.
\newblock Die Grundlehren der mathematischen Wissenschaften : a series of
  comprehensive studies in mathematics. Springer-Verlag, 1983.

\bibitem[HL87]{HardtLinHyperbolicRegularity}
Robert Hardt and Fang-Hua Lin.
\newblock Regularity at infinity for area-minimizing hypersurfaces in
  hyperbolic space.
\newblock {\em Inventiones mathematicae}, 88(1):217--224, 1987.

\bibitem[HL11]{HanLinEllipticPDE}
Qing Han and Fang-Hua Lin.
\newblock {\em Elliptic partial differential equations}.
\newblock Courant lecture notes. Courant Institute of Mathematical Sciences,
  Robotics Lab, New York University, 2nd ed. edition, 2011.

\bibitem[Hui86]{HuiskenContractingCvxHypers}
Gerhard Huisken.
\newblock Contracting convex hypersurfaces in riemannian manifolds by their
  mean curvature.
\newblock {\em Inventiones mathematicae}, 84(3):463--480, Oct 1986.

\bibitem[HW15]{HuangWang2015Counting}
Zheng Huang and Biao Wang.
\newblock Counting minimal surfaces in quasi-fuchsian three-manifolds.
\newblock {\em Transactions of the American Mathematical Society},
  367(9):6063--6083, 2015.

\bibitem[KZ18]{KetoverZhou}
Daniel Ketover and Xin Zhou.
\newblock {Entropy of closed surfaces and min-max theory}.
\newblock {\em Journal of Differential Geometry}, 110(1):31 -- 71, 2018.

\bibitem[Lin89a]{LinAMCinHighCodim}
Fang-Hua Lin.
\newblock Asymptotic-behavior of area-minimizing currents in hyperbolic space.
\newblock {\em Communications on Pure and Applied Mathematics}, 42(3):229--242,
  April 1989.

\bibitem[Lin89b]{LinHyperbolicMinimalGraph}
Fang-Hua Lin.
\newblock On the dirichlet problem for minimal graphs in hyperbolic space.
\newblock {\em Inventiones mathematicae}, 96(3):593--612, 1989.

\bibitem[LT13]{deLellisTasnady}
Camillo~De Lellis and Dominik Tasnady.
\newblock {The existence of embedded minimal hypersurfaces}.
\newblock {\em Journal of Differential Geometry}, 95(3):355 -- 388, 2013.

\bibitem[LY82]{LiYauConformalInv}
Peter Li and Shing-Tung Yau.
\newblock A new conformal invariant and its applications to the willmore
  conjecture and the first eigenvalue of compact surfaces.
\newblock {\em Inventiones mathematicae}, 69(2):269--291, Jun 1982.

\bibitem[Maz91]{MazzeoEdgeOperator}
Rafe Mazzeo.
\newblock Elliptic theory of differential edge operators i.
\newblock {\em Communications in Partial Differential Equations},
  16(10):1615--1664, 1991.

\bibitem[Mil15]{Milnor2015Cobordism}
J.~Milnor.
\newblock {\em Lectures on the H-Cobordism Theorem}.
\newblock Princeton Legacy Library. Princeton University Press, 2015.

\bibitem[MN14]{MarquesNevesWillmore}
Fernando~C. Marques and André Neves.
\newblock Min-max theory and the willmore conjecture.
\newblock {\em Annals of Mathematics}, 179(2):683--782, 2014.

\bibitem[MW13]{MartinWhite2013}
Francisco Martin and Brian White.
\newblock Properly embedded, area-minimizing surfaces in hyperbolic 3-space.
\newblock {\em arXiv: Differential Geometry}, 2013.

\bibitem[Ngu21]{NguyenWeightedMon}
Manh~Tien Nguyen.
\newblock {Weighted monotonicity theorems and applications to minimal surfaces
  in hyperbolic space}.
\newblock \url{https://arxiv.org/abs/2105.12625}, 2021.
\newblock Preprint.

\bibitem[Pit14]{PittsExiReg}
J.T. Pitts.
\newblock {\em Existence and Regularity of Minimal Surfaces on Riemannian
  Manifolds. (MN-27)}.
\newblock Mathematical Notes. Princeton University Press, 2014.

\bibitem[RW99]{GrahamWitten}
C.~{Robin Graham} and Edward Witten.
\newblock Conformal anomaly of submanifold observables in ads/cft
  correspondence.
\newblock {\em Nuclear Physics B}, 546(1):52--64, 1999.

\bibitem[Sim84]{SimonBook}
Leon Simon.
\newblock {\em Lectures on Geometric Measure Theory}.
\newblock Proceedings of the Centre of Mathematical Analysis, Australian
  National University. Australian National University, Centre of Mathematical
  Analysis, Canberra, 1984.

\bibitem[Sma65]{SmaleSard}
S.~Smale.
\newblock An infinite dimensional version of sard's theorem.
\newblock {\em American Journal of Mathematics}, 87(4):861--866, 1965.

\bibitem[Smi83]{SmithExistence2Sphere}
Francis~R. Smith.
\newblock On the existence of embedded minimal 2-spheres in the 3-sphere,
  endowed with an arbitrary metric.
\newblock {\em Bulletin of the Australian Mathematical Society},
  28(1):159–160, 1983.

\bibitem[SS81]{SchoenSimonRegularity}
Richard~M. Schoen and Leon Simon.
\newblock Regularity of stable minimal hypersurfaces.
\newblock {\em Communications on Pure and Applied Mathematics}, 34:741--797,
  1981.

\bibitem[SW89]{SolomonWhite}
Bruce Solomon and Brian White.
\newblock A strong maximum principle for varifolds that are stationary with
  respect to even parametric elliptic functionals.
\newblock {\em Indiana Univ. Math. J.}, 38(3):683--691, 1989.

\bibitem[Ton96]{TonegawaCMCinHpSpace}
Yoshihiro Tonegawa.
\newblock Existence and regularity of constant mean curvature hypersurfaces in
  hyperbolic space.
\newblock {\em Mathematische Zeitschrift}, 221(1):591--615, 1996.

\bibitem[Ton19]{Tonegawa}
Yoshihiro Tonegawa.
\newblock {\em Brakke's Mean Curvature Flow: An Introduction}.
\newblock Springer Briefs in Mathematics. Springer, Singapore, 2019.

\bibitem[Tre06]{TrevesTVSBook}
Francois Treves.
\newblock {\em Topological Vector Spaces, Distributions and Kernels}.
\newblock Dover books on mathematics. Dover Publications, 2006.

\bibitem[Whi91]{WhitePrescribedGenus}
Brian White.
\newblock Existence of smooth embedded surfaces of prescribed genus that
  minimize parametric even elliptic functionals on 3-manifolds.
\newblock {\em Journal of Differential Geometry}, 33(2):413--443, 1991.

\bibitem[Whi97]{WhiteStratification}
Brian White.
\newblock Stratification of minimal surfaces, mean curvature flows, and
  harmonic maps.
\newblock {\em Journal für die reine und angewandte Mathematik}, 488:1--36,
  1997.

\bibitem[Whi05]{WhiteLocalReg}
Brian White.
\newblock A local regularity theorem for mean curvature flow.
\newblock {\em Annals of mathematics, ISSN 0003-486X, Vol. 161, Nº 3, 2005,
  pags. 1487-1519}, 161, 05 2005.

\bibitem[Yao]{YaoRel}
Junfu Yao.
\newblock Relative entropy of hypersurfaces in hyperbolic space.
\newblock In preparation.

\end{thebibliography}

%\begin{thebibliography}{9}
%\bibitem{1}
        %S.B. Angenent, T. Ilmanen, and D.L. Chopp, \textit{A computed example of non-uniqueness of mean curvature flow in $\mathbb{R}^3$}, Commun. in Partial Differential Equations 20 (1995), no. 11-12, 1937-1958.

%\end{thebibliography}
\end{document}